\documentclass[preprint,aop]{imsart}

\RequirePackage{amsthm,amsmath,amsfonts,amssymb}
\RequirePackage[numbers]{natbib}
\RequirePackage[colorlinks,citecolor=blue,urlcolor=blue]{hyperref}
\RequirePackage{graphicx}
\RequirePackage{subfigure}
\RequirePackage{mathrsfs}
\RequirePackage{float}
\RequirePackage{bm}
\RequirePackage{dsfont}

\startlocaldefs
\numberwithin{equation}{section}
\theoremstyle{definition}
\newtheorem{defi}{Definition}[section]
\newtheorem{rmk}[defi]{Remark}

\theoremstyle{plain}
\newtheorem{lemma}[defi]{Lemma}
\newtheorem{coro}[defi]{Corollary}
\newtheorem{thm}[defi]{Theorem}
\newtheorem{prop}[defi]{Proposition}

\endlocaldefs

\begin{document}

\begin{frontmatter}
\title{Non-existence of several random fractals in Brownian motion and Brownian loop soup}
\runtitle{Non-existence of random fractals}

\begin{aug}
\author[A]{\fnms{Yifan}~\snm{Gao}\ead[label=e1]{gaoyifan75@westlake.edu.cn}},
\author[B]{\fnms{Xinyi}~\snm{Li}\ead[label=e2]{xinyili@bicmr.pku.edu.cn}}
\author[C]{\fnms{Runsheng}~\snm{Liu}\ead[label=e3]{liurunsheng@pku.edu.cn}}
\and
\author[D]{\fnms{Wei}~\snm{Qian}\ead[label=e4]{weiqian0@hku.hk}}
\address[A]{Institute for Theoretical Sciences, Westlake University\printead[presep={,\ }]{e1}}

\address[B]{Beijing International Center for Mathematical Research, Peking University\printead[presep={,\ }]{e2}}

\address[C]{School of Mathematical Sciences, Peking University\printead[presep={,\ }]{e3}}

\address[D]{Department of Mathematics, University of Hong Kong\printead[presep={,\ }]{e4}}
\end{aug}

\begin{abstract}
We develop a unified approach to establish non-existence of three types of random fractals: (1) pioneer triple points of planar Brownian motion, answering an open question in \cite{BW96}, (2) pioneer double cut points of planar and three-dimensional Brownian motions, and (3) double points on boundaries of clusters of the planar Brownian loop soup at the critical intensity, answering an open question in \cite{Qian21}. These fractals have a common feature that they are associated with an intersection or disconnection exponent which yields a Hausdorff dimension ``exactly zero''.
\end{abstract}

\begin{keyword}[class=MSC]
\kwd[Primary ]{60J65}
\kwd[; secondary ]{28A80}
\end{keyword}

\begin{keyword}
\kwd{Brownian motions}
\kwd{Brownian loop soup}
\kwd{random fractals}
\kwd{intersection exponents}
\kwd{disconnection exponents}
\end{keyword}

\end{frontmatter}

\section{Introduction}\label{sec:intro}
In the study of planar Brownian motion, Lawler, Schramm and Werner have established in the seminal works \cite{LSW01,LSW01a,LSW01b,LSW02} a family of Brownian intersection and disconnection exponents, using Schramm-Loewner evolutions (SLE), confirming earlier conjectures of Duplantier and Kwon \cite{DK88}. These values in turn determine Hausdorff dimensions of a class of random fractals inside planar Brownian motion. For example, the dimension of the frontier (i.e., outer boundary) of Brownian motion is $4/3=2-\xi(2)$ \cite{Law96a}, and the dimension of cut points in Brownian motion is $3/4=2-\xi(1,1)$ \cite{Law96}, where $\xi(\cdot)$ (resp.\ $\xi(\cdot,\cdot)$) is the Brownian disconnection (resp.\ intersection) exponent. These exponents also allow one to rule out the existence of certain subsets, such as triple points on the frontier \cite{BW96}, because $\xi(6)>2$. The derivation of the aforementioned Hausdorff dimensions mainly relies on the following intuition: standing on $x$, a frontier point (resp.\ cut point) of Brownian motion, the different portions of the Brownian motion to and away from $x$ (also called ``arms'' below) do not disconnect $x$ from infinity (resp.\ do not intersect).
\par
In two special cases, the Brownian disconnection exponent $\xi(5)$ and intersection exponent $\xi(2,1)$ are exactly equal to $2$. The exponents $\xi(5)$ and $\xi(1,2)$ are respectively associated with pioneer triple points and pioneer double cut points.
Therefore, these two fractals have dimension ``exactly zero'', but this is not sufficient to determine whether they exist.
\par
Let us list a few more examples of fractals in Brownian motion in dimensions two or three which have dimension ``exactly zero''. Some of these fractals do not exist: In \cite{DEKT57}, it is shown that three-dimensional Brownian motion does not admit triple points. Later on, it is shown in \cite{DEK61} that one-dimensional Brownian motion does not admit points of increase; see also \cite{Kni81,Ber83,Ade85,Bur90,Ber91,Per96,BB99} for various simplified proofs. The work \cite{LM92} shows that no one-sided $\pi/2$-cone points emerge in planar Brownian motion, in any direction. The monograph \cite{BB99} confirms non-existence of cut lines for planar Brownian motion in any direction. By contrast, there are also examples of zero-dimensional fractal subsets of Brownian motion that do exist. They include (local) maxima of one-dimensional Brownian motion, two-sided cone points of planar Brownian motion with cone angle $\pi$, tips of Brownian beads \cite{Vir03}, and cut (hyper)-planes of Brownian motion in $\mathbb{R}^{d}$, $d\geq3$ \cite{BB99}.
\par
In this paper, we focus on three types of random fractals of ``zero dimension'': pioneer triple points (PTP) of planar Brownian motion, pioneer double cut points (PDCP) of planar and three-dimensional Brownian motions, and double points on boundaries of clusters of the planar Brownian loop soup at the critical intensity, and develop a unified approach to establish their non-existence.
\subsection{PTP and PDCP of Brownian Motion}
We begin with pioneer triple points. Given a planar Brownian motion $(W_t)_{0\leq t<\infty}$, we call $z\in\mathbb{R}^{2}$ a {\it pioneer triple point} (PTP) for $W$ if there exist
$t_{1}<t_{2}<t_{3}$, such that
\[
\mbox{$W_{t_{1}}=W_{t_{2}}=W_{t_{3}}=z$ and $z$ is on the frontier of $W[0,t_{3}]$.}
\]
As mentioned above, Lawler, Schramm and Werner \cite{LSW01} showed that the Hausdorff dimension of the Brownian frontier is $2-\xi(2)=4/3$. Kiefer and M\"{o}rters \cite{KM10} showed that the Hausdorff dimension of double points on the Brownian frontier is $(1+\sqrt{97})/24=2-\xi(4)$. Burdzy and Werner \cite{BW96} showed that there is no triple point on the Brownian frontier by proving that $2-\xi(6)<0$. By a similar argument, one can also deduce that the Hausdorff dimension of pioneer triple points is $2-\xi(5)=0$. In the same work, Burdzy and Werner then conjectured that such points do not exist. M\"{o}rters and Peres also listed this question among the  ``selected open questions'' in their book {\it Brownian motion}; see \cite[p.\ 385]{MP10}. Our first result confirms this almost 30-year old conjecture.
\begin{thm}\label{ptp}
Almost surely, there do not exist pioneer triple points on $(W_t)_{0\leq t < \infty}$.
\end{thm}
We now turn to the second example, which concerns the random fractals related to the intersection exponents $\xi(1,2)$ and $\xi_{[3]}(1,2)$, respectively for Brownian motions in  2D and 3D. Let $(W_t)_{0\leq t < \infty}$ be a $d$-dimensional Brownian motion, $d=2,3$. We say a point $z\in\mathbb{R}^{d}$ is a {\it pioneer double cut point} (PDCP) with respect to $W$, if there exist $t_{1}<t_{2}$, such that
\[
\mbox{$W_{t_{1}}=W_{t_{2}}=z$, and $W[0,t_{1})\cap W(t_{1},t_{2})=\varnothing$.} 
\]
Much earlier than the derivation of all 2D Brownian exponents, Lawler \cite{Law89} has introduced a simple argument to deduce one special case of intersection exponents for both $2$- and $3$-dimensional Brownian motions. More precisely, he deduced  that $\xi(1,2)=2$ and $\xi_{[3]}(1,2)=1$ (see also \cite[Chapter 3.5]{Law13} or \cite[Chapter 10.2]{LL10} for simplified arguments, and Burdzy and Lawler \cite{BL90} for showing that the corresponding exponents for simple random walks are the same as for Brownian motions).
By similar arguments as in \cite{KM10,Law96}, these explicit values imply that the Hausdorff dimension of pioneer double cut points is $2-\xi(1,2)=0$ in 2D, and $3-2-\xi_{[3]}(1,2)=0$ in 3D. It is then very natural to ponder upon the (non-)existence of such points. Our next theorem shows that almost surely PDCP do not exist.
\begin{thm}\label{pdcp}
Almost surely, there do not exist pioneer double cut points on $(W_t)_{0\leq t < \infty}$.
\end{thm}
\subsection{Double Points on Boundaries of Critical Brownian Loop-Soup Clusters}
Our third result concerns double points on boundaries of critical Brownian loop-soup clusters. The Brownian loop soup was introduced in \cite{LW04}, as a Poisson point process with intensity $\mathbf{c}$ times a certain Brownian loop measure (see Section~\ref{subsec:bls} for more details).
Sheffield and Werner \cite{SW12} showed that the Brownian loop soup has a phase transition at the critical intensity $\mathbf{c}=1$: when $\mathbf{c}>1$, $\Gamma_{0}$ a.s.\ has only one single cluster; when $0<\mathbf{c}\leq1$, $\Gamma_{0}$ a.s.\ has infinitely many clusters. Furthermore, for $\mathbf{c}\in (0,1]$,  outer boundaries of outermost clusters in the Brownian loop soup are distributed as a \textit{conformal loop ensemble (CLE)}. 
\par
In \cite{Qian21}, the fourth author introduced generalized disconnection exponents for the Brownian loop soup with intensity $\mathbf{c}\in(0,1]$, gave an explicit formula (see \eqref{gde value}) for their values and conjectured that the Hausdorff dimension of simple (resp.\ double) points on boundaries of loop-soup clusters is $2-\xi_{\mathbf{c}}(2)$ (resp.\ $2-\xi_{\mathbf{c}}(4)$).
This conjecture is proved by three of the authors of this work in \cite{GLQ22}. In the critical $\mathbf{c}=1$ case, the Hausdorff dimension of double points on boundaries of loop-soup clusters is $2-\xi_{1}(4)=0$. Whether such points exist was raised as an open question in \cite{Qian21}, and our third result answers this question negatively.
\par
In order to state our result precisely, we first recall some notation in \cite{GLQ22}. Let $\Gamma_{0}$ be a Brownian loop soup with intensity $\mathbf{c}=1$ in the unit disc $\mathbb D$. For every cluster $K$ in $\Gamma_{0}$, let $\partial K$ be the collection of boundaries of $K$. Let $\mathcal{D}$ be the set of double points in $\Gamma_{0}$ (the points that are visited at least twice in total by all 
the loops in $\Gamma_0$), and a point in $\mathcal{D}\cap\partial K$ will be called a \textit{boundary double point} (BDP) of the Brownian loop soup $\Gamma_0$.
\begin{thm}\label{bdp}
Let $\Gamma_{0}$ be a Brownian loop soup with intensity $\mathbf{c}=1$ in $\mathbb D$. Almost surely, for any cluster $K$ in $\Gamma_{0}$, we have $\partial K\cap\mathcal{D}=\varnothing$.
\end{thm}
We now briefly mention an implication of Theorem \ref{bdp} on the decomposition of critical Brownian loop-soup clusters. 
Werner and the fourth author have shown in \cite{QW19,MR3896865} that one can decompose a critical Brownian loop-soup cluster into two independent parts: a Brownian loop soup in the interior of the cluster, and a Poisson point process of Brownian excursions, where the latter one arises from decomposing loops that touch the outer boundary of the cluster. It is an interesting question (see \cite[Section 5.4]{QW19}) what extra randomness is involved when one hooks excursions back into Brownian loops. 
One possible source of randomness comes from double points of the loop soup on the outer boundary of a cluster, since one can rewire loops at these double points \cite{MR3618142}.
Our result (Theorem \ref{bdp}) shows that there is no extra randomness coming from this source. On the other hand, recently, Lehmkuehler, Werner and the fourth author have shown in \cite{LQW24} that there exists other essential randomness in this hooking procedure.
\subsection{Strategy of the Proof}
Proofs of Theorems \ref{ptp}, \ref{pdcp} and \ref{bdp} share the same framework, hence we focus on the PTP case (Theorem \ref{ptp}) to illustrate the main strategy.
\par
More concretely, it suffices to show that there is no PTP on $W[0,\tau_{\mathbb{D}}]$, where $W$ is a planar Brownian motion started at $0$ and $\tau_{\mathbb{D}}$ is its first exit time of the unit disc. We then further reduce the problem to the non-existence of $\delta$-PTP for any $\delta>0$ (i.e., PTP's that have all five arms traveling a macroscopic distance at least $\delta$) in the annulus $(1-\iota)\mathbb{D}\setminus \iota\mathbb{D}$ for some $0<\iota<1/2$ (we will also say ``in the bulk''). We use $\mathfrak{S}_{n}$ to denote the collection of ``good boxes'' which, heuristically speaking, are squares of side length $2^{-n}$ in the bulk which contains an ``approximate'' $\delta$-PTP. The exact definition of good boxes will be given at the beginning of Section \ref{sec:moment}. It suffices to show that (see \eqref{eq:E0u} in the proof of Theorem~\ref{ptp}) for some $c>0$
\begin{equation}\label{one good square}
\mathbb{P}(\#\mathfrak{S}_{n}\geq 1)=O(n^{-c}).
\end{equation}

Our proof of \eqref{one good square} is based on the intuition that it is very likely to find many good boxes near a typical good box. In order to carry out this argument, one should work with  the ``planted'' conditional law $\mathbb{P}(\cdot|\widetilde{S}=S)$ (``$S$ being a {\it typical} good box''), where $S$ is a box in the bulk and $\widetilde{S}$ is uniformly sampled in $\mathfrak{S}_{n}$. However,
in practice, in the case of PTP, we are only able to work with the conditional law $\mathbb{P}(\cdot|S\in\mathfrak{S}_{n})$ (``$S$ being a good box''). Our solution is to consider the conditional negative first moment. We have
\begin{equation}\mathbb{P}(\#\mathfrak{S}_{n}\geq1)=\sum_{S}\mathbb{P}\Big(\widetilde{S}=S,\, \#\mathfrak{S}_{n}\geq1\Big)=\sum_{S}\mathbb{E}\Big[(\#\mathfrak{S}_{n})^{-1}\Big|S\in\mathfrak{S}_{n}\Big]\mathbb{P}(S\in\mathfrak{S}_{n}),\end{equation}
where the summation is over all boxes $S$ of side length $2^{-n}$ in the bulk. 
Note that $\mathbb{P}(S\in\mathfrak{S}_{n})$ is of order $2^{-2n}$, since PTP should have dimension exactly $0$.
Therefore, in order to get \eqref{one good square}, we only need to show that
\begin{equation}\label{-1 moment bound}\mathbb{E}\Big[(\#\mathfrak{S}_{n})^{-1}\Big|S\in\mathfrak{S}_{n}\Big]=O(n^{-c}),\end{equation}
where the estimate is uniform in $S$. It then boils down to the following large deviation bound:
\begin{equation}\label{concentrate}
\mathbb{P}\Big(\#\mathfrak{S}_{n}\leq c'\sqrt{n}\Big|S\in\mathfrak{S}_{n}\Big)\leq n^{-c}\quad\mbox{ for some $c,c'>0$.}\end{equation}
\par
In order to obtain \eqref{concentrate}, we count the number of good boxes in the following annuli around $S$.
Write $L:=[\sqrt{n}]$. For $1\leq i\leq L/2$, let $\mathcal{A}(S,2^{-n+iL},2^{-n+(i+1)L})$ be the annulus concentric with $S$ of radii $2^{-n+iL}$ and $2^{-n+(i+1)L}$, and let $N_{i}$ be the number of good boxes in this annulus. 
By (conditional) first moment estimate, we can obtain $\mathbb{E}[N_{i}|S\in\mathfrak{S}_{n}]\geq cL$, and by (conditional) second moment estimate, we can obtain $\mathbb{E}[N_{i}^2|S\in\mathfrak{S}_{n}]\leq c'iL^{2}$. Therefore, by Paley-Zygmund inequality, we have 
\begin{align}\label{eq:mm1}
\mathbb{P}(N_{i}\geq c'L|S\in\mathfrak{S}_{n})\geq c/i.
\end{align}
\par
We now claim that conditioned on $S\in\mathfrak{S}_{n}$, the quantities $N_{i}$ are ``almost independent'', which allows us to bound $\mathbb{P}\Big(\#\mathfrak{S}_{n}\leq c'\sqrt{n}\Big|S\in\mathfrak{S}_{n}\Big)$ by the product of probabilities $\mathbb{P}(N_{i}\leq c'L|S\in\mathfrak{S}_{n})$. To this end, we construct a probability space equipped with a filtration $(\mathscr{F}_i)_{1\le i\le L/2}$ that $(N_{i})_{1\leq i\leq L/2}$ is adapted to, and then argue that the moment bound \eqref{eq:mm1} still holds for each $i$ when we further condition on $\mathscr{F}_{i-1}$, thanks to a certain Markov property of the configuration. This is the main technical difficulty in our proof.  It is natural to use the five arms associated with PTP in $\mathfrak{S}_n$, stopped upon exiting the balls of radii $2^{-n+(i+1)L}$ concentric with $S$ for $1\leq i\leq L/2$, to generate a filtration. However, such a filtration cannot guarantee the Markov property of the configuration across these successive annuli, due to the possibility of some ``bad events''. For instance, the five arms can possibly go back and forth across multiple annuli. To fix this issue, we introduce a delicate path decomposition using stopping times and carefully choose what to include in the filtration. We also introduce a slightly tilted probability measure that assigns zero measure to some ``bad events'' which occur with tiny (conditional) probabilities (in the original measure) to ensure adaptedness. We refer the reader to Section \ref{sec:ptp} for a more precise explanation and the explicit construction of this filtration.
\par
Our method can be applied to discrete counterparts as well. As an example, we consider pioneer triple points of a random walk. Let $X$ be a random walk on $\mathbb{Z}^{2}$. A point $z\in X[0,N]$ is called a $\delta$-macroscopic PTP, if there exist $\delta N\leq t_{1}<t_{2}<t\leq N$ with $t_{2}-t_{1},t-t_{2}\geq\delta N$, such that $X_{t_{1}}=X_{t_{2}}=X_{t}=z$ and $z$ is on the outer boundary of $X[0,t]$. In a similar fashion, we can show that for some $c>0$, 
\begin{equation}
\mathbb{P}(\mbox{$X[0,N]$ contains a $\delta$-macroscopic PTP})=O((\log N)^{-c}).    
\end{equation}
\subsection{Structure of the Paper}
We now briefly discuss the structure of the paper. In Section \ref{sec:prelim}, we introduce notation and provide some preliminary results. In Section \ref{sec:moment}, we define good boxes with respect to each type of special points and derive various moment bounds. In Sections \ref{sec:ptp}-\ref{sec:bdp}, we will prove Theorems \ref{ptp}, \ref{pdcp} and \ref{bdp} respectively. As their proofs are rather similar, we will only give a detailed proof of Theorem \ref{ptp} in Section \ref{sec:ptp}, and will be quite brief for the proofs of  Theorems \ref{pdcp} and \ref{bdp}  (except at places where non-trivial modifications are needed).

\section{Notation and Preliminaries}\label{sec:prelim}
In this section, we introduce some general notation, and then briefly discuss Brownian path measures, Brownian loop soups, intersection and (generalized) disconnection exponents and separation lemmas for various non-intersection and non-disconnection events.
\subsection{Notation}\label{subsec:notation}
\textbf{Constants:} Lower-case constants such as $c,c',\widetilde{c},c_{1},c_{2},\ldots$ are defined locally. Upper-case symbols such as $C,C_{0},C_{1},\ldots$ stand for constants that will stay fixed within each \textit{section}. Dependence on other parameters will be listed at their first appearance. Throughout, all constants are positive.
\par
\smallskip
\noindent \textbf{Comparison between quantities:} For two functions $f,g:\mathbb{R}_{+}\to\mathbb{R}_{+}$, we write $f(x)\asymp g(x)$ if there exist $c_{1},c_{2}$, such that $f(x)\leq c_{1}g(x)$ and $g(x)\leq c_{2}f(x)$ for all $x$.
\par
We say two probability measures $P,\;Q$ are \textit{uniformly equivalent}, if they are mutually absolutely continuous to each other, and there exists a universal constant $c\in(1,\infty)$, such that 
\[c^{-1}\leq\mathrm{d}P/\mathrm{d}Q\leq c.\]
\par
In the following, we say two events $A,B$ are \textit{almost independent} under a probability measure $\mathbb{P}$, if there exists a universal constant $c$, such that 
\[c^{-1}\,\mathbb{P}(A)\mathbb{P}(B)\leq\mathbb{P}(A\cap B)\leq c\,\mathbb{P}(A)\mathbb{P}(B).\]
We say an event $A$ is \textit{almost independent} to a $\sigma$-algebra $\mathscr{F}$ with respect to a probability measure $\mathbb{P}$, if there exists a universal constant $c$, such that 
\[c^{-1}\mathbb{P}(A)\leq\mathbb{P}(A|\mathscr{F})\leq c\mathbb{P}(A).\]
\textbf{Disc and annulus:} Let $\mathcal{B}(z,r):=\{x\in\mathbb{R}^{d}:|x-z|\leq r\}$ be the disc in $\mathbb{R}^{d}$ with centre $z$ and radius $r$, and we use $\mathbb{D}$ to denote the $d$-dimensional unit disc (ball). Let $\mathcal{A}(z,r_{1},r_{2})$ denote the annulus with centre $z$, inner radius $r_{1}$ and outer radius $r_{2}$. For an annulus $\mathcal{A}$, let $\partial_{i}\mathcal{A}$ (resp.\ $\partial_{o}\mathcal{A}$) be the inner (resp.\ outer) boundary of $\mathcal{A}$. Moreover, for a set $S$ with barycentre $z$, let $\mathcal{B}(S,r):=\mathcal{B}(z,r)$ and $\mathcal{A}(S,r_{1},r_{2}):=\mathcal{A}(z,r_{1},r_{2})$.
\par
Since we often deal with dyadic scales, we also write 
\[\mathcal{D}_{n}(z):=\mathcal{B}(z,2^{n}),\quad \mathcal{C}_{n}(z):=\partial\mathcal{D}_{n}(z).\]
We also write $\mathcal{D}_{n}(S):=\mathcal{D}_{n}(z)$ and $\mathcal{C}_{n}(S):=\mathcal{C}_{n}(z)$ for a set $S$ with barycentre $z$ and write $\mathcal{D}_{n}$ (resp.\ $\mathcal{C}_{n}$) for $\mathcal{D}_{n}(0)$ (resp.\ $\mathcal{C}_{n}(0)$).
\par
\smallskip
\noindent \textbf{Dyadic boxes:} When $d=2$, let
\[\mathcal{S}_{n}^{i,j}:=(i2^{-n},(i+1)2^{-n})\times(j2^{-n},(j+1)2^{-n}),\quad i,j\in\mathbb{Z}\]
and when $d=3$, let
\[\mathcal{S}_{n}^{i,j,k}:=(i2^{-n},(i+1)2^{-n})\times(j2^{-n},(j+1)2^{-n})\times(k2^{-n},(k+1)2^{-n}),\quad i,j,k\in\mathbb{Z}\]
denote $n$-boxes of side length $2^{-n}$ in $\mathbb{R}^{d}$.
Let $\mathcal{S}_{n}$ be the collection of $\mathcal{S}_{n}^{i,j}$ (resp.\ $\mathcal{S}_{n}^{i,j,k}$). An element of $\mathcal{S}_{n}$ will be also called an $n$-box.
\par
For a set $A\subset\mathbb{R}^{d}$ and $S\in\mathcal{S}_{n}$, we use $S\prec A$ to denote the event that $\mathcal{D}_{-n}(S)\subset A$.
\par
\smallskip
\noindent \textbf{Distance:} For two disjoint centrally symmetric sets $A,B\in\mathbb{R}^{d}$, write $\mathrm{dist}(A,B)$ for the distance between the centres of $A,B$.
\par
\smallskip
\noindent \textbf{Curves and (stopping) times:} For a time-parameterized curve $\gamma=\gamma(t)$, $(t\in[0,T])$, and sets $A_{1},\ldots,A_{n}\subset\mathbb{R}^{d}$, define
\[\tau_{\gamma}(A_{1}):=\inf\{t:\gamma(t)\in A_{1}\},\quad\tau_{\gamma}(A_{1},\ldots,A_{n}):=\inf\{t\geq\tau(A_{1},\ldots,A_{n-1}):\gamma(t)\in A_{n}\}.\]
We also define the last hitting time
\[\sigma_{\gamma}(A_{1}):=\sup\{t:\gamma(t)\in A_{1}\}.\]
We also write
\[\tau_{\gamma}(\tau'_{\gamma},A):=\inf\{t\geq\tau'_{\gamma}:\gamma(t)\in A\}\]
for some (random) time $\tau'_{\gamma}$. Note that for a deterministic $A$ and a random $\gamma$, $\tau_{\gamma}(\tau'_{\gamma},A)$ is a stopping time (with respect to the natural filtration of $\gamma$) if $\tau'_{\gamma}$ is a stopping time. We abbreviate $\tau_{\gamma}$ (resp.\ $\sigma_{\gamma}$) as $\tau$ (resp.\ $\sigma$) if no confusion arises.
\par
For a time-parameterized curve $\gamma=\gamma(t)$, let $\gamma^{\mathcal{R}}$ be its time-reversal.
\par
\smallskip
\noindent \textbf{Openings:} This definition is only relevant when $d=2$.
\par
For an annulus $\mathcal{A}=\mathcal{A}(x,s,r)$ and a set $\mathcal{U}$, let $\mathcal{O}(\mathcal{U},\mathcal{A})$ stand for the  \textit{opening} of $\mathcal{U}$ with respect to $\mathcal{A}$, which is the collection of connected components of $\mathcal{A}\backslash\mathcal{U}$ that intersect both $\partial\mathcal{B}(x,s)$ and $\partial\mathcal{B}(x,r)$.
\par
Moreover, for $v\in(s,r)$, let $\mathcal{O}(\mathcal{U},\mathcal{A},v)$ be the \textit{cut set} of $\mathcal{O}(\mathcal{U},\mathcal{A})$ on $\partial\mathcal{B}(x,v)$ defined as follows.
First, let $\mathcal{G}$ be the collection of continuous curves $\gamma=(\gamma(t))_{0\le t\le1}$ in $\mathcal{A}$, such that $\gamma(0)\in\partial\mathcal{B}(x,s),\gamma(1)\in\partial\mathcal{B}(x,r)$ and $\gamma\cap\mathcal{U}=\varnothing$. We then define 
\[\mathcal{O}_{0}:=\{z\in\partial\mathcal{B}(x,v):\exists\gamma\in\mathcal{G},\ \mbox{such that $\tau_{\gamma}(\partial\mathcal{B}(x,v))=z$}\}.\]
We now define
\[\mathcal{O}(\mathcal{U},\mathcal{A},v):=\{z\in\mathcal{O}_{0}:\exists\gamma\in\mathcal{G},\ \mbox{$\tau_{\gamma}(\partial\mathcal{B}(x,v))=z$ and $\gamma(\tau_{\gamma}(\partial\mathcal{B}(x,v)),1]\cap\mathcal{O}_{0}=\varnothing$}\}.\]
Note that $\mathcal{O}(\mathcal{U},\mathcal{A},v)$ is a minimal collection of points such that every continuous path from $\partial\mathcal{B}(x,s)$ to $\partial\mathcal{B}(x,r)$ avoiding $\mathcal{U}$ must pass through one such point when it first hits $\partial\mathcal{B}(x,v)$.
For example, in the definition above, the cut set on $\partial\mathcal{B}(x,v)$ does not include ``dead ends''; see Figure \ref{fig_open}.
\begin{figure}[b]
\centering
\begin{subfigure}
\centering
\includegraphics[width=4.5cm]{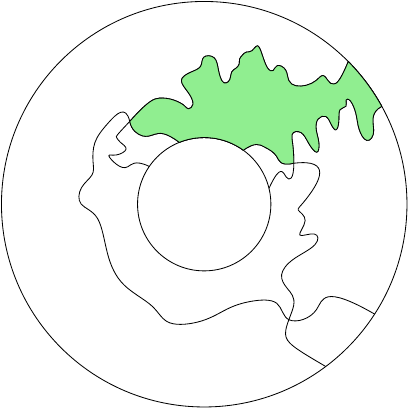}
\end{subfigure}
\begin{subfigure}
\centering
\includegraphics[width=4.5cm]{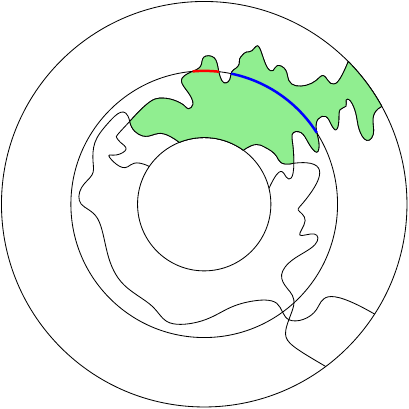}
\end{subfigure}
\caption{\small $\mathcal{U}$ is the union of the curves. The opening is in green, the cut set on the intermediate scale is in blue and the ``dead end'' is in red.}
\label{fig_open}
\end{figure}
\par
For later use, we also need to define the cut set with respect to a connected set in a general annulus. Let $D_{1}\subseteq D_{2}\subseteq D_{3}$ be three discs (possibly with different centres). Let $K_{0}$ be a connected set in $D_{3}\backslash D_{1}$. Let $\mathcal{G}'$ be the collection of continuous curves $\gamma=(\gamma(t))_{0\leq t\leq 1}$ in $D_{3}\backslash D_{1}$, such that $\gamma(0)\in\partial D_{1},\gamma(1)\in\partial D_{3}$ and $\gamma\subseteq K_{0}$. Let $\mathcal{O}'_{0}:=\{z\in\partial D_{2}:\exists\gamma\in\mathcal{G}',\ \tau_{\gamma}(\partial D_{2})=z$\}. We then define the cut set of $K_{0}$ on $\partial D_{2}$ (with respect to the general annulus $D_{3}\backslash D_{1}$) as the following collection of points: 
\[\{z\in\mathcal{O}'_{0}:\exists\gamma\in\mathcal{G}',\ \tau_{\gamma}(\partial D_{2})=z,\ \gamma(\tau_{\gamma}(\partial D_{2}),1]\cap\mathcal{O}'_{0}=\varnothing\}.\]
\subsection{Brownian Motion and Path Measures}\label{subsec:path_measure}
We write $\mathbb{P}^{x}$ for the probability measure of the Brownian motion started from $x$, and write $\mathbb{P}$ for $\mathbb{P}^{0}$ for brevity. We  write $\mathbb{P}^{x_{1},\ldots,x_{n}}$ for the product measure of $n$ independent Brownian motions started from $x_{1},\ldots,x_{n}$, respectively.
\par 
Let $D$ be a smooth domain. If $x,y$ are two distinct points in $\overline{D}$, let $\mu_{x,y}^{D}$ denote the measure on Brownian paths from $x$ to $y$ in $\overline{D}$, as defined in \cite{LW04}.
When $x, y$ are interior points of $D$, we have $|\mu_{x,y}^{D}|=G^{D}(x,y)$, where $G^{D}(\cdot,\cdot)$ is the Green's function in $D$. When one (resp.\ both) of the endpoints $x,y$ are on $\partial D$, $|\mu_{x,y}^{D}|$ is given by the (resp.\ boundary) Poisson kernel. For $x\in\partial D$, the \textit{Brownian bubble measure} $\mu_{D}^{\mathrm{bub}}(x)$ is the $y\to x$ limit of the interior-to-boundary measure $\mu^D_{y,x}$ for $y\in D$ with an appropriate normalizing factor; see \cite{LW04} for more details.
\par
Fix a smooth domain $D$ and an interior point $x\in D$. Let $W$ be a Brownian motion started from $x\in D$ and stopped upon reaching $\partial D$. The law of this process, denoted by $\mathbb{P}_{D}^{x}$, can be decomposed:
\[\mathbb{P}_{D}^{x}=\int_{\partial D}\mu_{x,y}^{D}\,\mathrm{d}y\]
where $\mu_{x,y}^{D}$ is the interior-to-boundary measure. Here and below, $\mathrm{d}y$ represents integration with respect to arc length if $d=2$ and area if $d=3$.
\par
We now recall the definition of Brownian excursions in an annulus. Let $W_{t}$ be a Brownian motion started from 0. For $r_2>r_1>0$, let $\sigma_{r_{1}}$ be the last time that $W_{t}$ hits $\partial\mathcal{B}(0,r_{1})$ before $\tau_{r_{2}}$ (the first hitting time of $\partial\mathcal{B}(0,r_{2})$). Then we say that $Y_{t}=W_{t+\sigma_{r_{1}}}(t\in[0,\tau_{r_{2}}-\sigma_{r_{1}}))$ has the law of a {\it Brownian excursion} in $\mathcal{A}(0,r_{1},r_{2})$. We can also characterize the Brownian excursion as an integral of the boundary-to-boundary measure. Let 
\[\mu_{r_{1},r_{2}}=\int_{\partial\mathcal{B}(0,r_{1})}\int_{\partial\mathcal{B}(0,r_{2})}\mu_{x,y}^{\mathcal{A}(0,r_{1},r_{2})}\,\mathrm{d}x\,\mathrm{d}y.\]
Then, its normalized version $\mu_{r_{1},r_{2}}^{\#}:=\mu_{r_{1},r_{2}}/|\mu_{r_{1},r_{2}}|$ gives the law of a Brownian excursion in the annulus $\mathcal{A}(0,r_{1},r_{2})$.
\medskip

We now turn to the first- and last-exit decomposition of Brownian path measures. The first-exit decomposition is a direct consequence of the strong Markov property, and we refer readers to \cite[Section 5.2]{Law05} for the last-exit decomposition.
Throughout, we use the symbol $\oplus$ to denote the natural concatenation of paths or measures (the latter can be induced from the former).
\begin{lemma}\label{last exit}
Let $D',D$ be two domains such that $D'\subset D$. For $x\in D'$ and $y\in\partial D$, we have
$$
\mu^{D}_{x,y}=\int_{\partial D'}\mu_{x,z}^{D'}\oplus\mu_{z,y}^{D}\,\mathrm{d}z,
$$
and
\[\mu^{D}_{x,y}=\int_{\partial D'}\mu_{x,z}^{D}\oplus\mu_{z,y}^{D\backslash D'}\,\mathrm{d}z.\]
\end{lemma}
By the Harnack principle, we have the following useful result for Brownian motions. Later, we will also use straightforward generalizations of this result (such as Brownian motions with some random starting points, conditional Brownian motions, and Brownian excursions).
\begin{lemma}\label{lem:harnack}
Let $W$ be a Brownian motion in $\mathbb{R}^{d}$, started from $x\in\partial\mathcal{D}_{0}$. 
For any $r>0$, the law of $W(\tau_{\partial\mathcal{D}_{r}})$ is uniformly equivalent to the uniform distribution on $\partial\mathcal{D}_{r}$, where the implicit constants only depend on $r$ (independent of $x$). 
\end{lemma}
\subsection{The Brownian Loop Soup}\label{subsec:bls}
In this subsection, we give a brief review on the Brownian loop soup. The \textit{Brownian loop measure} $\mu^{\mathrm{loop}}$ is defined in \cite[Section 4]{LW04} by
\begin{equation}\label{loop measure}
\mu^{\mathrm{loop}}=\int_{\mathbb{R}^{2}}\int_{0}^{\infty}\frac{1}{2\pi t^{2}}\mu^{\#}(z,z;t)\,\mathrm{d}t\,\mathrm{d}z,
\end{equation}
where $\mu^{\#}(z,z;t)$ is the probability measure on Brownian bridges from $z$ to itself with duration $t$. The \textit{Brownian loop soup} (in the whole plane) with intensity $\mathbf{c}>0$ is a Poisson point process with intensity $\frac{\mathbf{c}}{2}\mu^{\mathrm{loop}}$, which is denoted by $\Gamma$ below. 
For a domain $D$, let $\Gamma_{D}$ be the collection of loops in $\Gamma$ which are fully contained in $D$. By the restriction property, $\Gamma_{D}$ has the same distribution as a Brownian loop soup in $D$, that is, a PPP with intensity $\frac{\mathbf{c}}{2}\mu_{D}^{\mathrm{loop}}$, where $\mu_{D}^{\mathrm{loop}}$ is $\mu^{\mathrm{loop}}$ restricted to the loops that are fully contained in $D$.
In the following, we write $\Gamma_{0}:=\Gamma_{\mathbb{D}}$ for the Brownian loop soup in the unit disc.

\subsection{Intersection and Disconnection Exponents}
In this subsection, we recall the definition and basic properties of classical intersection and disconnection exponents and their generalized versions.
\subsubsection*{Brownian Exponents}
We begin with intersection exponent of Brownian motions; see e.g.\ \cite{Law98,LSW01b}.
\begin{defi}\label{ie real}
Assume that $d=2,3$. Let $X_{0},X_{1},\ldots,X_{k}$ be independent Brownian motions in $\mathbb{R}^{d}$ started uniformly from $\mathcal{C}_0$, and stopped upon reaching $\partial\mathcal{B}(0,e^{r})$.
Write
\[Z_{r,d}=Z_{r,d}(X_{1},\ldots,X_{k}):=\mathbb{P}(X_{0}\cap(\cup_{j=1}^{k}X_{j})=\varnothing|X_{1},\ldots,X_{k}).\]
For all $\lambda>0$, let $a_{r,d}=a_{r,d}(\lambda):=\mathbb{E}[Z_{r,d}^{\lambda}]$. Then, the following limit can be shown to exist.
\[\xi_{[d]}(k,\lambda):=-\log\lim_{r\to\infty}(a_{r,d})^{1/r},\]
which is called the intersection exponent. 
\end{defi}
In fact, the existence of $\xi_{[d]}(k,\lambda)$ follows from a simple sub-multiplicative argument (Fekete's lemma). Since we will mostly work in two dimensions, for brevity, we write $\xi(\cdot,\cdot):=\xi_{[2]}(\cdot,\cdot)$ below. 
Thanks to a series of pioneer works by Lawler, Schramm and Werner \cite{LSW01a,LSW01b,LSW02,LSW02a}, the explicit values of 2D intersection exponents are rigorously established now (conjectured by Duplantier and Kwon in \cite{DK88}):
\begin{equation}\label{ie value}
\xi(k,\lambda)=\frac{(\sqrt{24k+1}+\sqrt{24\lambda+1}-2)^{2}-4}{48}.
\end{equation}
However, most 3D exponents remain unknown, except for $\xi_{[3]}(1,2)=1$ as already explained above Theorem~\ref{pdcp}.

The following theorem gives up-to-constants bounds on $a_{r,d}(\lambda)$, which has been derived in \cite{Law98,LSW02b}. Throughout this section, results hold for both $d=2,3$ if the dimension is not specified.
\begin{thm}\label{a_r}
For any $\lambda_{0}>0$, there exist $c_{1},c_{2}>0$ such that for any $\lambda\in(0,\lambda_{0}]$ and $r\geq2$, we have
\[c_{1}e^{-r\xi_{[d]}(k,\lambda)}\leq a_{r,d}(\lambda)\leq c_{2}e^{-r\xi_{[d]}(k,\lambda)}.\]
\end{thm}
We also record the following up-to-constants bounds of non-intersection probability of Brownian motions as a direct corollary of Theorem \ref{a_r}.
\begin{coro}\label{intersect prob}
Let $X_{1},\ldots,X_{k},Y_{1},\ldots,Y_{l}$ be $(k+l)$ independent Brownian motions started uniformly from $\mathcal{C}_0$, and stopped upon reaching $\partial\mathcal{B}(0,e^{r})$. We have
\[\mathbb{P}\bigg(\Big(\bigcup_{i=1}^{k}X_{i}\Big)\cap\Big(\bigcup_{j=1}^{l}Y_{j}\Big)=\varnothing\bigg)\asymp e^{-r\xi_{[d]}(k,l)}.\]
\end{coro}
\par
For later use, we also recall the following setup using Brownian excursions; see e.g.\ \cite{LSW02b}.
\begin{lemma}\label{ni prob}
Let $X_{1},\ldots,X_{k},Y_{1},\ldots,Y_{l}$ be $(k+l)$ independent Brownian excursions in $\mathcal{A}(0,1,e^r)$ with one endpoint uniformly sampled on $\mathcal{C}_0$ and another endpoint on $\partial\mathcal{B}(0,e^{r})$. We then have
\[\mathbb{P}\bigg(\Big(\bigcup_{i=1}^{k}X_{i}\Big)\cap\Big(\bigcup_{j=1}^{l}Y_{j}\Big)=\varnothing\bigg)\asymp r^{k+l}e^{-r\xi_{[d]}(k,l)}.\]
\end{lemma}
We then turn to disconnection exponents of planar Brownian motion.
\begin{defi}\label{de}
Let $W_{1},\ldots,W_{k}$ be $k$ independent planar Brownian motions started uniformly from $\mathcal{C}_0$, and stopped upon reaching $\partial\mathcal{B}(0,e^{r})$. Denote by $V_{k}$ the event that the origin is not disconnected from infinity by the union of these paths. Then the Brownian disconnection exponent $\xi(k)$ is defined by 
\[\xi(k):=-\log\lim_{r\to\infty}(\mathbb{P}(V_{k}))^{1/r},\]
where the limit exists by Fekete's lemma.
\end{defi}
By Theorem~\ref{a_r}, it follows immediately that $\mathbb{P}(V_{k})\asymp e^{-r\xi(k)}$ and $\xi(k)=\lim_{\lambda\downarrow0}\xi(k,\lambda)$. Hence, the Brownian disconnection exponent is given by 
\begin{equation}\label{eq:bde}
\xi(k)=\frac{(\sqrt{24k+1}-1)^{2}-4}{48}.
\end{equation}
In particular, $\xi(5)=2$.
\subsubsection*{Generalized Exponents}
We now discuss the generalized intersection and disconnection exponents defined in \cite{Qian21} using generalized radial restriction measures. It is later shown in \cite{GLQ22} (using the uniqueness of generalized radial restriction measures from \cite{CG25}) that the original definition is equivalent to the following one, similar to Definitions \ref{ie real} and \ref{de}, by taking an independent Brownian loop soup with intensity $\mathbf{c}\in (0,1]$ into consideration.
\begin{defi}\label{gie}
Let $X_{0},X_{1},\ldots,X_{k}$ be independent Brownian motions started uniformly from $\mathcal{C}_0$, and stopped upon reaching $\partial\mathcal{B}(0,e^{r})$. Let $\Gamma$ be a Brownian loop soup with intensity $\mathbf{c}\in (0,1]$ in the plane, and $\Gamma_{r}$ be the collection of loops in $\Gamma$ which are fully contained in $\mathcal{B}(0,e^{r})$. We write $\widetilde{X}_{i}$ for the union of $X_{i}$ and the clusters in $\Gamma_{r}\backslash\Gamma_{0}$ that it intersects. Let $\Pi_{r}$ be the $\sigma$-algebra generated by $X_{1},\ldots,X_{k},\Gamma_{r}\backslash\Gamma_{0}$. Define
\[\widetilde{Z}_{r}:=\mathbb{P}(X_{0}\cap(\cup_{j=1}^{k}\widetilde{X}_{j})=\varnothing|\Pi_{r}).\]
For all $\lambda>0$, let $p(\mathbf{c},k,r,\lambda)=\mathbb{E}[\widetilde{Z}_{r}^{\lambda}]$. Then the generalized intersection exponent $\xi_{\mathbf{c}}(k,\lambda)$ is defined by
\[\xi_{\mathbf{c}}(k,\lambda):=-\log\lim_{r\to\infty}(p(\mathbf{c},k,r,\lambda))^{1/r}.\]
\end{defi}
See \cite[Theorem 3.1]{GLQ22} for details. We obtain the definition of generalized disconnection exponents by taking the $\lambda\to0$ limit in Definition \ref{gie}. It is shown in \cite[Theorem 3.9]{GLQ22}  that the following definition is the same as the original definition in \cite{Qian21}.
\begin{defi}\label{gde}
Recall the notation in Definition \ref{gie}. Let
\[p(\mathbf{c},k,r,0)=\mathbb{P}(\widetilde{Z}_{r}>0)=\lim_{\lambda\to0+}p(\mathbf{c},k,r,\lambda).\]
Then the generalized disconnection exponent $\xi_{\mathbf{c}}(k)$ is the unique exponent such that
\[p(\mathbf{c},k,r,0)\asymp e^{-r\xi_{\mathbf{c}}(k)}.\]
\end{defi}
In the setting of Brownian excursions, we have
\begin{lemma}[{\cite[Theorem 3.8]{GLQ22}}]\label{gnd prob}
Let $X_{1},\ldots,X_{k}$ be $k$ independent Brownian excursions in $\mathcal{A}(0,1, e^r)$ with one endpoint uniformly sampled on $\mathcal{C}_0$ and another endpoint on $\partial\mathcal{B}(0,e^{r})$. Recall the notation $\Gamma_{r}$ in Definition \ref{gie}, we have
\[\mathbb{P}(\widetilde{V}_{k})\asymp r^{k}e^{-r\xi_{\mathbf{c}}(k)},\]
where $\widetilde{V}_{k}$ is the event that the union of $X_{i}$, $i=1,\ldots,k$ and $\Gamma_{r}\backslash\Gamma_{0}$ does not disconnect the origin from infinity.
\par
Moreover, we can take $\mathbf{c}=0$ in this lemma, namely we do not add a Brownian loop soup, and $\xi_{0}(k)=\xi(k)$ is the Brownian disconnection exponent in \eqref{eq:bde}.
\end{lemma}

The value of generalized disconnection exponents, derived in \cite{Qian21}, is given through the following formula:
\begin{equation}\label{gde value}\xi_{\mathbf{c}}(k)=\frac{1}{48}\left(\left(\sqrt{24k+1-\mathbf{c}}-\sqrt{1-\mathbf{c}}\right)^{2}-4(1-\mathbf{c})\right).\end{equation}
The generalized disconnection exponents were then used in \cite{GLQ22} to compute the Hausdorff dimension of multiple points on boundaries of Brownian loop-soup clusters. 
\begin{thm}[{\cite[Theorem 1.1]{GLQ22}}]
Let $\Gamma_{0}$ be a Brownian loop soup in $\mathbb{D}$ with intensity $\mathbf{c}\in(0,1]$. We say that a point is a simple (resp.\ double, n-tuple) point of $\Gamma_{0}$, if it is visited at least once (resp.\ twice, n times) by the loop(s) in $\Gamma_{0}$. The set of simple (resp.\ double) points is denoted by $\mathcal{S}$ (resp.\ $\mathcal{D}$). Then the following holds almost surely.
\par
For any cluster $K$ in $\Gamma_{0}$, we have
\[\dim_{\mathrm{H}}(\partial K\cap\mathcal{S})=2-\xi_{\mathbf{c}}(2),\quad\dim_{\mathrm{H}}(\partial K\cap\mathcal{D})=2-\xi_{\mathbf{c}}(4),\]
where $\dim_{\mathrm{H}}(\cdot)$ stands for the Hausdorff dimension.
\end{thm}
Since $\xi_{1}(4)=2$ by \eqref{gde value}, the Hausdorff dimension of double points on boundaries of the critical Brownian loop-soup clusters is 0.

\subsection{Separation Lemmas}\label{subsec:sep}
We now turn to the separation lemmas for non-intersection, non-disconnection or generalized non-disconnection events, which are the key ingredients for various up-to-constants moment bounds.
We omit their proofs as they are either directly cited from the literature (see e.g.\ \cite{Law96,Law96a,LSW02b,KM10,LV12,HLLS22}) or trivial generalizations of existing ones.
\par
We begin with the separation lemma for non-disconnecting events. We first recall an explicit definition of  the separation event from \cite{KM10}, which we refer to as $\alpha$-nice configurations in this work.
\begin{defi}\label{nice}
Let $\gamma_{s}^{(j)}:0\leq s\leq\tau^{(j)}$, $j=1,2,\ldots,k$, be planar curves started from the boundary of a fixed annulus $A$ (with centre $x$) and stopped upon reaching the opposite boundary. This configuration of curves is called $\alpha$-nice if
\par
(1)\;$\{\gamma_{s}^{(j)}:0\leq s\leq\tau^{(j)}\}\backslash A\subset\mathcal{B}\big(\gamma_{0}^{(j)},\alpha|\gamma_{0}^{(j)}-x|\big)$, and
\par
(2)\;the following set does not disconnect $x$ from infinity:
\[\bigcup_{j=1}^{k}\big\{\gamma_{s}^{(j)}:0\leq s\leq\tau^{(j)}\big\}\cup\bigg[\bigcup_{j=1}^{k}\mathcal{B}\big(\gamma_{0}^{(j)},\alpha|\gamma_{0}^{(j)}-x|\big)\bigg]\cup\bigg[\bigcup_{j=1}^{k}\mathcal{B}\big(\gamma_{\tau^{(j)}}^{(j)},\alpha|\gamma_{\tau^{(j)}}^{(j)}-x|\big)\bigg].\]
\end{defi}
With this definition, we now state a standard version of the separation lemma for non-disconnection of Brownian motions.
\begin{lemma}[{\cite[Lemma 3.6]{KM10}}]\label{sep PTP}
Fix an integer $k$. For $r_{1}<r_{2}$,  let $X_{i}$, $i=1,2,\ldots,k$, be independent Brownian motions started uniformly from either $\partial\mathcal{B}(0,e^{r_{1}})$ or $\partial\mathcal{B}(0,e^{r_{2}})$ and stopped upon hitting the opposite boundary of the annulus $\mathcal{A}(0,e^{r_{1}},e^{r_{2}})$. Let $p(r_{1},r_{2},k)$ be the probability that the union of $X_{i}$ does not disconnect $\partial\mathcal{B}(0,e^{r_{1}})$ from infinity, and $p_{\alpha}(r_{1},r_{2},k)$ be the probability that the paths form an $\alpha$-nice configuration. Then there exist $c,c'>0$, and $\alpha_{0}>0$, such that for any $r_{1}<r_{2}$ and $\alpha\in[0,\alpha_{0}]$,
\[c\alpha^{k}e^{(r_{1}-r_{2})\xi(k)}\leq p_{\alpha}(r_{1},r_{2},k)\leq p(r_{1},r_{2},k)\leq c'e^{(r_{1}-r_{2})\xi(k)}.\]
\end{lemma}
We then turn to the separation lemma for non-intersection of Brownian motions. We define the $\alpha$-nice configuration in a similar fashion as follows.
\begin{defi}\label{nice12}
Let $\gamma_{s}^{(j)}:0\leq s\leq\tau^{(j)}$, $j=1,2,\ldots,p$ and $\eta_{s}^{(k)}:0\leq s\leq\tau^{(k)}$, $k=1,2,\ldots,q,$ be planar curves started on the boundary of a fixed annulus $A$ (with centre $x$) and stopped upon reaching the opposite boundary circle. This configuration of curves is called $\alpha(p,q)$-nice if
\par
(1)\;$\{\gamma_{s}^{(j)}:0\leq s\leq\tau^{(j)}\}\backslash A\subset\mathcal{B}(\gamma_{0}^{(j)},\alpha|\gamma_{0}^{(j)}-x|)$, and
\par
(2)\;the set
\[\bigcup_{j=1}^{p}\big\{\gamma_{s}^{(j)}:0\leq s\leq\tau^{(j)}\big\}\cup\bigg[\bigcup_{j=1}^{p}\mathcal{B}\big(\gamma_{0}^{(j)},\alpha|\gamma_{0}^{(j)}-x|\big)\bigg]\cup\bigg[\bigcup_{j=1}^{p}\mathcal{B}\big(\gamma_{\tau^{(j)}}^{(j)},\alpha|\gamma_{\tau^{(j)}}^{(j)}-x|\big)\bigg]\]
and the set
\[\bigcup_{k=1}^{q}\big\{\eta_{s}^{(k)}:0\leq s\leq\tau^{(k)}\big\}\cup\bigg[\bigcup_{k=1}^{q}\mathcal{B}\big(\eta_{0}^{(k)},\alpha|\eta_{0}^{(k)}-x|\big)\bigg]\cup\bigg[\bigcup_{k=1}^{q}\mathcal{B}\big(\eta_{\tau^{(k)}}^{(k)},\alpha|\eta_{\tau^{(k)}}^{(k)}-x|\big)\bigg]\]
do not intersect.
\end{defi}
We now state the separation lemma for non-intersection events for Brownian motions, and we refer readers to \cite[Section 3]{Law96} for the proof of a similar one.
\begin{lemma}\label{sep PDCP}
Fix integers $k,l$. For $r_{1}<r_{2}$, let $X_{i}$, $i=1,2,\ldots,k$ and $Y_{j}$, $j=1,2,\ldots,l$, be independent Brownian motions started uniformly from either $\partial\mathcal{B}(0,e^{r_{1}})$ or $\partial\mathcal{B}(0,e^{r_{2}})$ and stopped upon hitting the opposite boundary of the annulus $\mathcal{A}(0,e^{r_{1}},e^{r_{2}})$. Let $q(r_{1},r_{2},k,l)$ be the probability that the union of $X_{i}$ and the union of $Y_{j}$ do not intersect, and $q_{\alpha}(r_{1},r_{2},k,l)$ be the probability that the paths form an $\alpha(k,l)$-nice configuration. Then there exist $c,c'>0$, and $\alpha_{0}>0$, such that for any $r_{1}<r_{2}$ and $\alpha\in[0,\alpha_{0}]$,
\[c\alpha^{k+l}e^{(r_{1}-r_{2})\xi(k,l)}\leq q_{\alpha}(r_{1},r_{2},k,l)\leq q(r_{1},r_{2},k,l)\leq c'e^{(r_{1}-r_{2})\xi(k,l)}.\]
\end{lemma}
We now cite from \cite[Lemma 3.5]{GLQ22} the separation lemma for generalized non-disconnection events of Brownian motions.
Note that \cite{GLQ22} uses the language of extremal distances for nice configurations and separation lemmas, while here we translate them into the language of non-disconnection probabilities to suit our needs.
\begin{lemma}\label{sep BDP}
Recall \cite[Definition 3.7]{GLQ22} for the $\alpha$-separated event for generalized non-disconnection of Brownian motions. Fix an integer $k>0$. Let $X_{i}$, $i=1,2,\ldots,k$, be independent Brownian motions started uniformly from either $\partial\mathcal{B}(0,e^{r_{1}})$ or $\partial\mathcal{B}(0,e^{r_{2}})$ and stopped upon hitting the opposite boundary of the annulus $\mathcal{A}(0,e^{r_{1}},e^{r_{2}})$. Let $\Gamma$ be an independent Brownian loop soup in $\mathcal{B}(0,e^{r_{2}})$ with intensity $\mathbf{c}$. Let $p(r_{1},r_{2},k,\mathbf{c})$ be the probability that the union of $X_{i}$'s and $\Gamma$ do not disconnect $\partial\mathcal{B}(0,e^{r_{1}})$ from infinity, and $p_{\alpha}(r_{1},r_{2},k,\mathbf{c})$ be the probability that the paths are $\alpha$-separated. Then there exists $c',c''>0$, and $\alpha_{0}>0$, such that for any $r_{1}<r_{2}$ and $\alpha\in[0,\alpha_{0}]$,
\[c'\alpha^{k}e^{(r_{1}-r_{2})\xi_{\mathbf{c}}(k)}\leq p_{\alpha}(r_{1},r_{2},k,\mathbf{c})\leq p(r_{1},r_{2},k,\mathbf{c})\leq c''e^{(r_{1}-r_{2})\xi_{\mathbf{c}}(k)}.\]
\end{lemma}

\medskip

Some arguments of this work requires separation lemmas with initial configurations, which we now state without proof below.
\par
We first consider the non-disconnection case: let $W_{1},\ldots,W_{k}$ be independent planar Brownian motions started from $x_{1},\ldots,x_{k}\in\mathcal{C}_{0}$ and stopped upon reaching $\mathcal{C}_{j}$, and let $\tau_{i}$ be the first hitting time of $\mathcal{C}_{j}$ with respect to $W_{i}$. Let $V\subset\mathcal{D}_{0}$ be a closed set, which we also refer to as ``initial configurations'', regardless whether it is actually the trace of some curves or not. We require that $V$ does not disconnect $\mathcal{C}_{-n}$ from infinity. We then define the $\alpha$-nice-at-the-end event:
\[\mathrm{End}_{j,1}^{\alpha}:=\left\{ \begin{array}{c}
    \text{The union of $V$ and $W_{i}[0,\tau_{i}],\mathcal{B}(W_{i}(\tau_{i}),\alpha 2^{j})$ for}\\ 
    \text{all $1\le i\le k$ does not disconnect $\mathcal{C}_{0}$ from infinity}
 \end{array}	\right\}.\]
We also write
\[A_{j,1}:=\{\mbox{The union of $V$ and $W_{i}[0,\tau_{i}]$ for all $1\le i\le k$ does not disconnect $\mathcal{C}_{-n}$ from infinity}\}.\]
We then have
\begin{lemma}\label{sep PTP-1}
There exists $c=c(k,\alpha)>0$, such that for all $n\ge 1$ and all $j>0$,
\[\mathbb{P}^{x_{1},\ldots,x_{k}}(\mathrm{End}_{j,1}^{\alpha}|A_{j,1})\geq c.\]
Moreover, the same result holds if $W_{i}$'s are Brownian excursions.
\end{lemma}
For the non-intersection separation lemma with initial conditions, we omit the precise statement here for sake of brevity. We refer the reader to \cite[Lemma 3.2]{LV12} for the three-dimensional case, and \cite[Lemma 3.4]{Law95} for the two-dimensional case; see also \cite[Section 3]{HLLS22} for multiple variants of 2D and 3D separation lemmas.
\par
For the generalized non-disconnection separation lemma, we can similarly define the $\alpha$-separated-at-the-end event, by considering only the ``landing zones'' at the end of the path in \cite[Definition 3.7]{GLQ22}. Again, we omit the precise statement here and refer readers to \cite[Section 4]{GLQ22} for the explicit form of separation lemma and its proof.
\section{Good Boxes and Moment Estimates}\label{sec:moment}
In this section, we introduce the notion of various (random) ``good boxes'', which are used to  approximate the exceptional sets of points under consideration. We then derive moment bounds for such boxes. Since Sections \ref{sec:ptp}-\ref{sec:bdp} are parallel to each other, we will use the same symbols (e.g.\ $\mathfrak{S}_n$) to denote good boxes (and other related events) for the different types of exceptional sets.

\subsection{Definition of Good Boxes}
We start by defining good boxes for pioneer triple points of 2D Brownian motion $W_{t}$ started from $0$ and stopped upon exiting $\mathbb{D}$. This definition is only used in Section~\ref{sec:ptp}.
\begin{defi}\label{gs}
Fix $\delta>0$. Let $N(\delta)$ be the smallest integer $n$ such that $2^{-n}<\delta/16$. For $n\geq N(\delta)$, let $\mathfrak{S}_{n}=\mathfrak{S}_{n}(\delta)$ stand for the collection of $\delta$-good boxes, i.e.\ the set of all $S\in\mathcal{S}_{n}$ with the following properties:
\par
(1) $S$ is visited three times by $W_{t}$ and the trajectories between successive visits reach distance at least $(\delta-2^{-n-1/2})$ to the centre of $S$. Using the notation in Section~\ref{subsec:notation}, that is,
\begin{equation}\label{eq:Tc}
\mathcal{T}(S):=\tau(S,\partial\mathcal{B}(S,\delta-2^{-n-1/2}),S,\partial\mathcal{B}(S,\delta-2^{-n-1/2}),S)<\tau_{\mathbb{D}}.
\end{equation}
We use $V_{\delta}^{\mathrm{PTP}}(S)$ to denote this event.
\par
(2) $\mathcal{D}_{-n}(S)$ is not disconnected from infinity by the path $\{W_{s}:0\leq s\leq\mathcal{T}(S)\}$.
\end{defi}
\begin{rmk}\label{rm:gb}
We have introduced the ``$-2^{-n-1/2}$'' term in Definition~\ref{gs} to ensure the following fact: for two dyadic boxes $S\in \mathcal{S}_n$ and $T\in\mathcal{S}_{n-1}$ with $S\subset T$, if $S$ is $\delta$-good, then so is $T$. Furthermore, we observe that if a point $x$ is contained in every member of a decreasing sequence of $\delta$-good boxes, then $x$ is a pioneer triple point. The same observations also hold for pioneer double cut points of Brownian motion and boundary double points of the Brownian loop soup later. 
\end{rmk}
\par
We now give the definition of good boxes for pioneer double cut points (in both 2D and 3D) for a Brownian motion $W_{t}$ started from $0$ and stopped upon exiting $\mathbb{D}$. This definition is only used in Section~\ref{sec:pdcp}.
\begin{defi}\label{gs12}
Fix $\delta>0$. Let $N(\delta)$ be the smallest integer $n$ such that $2^{-n}<\delta/16$. For $n\geq N(\delta)$, let $\mathfrak{S}_{n}=\mathfrak{S}_{n}(\delta)$ stand for the collection of $\delta$-good boxes, i.e., the set of all $S\in\mathcal{S}_{n}$ with the following properties:
\par
(1) $S$ is visited twice by $W_{t}$ and the trajectories between successive visits reach distance at least $(\delta-2^{-n-1/2})$ to the centre of $S$. More precisely,
\[\mathcal{T}(S):=\tau(S,\partial\mathcal{B}(S,\delta-2^{-n-1/2}),S)<\tau_{\mathbb{D}}.\]
We use $V_{\delta}^{\mathrm{PDCP}}(S)$ to denote this event.
\par
(2) $W[0,\tau(\mathcal{D}_{-n}(S))]$ and $W[\tau(S,\mathcal{C}_{-n}(S)),\mathcal{T}(S)]$ do not intersect.
\par
For $d=2$, we additionally require that:
\par
(3) There exist two different connected components of $\mathcal{B}(S,\delta-2^{-n-1/2})\setminus W[0,\mathcal{T}(S)]$ such that both of them intersect $\mathcal{D}_{-n}(S)$ and are connected to infinity.
\end{defi}
\begin{rmk}
We add the additional condition (3) in the 2D case for the following technical reason. In order to control the size of the intermediate ``bridges'' in Lemma~\ref{bridge12}, we apply the non-disconnection probabilities in the 2D case, where the condition (3) is used. In the 3D case, such an estimate can be simply obtained from transience of 3D Brownian motion.
\end{rmk}
Let us now define good boxes in the BDP case, which will be used in Section~\ref{sec:bdp}.
Note that a double point can be visited by a single loop or two different loops. We will in fact prove non-existence of the first type of BDP (i.e.\ visited by a single loop), and argue that the proof for the second type is very similar.

In the following, we recall the definition of $\delta$-good boxes visited by a single loop $\gamma$ in \cite{GLQ22}, and discuss in Remark \ref{rm:twoloops} the definition of $\delta$-good boxes visited by two different loops.
Recall the bubble measure and the loop measure introduced in Sections~\ref{subsec:path_measure} and~\ref{subsec:bls}, respectively.

\begin{defi}\label{gs4}
Fix $\iota\in(0,1/2)$. Let
\begin{equation}\label{eq:mu_iota}
    \mu_{\iota}:=\frac{1}{2\pi}\int_{0}^{2\pi}\int_{\iota}^{1-\iota}\mu_{\mathcal{B}(0,r)}^{\mathrm{bub}}(re^{i\theta})(\cdot\,;\mathrm{diam}(\gamma)>\iota/2)\,r\,\mathrm{d}r\,\mathrm{d}\theta, \quad \mu_{\iota}^{\#}:=\mu_{\iota}/|\mu_{\iota}|,
\end{equation}
where $\mu_{\mathcal{B}(0,r)}^{\mathrm{bub}}(re^{i\theta})(\cdot\, ;\mathrm{diam}(\gamma)>\iota/2)$ stands for the bubble measure restricted to loops with diameter greater than $\iota/2$. Note that the total mass $|\mu_{\iota}|$ is finite due to this restriction, and hence $\mu_{\iota}^{\#}$ is well defined. 
Let $\gamma$ be a (rooted) Brownian loop sampled according to $\mu_{\iota}^{\#}$, and let $t_{\gamma}$ be its duration. 

Let $N(\delta)$ be the smallest integer $n$ such that $2^{-n}<\delta/16$. For $n\geq N(\delta)$, let $\mathfrak{S}_{n}=\mathfrak{S}_{n}(\delta)$ be the collection of $\delta$-good boxes, i.e.\ the set of all $S\in\mathcal{S}_{n}$ with the following properties:
\par
(1) $S$ is visited twice by $\gamma$ and the path makes an excursion of distance at least $\delta-2^{-n-1/2}$ between visits. More precisely,
\[\tau_{\gamma}(S,\partial\mathcal{B}(S,\delta-2^{-n-1/2}),S)<t_{\gamma},\]
and we use $V_{\delta}^{\mathrm{BDP}}(S)$ to denote this event.
\par
(2) $S$ is not disconnected from infinity by $\widetilde{\gamma}$ (the cluster of $\{\gamma\}\cup\Gamma_{0}$ which contains $\gamma$). 
\end{defi}
\begin{rmk}\label{rm:twoloops}
We can also define good boxes which are visited by two different loops $\gamma_{1},\gamma_{2}$ independently sampled according to $\mu_{\iota}^\#$. In this case, conditions (1) and (2) are replaced by
\par
(1') $S$ is visited by both $\gamma_{1}$ and $\gamma_{2}$.
\par
(2') $\gamma_{1},\gamma_{2}$ belong to the same cluster in $\Gamma_{0}$, and $S$ is not disconnected from infinity by the cluster of $\{\gamma_{1},\gamma_{2}\}\cup\Gamma_{0}$ that contains $\gamma_{1},\gamma_{2}$.
\end{rmk}
\subsection{Moment Bounds for Good Boxes}
We now turn to moment bounds for good boxes in the previous subsection. As bounds for all three cases take the same form, we will state them in a unified manner\footnote{When cited later, we will specify which case (e.g., ``Lemma \ref{fme} for PTP'') we are referring to.}. For proofs, we will discuss in detail the PTP case and be quite brief for PDCP and BDP cases, as they follow from similar arguments if we replace some ingredients appropriately; e.g., Lemma \ref{sep PTP} for PTP will be replaced by Lemmas \ref{sep PDCP} (for PDCP) and \ref{sep BDP} (for BDP).
\par
Throughout this section, we fix the constants $\iota\in(0,1/2)$, $0<\delta\leq\iota/2$, and assume that $S\in\mathcal{S}_{n}$ and $S\subset\mathcal{A}(0,\iota,1-\iota)$. We allow constants to depend on $\iota$ and $\delta$, but they should be independent of $n$ and $S$.
\begin{lemma}\label{fme}
Recall that $\mathfrak{S}_{n}$ stands for the collection of $\delta$-good boxes in Definitions \ref{gs}, \ref{gs12} and \ref{gs4} respectively for three cases. Let $d$ be the dimension of the space, which is $2$ or $3$ ($d=3$ is only relevant in the PDCP case). There exist $c_{1},c_{2}>0$, such that for all $n\geq N(\delta)$ and any dyadic box $S\in\mathcal{S}_{n}$ such that $S\subset\mathcal{A}(0,\iota,1-\iota)$, for all three cases we have
\[c_{1}2^{-dn}\leq\mathbb{P}(S\in\mathfrak{S}_{n})\leq c_{2}2^{-dn}.\]
\end{lemma}
The proof of Lemma~\ref{fme} for the PTP case is almost the same as the proof of \cite[Lemma~3.2]{KM10}, except that we need to replace $\xi(4)$ in \cite[Lemma 3.2]{KM10} by $\xi(5)=d=2$ here. The result for the BDP case was obtained in \cite[Lemma 5.3]{GLQ22}. Bounds for the PDCP case follow from similar path decomposition (and the corresponding version of separation lemma (Lemma \ref{sep PDCP}) additionally for the lower bound) as the other two cases. We omit the details for brevity.
\par
We now turn to the second moment.
\begin{lemma}\label{sme}
For all three cases, there exist $c_{1},c_{2}>0$, such that for all $n\geq N(\delta)$, any $m\in\mathbb{Z}^+$ such that $m\le n-2$ and $2^{-m}\ll\delta$, and any dyadic boxes $S,T\subset\mathcal{A}(0,\iota,1-\iota)$ such that $S,T\in\mathcal{S}_{n}$ with $\mathrm{dist}(S,T)\in [2^{-m-1},2^{-m}]$, we have
\[c_{1}2^{d(m-2n)}\leq\mathbb{P}(S,T\in\mathfrak{S}_{n})\leq c_{2}2^{d(m-2n)}.\]
\end{lemma}
For the PTP case, the upper bound in Lemma \ref{sme} follows from similar arguments as the proof of \cite[Lemma~3.2]{KM10}. For the BDP case, the upper bound has been proved in \cite[Lemma~5.3]{GLQ22}. The proof of the upper bound for the PDCP case is very similar to those for the PTP and BDP cases, so we omit details. In the following, we prove the lower bound for the PTP case, and omit the proof of the other two cases.
\begin{proof}[Proof of the lower bound in Lemma \ref{sme} for PTP]
Let $z$ be the midpoint of centres of $S$ and $T$. Note that if 
\[\mathcal{T}:=\tau(S,T,\partial\mathcal{B}(z,3\delta/4),T,S,\partial\mathcal{B}(z,3\delta/4),S,T)<\tau_{\mathbb{D}},\]
and both of $\mathcal{D}_{-n}(S)$ and $\mathcal{D}_{-n}(T)$ are not disconnected from infinity by the path $\{W_{s}:0\leq s\leq\mathcal{T}\}$, then we have $S,T\in\mathfrak{S}_{n}$. Therefore, we only need to show a corresponding lower bound for this event.
For this, we consider three disjoint annuli:
\[\mathcal{A}_{0}:=\mathcal{A}(z,2^{-m+3},3\delta/4),\ \mathcal{A}_{1}:=\mathcal{A}(S,2^{-n},2^{-m-2}),\ \mathcal{A}_{2}:=\mathcal{A}(T,2^{-n},2^{-m-2}).\]
We then consider an increasing sequence of stopping times with respect to $W_{t}$: let
\[\tau_{1}:=\tau(\partial_{i}\mathcal{A}_{1}),\ \tau_{2}:=\tau(\tau_{1},\partial_{i}\mathcal{A}_{2}),\ \tau_{3}:=\tau(\tau_{2},\partial_{o}\mathcal{A}_{0}),\ \tau_{4}:=\tau(\tau_{3},\partial_{i}\mathcal{A}_{2}),\ \tau_{5}:=\tau(\tau_{4},\partial_{i}\mathcal{A}_{1}),\]
\[\tau_{6}:=\tau(\tau_{5},\partial_{o}\mathcal{A}_{0}),\ \tau_{7}:=\tau(\tau_{6},\partial_{i}\mathcal{A}_{1}),\ \tau_{8}:=\tau(\tau_{7},\partial_{i}\mathcal{A}_{2}),\ \tau_{9}:=\tau(\tau_{8},T).\]
For $i=1,2,\ldots,9$, let $\beta_{i}=W[\tau_{i-1},\tau_{i}]$ (letting $\tau_{0}=0$). We then further decompose $\beta_{i}$ into ``crossings'' in annulus $\mathcal{A}_{j}$ and ``links'' between ``crossings''. As an example, we decompose $\beta_{1}$ into \[\beta_{1}:=\lambda_{1}\oplus\eta_{1}\oplus\lambda_{2}\oplus\eta_{2},\]
where $\lambda_{i}$'s are ``links'' and $\eta_{i}$'s are ``crossings''. We have similar decomposition for $\beta_{i}$, $i=2,\ldots,8,$ such that
\[\beta_{i}:=\lambda_{2i-1}\oplus\eta_{2i-1}\oplus\lambda_{2i}\oplus\eta_{2i},\]
and $\beta_{9}:=\lambda_{17}$ is a ``link'' itself; see Figure \ref{fig_ST_decompose}.
\begin{figure}[ht]
\centering
\includegraphics[width=9cm]{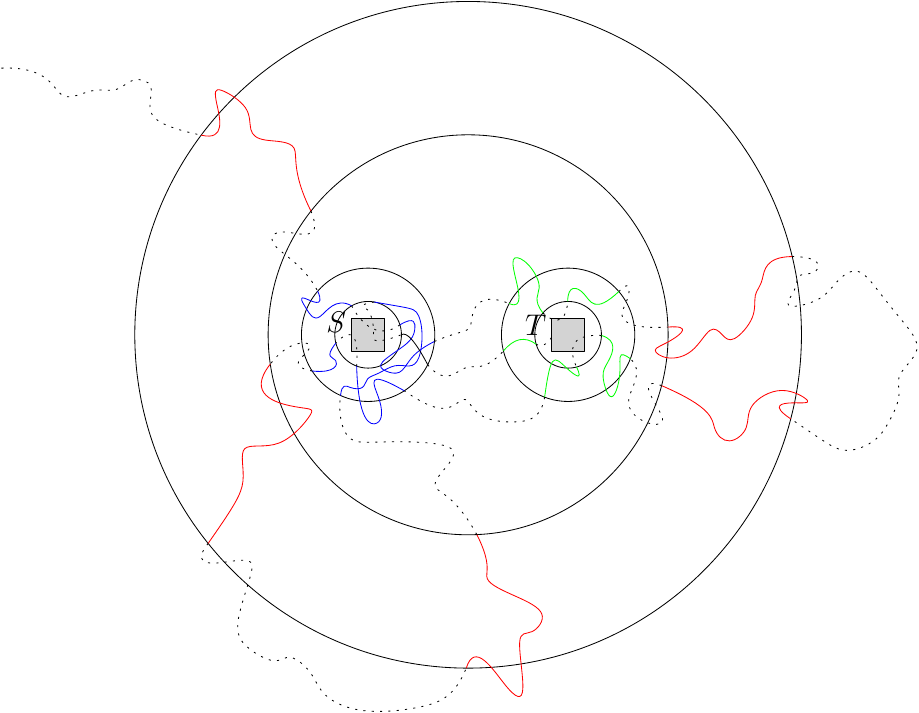}
\caption{The ``links'' and ``crossings'' in the aforementioned decomposition, where ``links'' are dotted. $\eta_{1},\eta_{6},\eta_{7},\eta_{12},\eta_{13}$ are in red. $\eta_{2},\eta_{3},\eta_{10},\eta_{11},\eta_{14}$ are in blue. $\eta_{4},\eta_{5},\eta_{8},\eta_{9},\eta_{16}$ are in green.}
\label{fig_ST_decompose}
\end{figure}
We then consider the following three events (for some fixed $\alpha>0$): 
\par
($F_{1}$)\quad $\eta_{2},\eta_{3},\eta_{10},\eta_{11},\eta_{14}$ form an $\alpha$-nice configuration;
\par
($F_{2}$)\quad $\eta_{4},\eta_{5},\eta_{8},\eta_{9},\eta_{16}$ form an $\alpha$-nice configuration;
\par
($F_{3}$)\quad $\eta_{1},\eta_{6},\eta_{7},\eta_{12},\eta_{13}$ form an $32\alpha$-nice configuration. 
\par
Thanks to \cite[Lemma 3.3]{KM10}, these curve segments are uniformly equivalent to (see the definition in Section~\ref{subsec:notation}) independent Brownian motions (starting from the boundary of annulus and stopped upon the first time reaching the opposite boundary). Therefore, by Lemma \ref{sep PTP} we have
\[\mathbb{P}(F_{1}\cap F_{2}\cap F_{3})\asymp 2^{-2m}\cdot(2^{2m-2n})^{2}=2^{2m}2^{-4n}.\]
Conditioned on $F_{1}\cap F_{2}\cap F_{3}$, there exist some well-chosen ``tubes'' (see Figure \ref{fig2} for an illustration) such that if we force the ``links'' to stay inside corresponding tubes then the desired event $\{S,T\in\mathfrak{S}_{n}\}$ occurs. Moreover, the aspect ratio of these tubes can be chosen to be some small constant, and thus the probability of ``links'' staying in respective tubes is bounded from below by some universal constant. Therefore we conclude that
\[\mathbb{P}(S,T\in\mathfrak{S}_{n})\geq c2^{2m}2^{-4n},\]
which completes the proof.
\end{proof}
\begin{figure}[H]
\centering
\includegraphics[width=4cm]{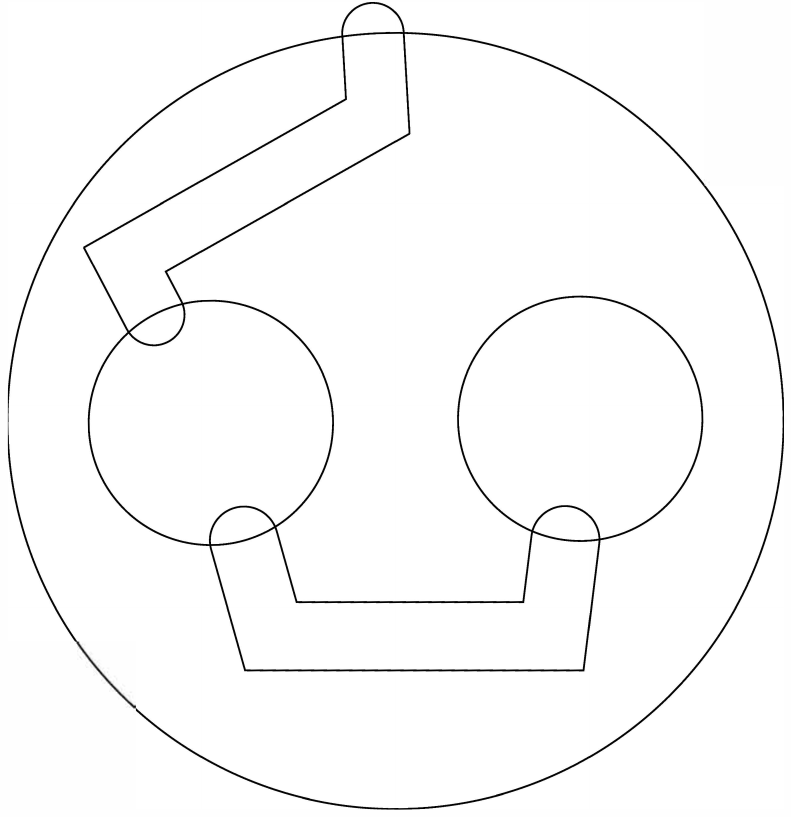}
\caption{\small An example of ``tubes'' (in the intermediate region), where the outer circle is $\mathcal{D}_{-m+3}(z)$, and two inner circles are $\mathcal{D}_{-m-2}(S)$ and $\mathcal{D}_{-m-2}(T)$.}
\label{fig2}
\end{figure}
Finally, we give an upper bound for the third moment of $\delta$-good boxes.
\begin{lemma}\label{tme}
For all three cases, there exists $c>0$, such that for all $n\geq N(\delta)$ and any $n$-boxes $S,T,U\subset\mathcal{A}(0,\iota,1-\iota)$ such that $S,T,U\in\mathcal{S}_{n}$, we have
\[\mathbb{P}(S,T,U\in\mathfrak{S}_{n})\leq cD_{1}^{-d}D_{2}^{-d}2^{-3dn},\]
where 
\begin{center}
$D_{1}=\min\{\mathrm{dist}(S,T),\mathrm{dist}(T,U),\mathrm{dist}(U,S)\}$ and $D_{2}=\max\{\mathrm{dist}(S,T),\mathrm{dist}(T,U),\mathrm{dist}(U,S)\}$.     
\end{center}
\end{lemma}
\begin{proof}
Since the proof in all three cases are almost the same, we only write down the PTP case. We separately consider the following two possibilities:
\par
Case 1: $D_{2}\geq 256D_{1}$. Without loss of generality, we assume that $D_{1}=\mathrm{dist}(T,U)$, $D_{2}=\mathrm{dist}(S,T)$. We first assume that $D_{2}<\delta/16$. Let $z$ be the midpoint of centres of $T,U$, and $y$ be the midpoint of $z$ and the centre of $S$. If $S,T,U\in\mathfrak{S}_{n}$, then for each $X\in\{S,T,U\}$, Brownian motion should visit the sets $(X,\partial\mathcal{B}(y,3\delta/4),X,\partial\mathcal{B}(y,3\delta/4),X)$ in order.
\par
Although there are (finitely) many different possibilities when we count the order of hittings of these boxes, we can derive the same probability upper bound for all cases (with possibly different constants) with basically the same reasoning. To illustrate, we consider the event $E_0$ that the path visits these sets in the following order:
\[(E_{0})\quad S\rightsquigarrow T\rightsquigarrow U\rightsquigarrow\partial\mathcal{B}(y,3\delta/4)\rightsquigarrow U\rightsquigarrow T\rightsquigarrow S\rightsquigarrow\partial\mathcal{B}(y,3\delta/4)\rightsquigarrow S\rightsquigarrow T\rightsquigarrow U.\]
We now write
\[\mathcal{A}_{1}:=\mathcal{A}(y,2D_{2},3\delta/4),\ \mathcal{A}_{2}:=\mathcal{A}(S,2^{-n},D_{2}/4),\ \mathcal{A}_{3}:=\mathcal{A}(z,2D_{1},D_{2}/4),\] 
\[\mathcal{A}_{4}:=\mathcal{A}(T,2^{-n},D_{1}/4),\ \mathcal{A}_{5}:=\mathcal{A}(U,2^{-n},D_{1}/4).\]
\par
We now consider a similar decomposition as in the proof of Lemma \ref{sme}. In this case, we only use conditions on ``crossings'' to obtain the desired upper bound. Note that on the event $\{S,T,U\in\mathfrak{S}_{n}\}$, each annulus $\mathcal{A}_{i}$ contains five crossings, whose union does not disconnect the inner boundary of the annulus from infinity. Using a similar argument as \cite[Lemma 3.3]{KM10}, crossings are uniformly equivalent to five independent Brownian motions. This combined with Lemma~\ref{sep PTP} gives that
\[\mathbb{P}(E_{0})\leq c(1/D_{2})^{-d}(D_{2}/2^{-n})^{-d}(D_{2}/D_{1})^{-d}(D_{1}/2^{-n})^{-2d}=cD_{1}^{-d}D_{2}^{-d}2^{-3dn},\]
which completes the proof of this case. If $D_{2}\geq\delta/16$, it suffices to consider crossings in disjoint annuli $\mathcal{A}_{2},\mathcal{A}_{3},\mathcal{A}_{4},\mathcal{A}_{5}$ to obtain the same estimate since $D_{2}\asymp1$, and the proof still follows.
\par
Case 2: $D_{2}<256 D_{1}$. Again, we first assume that $D_{2}<\delta/16$. If centres of $S,T,U$ form an acute triangle, let $y$ be the circumcentre of the triangle, and otherwise, let $y$ be the midpoint of the longest edge of the triangle. Similar to Case 1, if $S,T,U\in\mathfrak{S}_{n}$, then Brownian motion visits the sets $(X,\partial\mathcal{B}(y,3\delta/4),X,\partial\mathcal{B}(y,3\delta/4),X)$ in order, for each $X\in\{S,T,U\}$.
\par
We take the following event $E_0'$ for illustration:
\[(E'_{0})\quad S\rightsquigarrow T\rightsquigarrow U\rightsquigarrow\partial\mathcal{B}(y,3\delta/4)\rightsquigarrow U\rightsquigarrow T\rightsquigarrow S\rightsquigarrow\partial\mathcal{B}(y,3\delta/4)\rightsquigarrow S\rightsquigarrow T\rightsquigarrow U.\]
Considering the following four annuli 
\begin{equation}\label{annuli}\mathcal{A}(y,2D_{2},3\delta/4),\ \mathcal{A}(S,2^{-n},D_{1}/4),\ \mathcal{A}(T,2^{-n},D_{1}/4),\ \mathcal{A}(U,2^{-n},D_{1}/4),\end{equation}
and non-disconnection events for corresponding crossings in them respectively, we have
\[\mathbb{P}(E'_{0})\leq c'(1/D_{2})^{-d}(D_{1}/2^{-n})^{-3d}\leq c'D_{1}^{-d}D_{2}^{-d}2^{-3dn}\]
since $D_{2}\leq D_{1}<256 D_{2}$. For the case that $D_{2}\geq\delta/16$, note that $D_{1}\asymp D_{2}\asymp1$, we can easily deduce the same estimate by considering crossings in the last three annuli defined in \eqref{annuli}. We thus conclude the proof of the lemma.
\end{proof}
Lemmas \ref{fme}, \ref{sme} and \ref{tme} immediately imply the following conditional moment estimate.
\begin{coro}\label{cme}
There exist $c,c_{1},c_{2}>0$, such that for any $n$ and any $S,T,U\in\mathcal{S}_n$, we have:
\[c_{1}\mathrm{dist}(S,T)^{-d}2^{-dn}\leq\mathbb{P}(T\in\mathfrak{S}_{n}|S\in\mathfrak{S}_{n})\leq c_{2}\mathrm{dist}(S,T)^{-d}2^{-dn},\]
and
\[\mathbb{P}(T,U\in\mathfrak{S}_{n}|S\in\mathfrak{S}_{n})\leq cD_{1}^{-d}D_{2}^{-d}2^{-2dn},\]
where $d=2,3$ is the dimension of the space, and $D_{1},D_{2}$ are defined in Lemma \ref{tme}.
\end{coro}
\begin{rmk}
From the proof of Lemmas \ref{fme}, \ref{sme} and \ref{tme}, we can further show that the conditional moment estimate in Corollary \ref{cme} does not depend on the path configuration close to $S$ or away from $S$, i.e.\ the path configuration in $\mathcal{D}_{-n+i}(S)$ or outside $\mathcal{B}(S,\delta/2^{i})$, where $i$ is an integer to be determined. We refer readers to Lemma \ref{cme-f} for the precise statement.
\end{rmk}
\section{Non-Existence of Pioneer Triple Points}\label{sec:ptp}
In this section we give the proof of Theorem \ref{ptp}. Throughout this section, let $\iota\in(0,1/2)$, $0<\delta\leq\iota/2$, and assume $S\in\mathcal{S}_{n}$, $S\subset\mathcal{A}(0,\iota,1-\iota)$. We allow constants to depend on $\iota$ and $\delta$, but they should be independent of $n$ and $S$. When we talk about ``good boxes'', we mean $\delta$-good boxes (i.e.\ boxes in $\mathfrak{S}_n(\delta)$), unless we specify otherwise.
\par
As discussed in Section \ref{sec:intro}, the key is to prove the following large deviation bound, which is also the precise version of \eqref{concentrate}. Using this bound, we will prove Theorem \ref{ptp} in Section \ref{subsec:pf-ptp}.
\begin{prop}\label{cgs}
There exist constants $c,\lambda,\gamma>0$ such that for all $n\ge N(\delta)$,
\[\mathbb{P}(\#\mathfrak{S}_{n}(\delta)<\lambda\sqrt{n}|S\in\mathfrak{S}_{n}(\delta+2^{-\sqrt{n}}))\leq cn^{-\gamma}.\]
\end{prop}
We briefly explain the idea of the proof and the structure of this section. Let $(N_{i})_{1\leq i\leq L/2}$ be the number of good boxes in $\mathcal{A}(S,2^{-n+iL},2^{-n+(i+1)L})$ where $L:=[\sqrt{n}]$. We want $(N_{i})_{1\leq i\leq L/2}$ to be adapted to some filtration under the probability measure \begin{equation}\label{conditioned on box}\widetilde{\mathbb{P}}_{n}:=\mathbb{P}(\cdot|S\in\mathfrak{S}_{n}(\delta+2^{-\sqrt{n}})).\end{equation}
Note that we require $S$ to be $(\delta+2^{-\sqrt{n}})$-good in order to ensure that good boxes near $S$ (i.e.\ within distance $2^{-\sqrt{n}}$ from $S$) are $\delta$-good. For adaptedness, we want the filtration to include information of the path configuration in $\mathcal{D}_{-n+mL}(S)$, for $0\leq m\leq L/2$. Hence we set (see \eqref{filtration} below for the rigorous definition)
\begin{equation*}
\mathscr{F}_{m}:=\sigma\left(X^{0},\;X^i,\;W^{i}[0,\tau_{i,mL}],\;W^{i}[\sigma_{i,n-K-L},\tau_{i,n-K}],\;i=1,\ldots,5\right),
\end{equation*}
where $K=-\log_{2}(\delta/2)$ (throughout, we will only consider values of $\delta$ for which $K$ is an integer), and $X^{i},W^{i}$'s are ``bridges'' and ``excursions'' decomposed from Brownian motion (see \eqref{excursion and bridge} for rigorous definitions), and
$\tau_{i,j}$ (resp.\ $\sigma_{i,j}$) are first (resp.\ last) hitting times of $\mathcal{C}_{-n+j}(S)$ with respect to $W^{i}$. 
\par
In order to ensure ``adaptedness'', i.e., the above filtration $\mathscr{F}_{m}$ contains all information of the path configuration in $\mathcal{D}_{-n+mL}(S)$, we need to first exclude some ``bad events'' (denoted by $B_i$'s throughout): back-tracking excursions, run-away bridges or multiple global ``openings''. Under a slightly tilted version of $\widetilde{\mathbb{P}}_{n}$ (that is, $\overline{\mathbb{P}}_n$ defined in \eqref{conditional measure}), in which we assign zero measure to bad events, we can obtain that $N_{i}$ is $\mathscr{F}_{(i+3)}$-measurable. We refer the reader to Section \ref{subsec:bad} for the rigorous construction of this filtration, and the precise characterization of ``bad events''.
\par
Once we complete the construction of the filtration, we will derive moment bounds of good boxes conditioned on this filtration. Proposition~\ref{cgs} finally follows from Paley-Zygmund inequality and these moment bounds. We refer the reader to Sections \ref{subsec:martingale}--\ref{subsec:pf-4.13} for details.
\par
Finally, we complete the proof of Theorem \ref{ptp} in Section \ref{subsec:pf-ptp}.
\subsection{Decomposition of the Brownian Path}\label{subsec:decompose}
We begin with excursion-bridge decomposition of Brownian motion, which has a similar flavour to the decomposition in \cite[Section 5.2]{GLQ22}. Recall the event $V_{\delta}^{\mathrm{PTP}}(S)$ and the stopping time $\mathcal{T}(S)$ in Definition~\ref{gs}. 
Below we will fix $S$ and write $\mathcal{T}=\mathcal{T}(S)$ for brevity.
On the event $V_{\delta}^{\mathrm{PTP}}(S)$, let 
\[u_{1}=\tau(\partial\mathcal{B}(S,\delta-2^{-n-1/2})).\]
For $i=1,2,3,4$, let 
\begin{equation}\label{decomposition}
\begin{split}
&v_{i}=\tau(u_{i},S),\quad u_{i+1}=\tau(v_{i},\partial\mathcal{B}(S,\delta-2^{-n-1/2})),\\
&s_{2i-1}=\sup\{t<v_{2i-1}:W_{t}\in\partial\mathcal{B}(S,\delta/2)\},\quad t_{2i-1}=\tau(s_{2i-1},\mathcal{C}_{-n}(S)),\\
&s_{2i}=\sup\{t<u_{2i+1}:W_{t}\in\mathcal{C}_{-n}(S)\},\quad t_{2i}=\tau(s_{2i},\partial\mathcal{B}(S,\delta/2)),\\
\end{split}
\end{equation}
and for $k=1,2$ let (recall that $\gamma^\mathcal{R}$ means the time reversal of $\gamma$)
\begin{equation}\label{excursion and bridge}
\begin{split}
&X^{0}=W[0,s_{1}],\ W^{1}=(W[s_{1},t_{1}])^{\mathcal{R}},\ X^{5}=W[t_{5},\mathcal{T}],\\
&X^{i}=W[t_{i},s_{i+1}],\ W^{2k}=W[s_{2k},t_{2k}],\ W^{2k+1}=(W[s_{2k+1},t_{2k+1}])^{\mathcal{R}}.\\
\end{split}
\end{equation}
See Figure \ref{fig_decompose} for an illustration of such a decomposition. We call $X^{i}$ ``bridge'' and $W^{i}$ ``excursion'' below. Note that all excursions are oriented from inside to outside in $\mathcal{A}(S,2^{-n},\delta/2)$.
\begin{figure}[b]
\centering
\includegraphics[width=10cm]{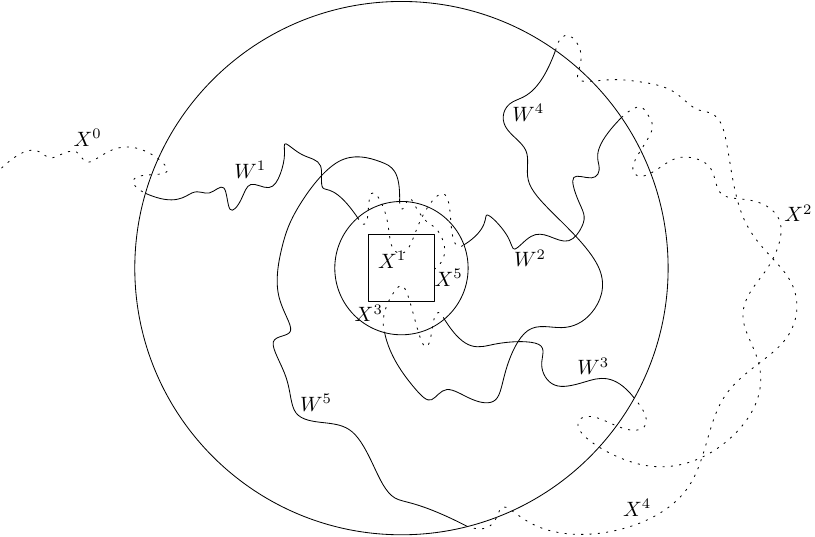}
\caption{\small The decomposition of the PTP case, where ``bridges'' are dotted.}
\label{fig_decompose}
\end{figure}
\par
The following lemma describes the law of $X^{i}$ and $W^{i}$ (see e.g. \cite[Lemma 5.4]{GLQ22}).
\begin{lemma}\label{ac excursions}
Given $S\in\mathcal{S}_{n}$, on the event $V_{\delta}^{\mathrm{PTP}}(S)$, the followings hold.
\par
(1) The joint law of $W^{1},\ldots,W^{5}$ is uniformly equivalent to the joint law of five independent Brownian excursions in $\mathcal{A}(S,2^{-n},\delta/2)$.
\par
(2) Conditioned on $X^{i}$, $i=0,1,\ldots,5$, $W^{1},\ldots,W^{5}$ are conditionally independent and are distributed as Brownian excursions with fixed starting and ending points (i.e.\ the normalized boundary-to-boundary measure in $\mathcal{A}(S,2^{-n},\delta/2)$).
\par
(3) Conditioned on $W(s_{i}),W(t_{i})$, $i=1,\ldots,5$, $X^{j}$ are conditionally independent of each other, and further conditionally independent of $W^{j}$. Moreover, if $i=1,3$ (resp.\ $2,4$), $X^{i}$ is distributed as a Brownian bridge in $\mathcal{B}(S,\delta-2^{-n-1/2})$ (resp.\ $\mathbb D\setminus S$) conditioned to hit $S$ (resp.\ $\partial\mathcal{B}(S,\delta-2^{-n-1/2})$).
\end{lemma}
Note that for technical reasons (see Remark~\ref{rm:gb}), in Proposition \ref{cgs} we have conditioned on $S\in\mathfrak{S}_{n}(\delta+2^{-\sqrt{n}})$. Therefore, we need the following lemma to show that a $\delta$-good box is also a $(\delta+2^{-\sqrt{n}})$-good box with high probability as $n\to\infty$.
\begin{lemma}\label{delta+ good box}
There exists $c=c(\iota,\delta)>0$ such that for any $n\ge N(\delta)$, and any $n$-box $S\subset\mathcal{A}(0,\iota,1-\iota)$, we have
\[\mathbb{P}(S\in\mathfrak{S}_{n}(\delta+2^{-\sqrt{n}})|S\in\mathfrak{S}_{n}(\delta))\geq1-c2^{-\sqrt{n}}.\]
\end{lemma}
\begin{proof}
Let $\varepsilon:=2^{-\sqrt{n}}$. Thanks to Lemma \ref{fme}, it suffices to show that
\[\mathbb{P}(S\in\mathfrak{S}_{n}(\delta),S\notin\mathfrak{S}_{n}(\delta+\varepsilon))\leq c2^{-2n}\cdot2^{-\sqrt{n}}.\]
We use the notation introduced in \eqref{excursion and bridge} and write $x$ for the centre of $S$. The event $\{S\in\mathfrak{S}_{n}(\delta)\}\cap\{S\notin\mathfrak{S}_{n}(\delta+\varepsilon)\}$ implies the following events:
\par
($H^{1}$)\quad $V_{\delta}^{\mathrm{PTP}}(S)$ occurs, and $X^{1},X^{3}$ do not disconnect $S$ from infinity.
\par
($H^{2}$)\quad The union of $W^{i}[\sigma_{i},\tau_{i}](i=1,\ldots,5)$ does not disconnect $\mathcal{D}_{-n}(x)$ from infinity, where $\sigma_{i}$ (resp.\ $\tau_{i}$) is the last hitting time of $\mathcal{C}_{-n+1}(S)$ (resp.\ first hitting time of $\partial\mathcal{B}(x,\delta/4)$) with respect to $W^{i}$.
\par
($H^{3}$)\quad At least one of the bridges $X^{2},X^{4}$ does not intersect $\mathcal{B}(x,\delta+\varepsilon-2^{-n-1/2})$.
\par
By Lemma~\ref{ac excursions} (2), on $V_{\delta}^{\mathrm{PTP}}(S)$ and conditioned on $X^{1},\ldots,X^{5}$, the joint law of $W^{i},\, i=1,\ldots,5$ is uniformly equivalent to that of five independent Brownian excursions with fixed starting and ending points. Therefore, the law of $W^{i}[\sigma_{i},\tau_{i}], i=1,\ldots,5$ is uniformly equivalent to that of five independent Brownian excursions, and we have by Lemma~\ref{gnd prob}
\begin{equation}\label{delta+1}
\mathbb{P}(H^{2}|H^{1})\asymp n^{5}2^{-2n}.
\end{equation}
By standard estimates on the hitting probability of Brownian motion and \cite[Lemma 2.13]{GLQ22}, we have
\begin{equation}\label{delta+2}
\mathbb{P}(H^{1})\leq cn^{-3}\cdot(n^{-1})^{2}=cn^{-5}.
\end{equation}
By Lemma \ref{ac excursions} (3), conditioned on $H^{1}$ and excursions $W^{1},\ldots,W^{5}$, bridges $X^{2},X^{4}$ are distributed as independent Brownian bridges in $\mathbb D\setminus S$, conditioned to hit $\partial\mathcal{B}(x,\delta-2^{-n-1/2})$. Using this description, we will show that
\begin{equation}\label{delta+7}
\mathbb{P}(H^{3}|H^{1},H^{2})\leq c'2^{-\sqrt{n}}.
\end{equation}
Assuming \eqref{delta+7}, and combining with \eqref{delta+1}, \eqref{delta+2}, we conclude that
\[\mathbb{P}(H^{1}\cap H^{2}\cap H^{3})=\mathbb{P}(H^{3}|H^{1},H^{2})\mathbb{P}(H^{2}|H^{1})\mathbb{P}(H^{1})\leq c''2^{-2n}\cdot 2^{-\sqrt{n}},\]
which completes the proof subject to the proof of \eqref{delta+7}.
\par
It remains to prove \eqref{delta+7}. We now assume that $X$ is a Brownian bridge in $\mathbb{D}\backslash S$ conditioned to hit $\partial\mathcal{B}(x,\delta-2^{-n-1/2})$ with end points $u,v\in\partial\mathcal{B}(x,\delta/2)$. Consider the part of $X$ from its first hitting of $\partial\mathcal{B}(x,\delta-2^{-n-1/2})$, denoted by $z$, to its first return of $\partial\mathcal{B}(x,3\delta/4)$. The law of this part is uniformly equivalent to a Brownian motion started from $z$ and stopped when it hits $\partial\mathcal{B}(x,3\delta/4)$. Hence, it avoids $\partial\mathcal{B}(x,\delta+\varepsilon-2^{-n-1/2})$ with probability of order $\varepsilon/\delta$ by Gambler's ruin estimate. Therefore, we have
\[\mathbb{P}_{u,v}^{\mathrm{bridge}}(X\ \mathrm{never\ exits}\ \mathcal{B}(x,\delta+\varepsilon-2^{-n-1/2}))\asymp\varepsilon/\delta.\]
Recall the description of $X^{2},X^{4}$ in Lemma \ref{ac excursions}. For $i=2,4$, we have
\[\mathbb{P}(\mbox{$X^{i}$ never exits $\mathcal{B}(x,\delta+\varepsilon-2^{-n-1/2}$})|\mbox{$u,v$ are the end points of $X^{i}$})\asymp\varepsilon/\delta.\]
Hence, we obtain \eqref{delta+7} as desired by taking expectation with respect to endpoints of $X^i$, and complete the proof.
\end{proof}
\subsection{Controlling the Probability of Bad Events}\label{subsec:bad}
As mentioned right after Proposition \ref{cgs}, we need to exclude a series of bad events that occur with very small probability, which is the main purpose of this subsection. Throughout this section, we set 
\begin{equation}\label{l-value}
L:=[\sqrt{n}].
\end{equation}
\par
By Lemma \ref{ac excursions} (1), excursions $W^{1},\ldots,W^{5}$ are uniformly equivalent to five independent Brownian excursions. We first prove a lemma for independent Brownian excursions.
\par
Recall the definition of openings in Section \ref{subsec:notation}. Our next lemma (together with Remark~\ref{only one opening}) shows that 5 non-disconnecting Brownian excursions admit a unique global opening in an annulus with high probability. Indeed, we will prove a slightly stronger result, showing that with high probability, multiple openings can only survive a few scales.
\par
Let $W_{i}$, $i=1,2,\ldots,5$, be five independent Brownian excursions in $\mathcal{A}(0,1,2^n)$ from $\mathcal{C}_{0}$ to $\mathcal{C}_{n}$. Let $E_{1}$ be the event that the union of $W_{i}$, $i=1,2,\ldots,5$ does not disconnect $\mathcal{C}_{0}$ from infinity. Let $\tau_{i}^{j}$ (resp.\ $\sigma_{i}^{j}$) be the first (resp.\ last) hitting time of $\mathcal{C}_{j}$ with respect to $W_{i}$. For any $j\leq n/2$, let $E_{2}=E_{2}(j)$ (resp.\ $E_{3}=E_{3}(j)$) be the event that there exists nonempty $A\subsetneq\{1,2,3,4,5\}$, such that the packets of Brownian excursions $\{W_{i}[0,\tau_{i}^{j}]:i\in A\}$ and $\{W_{i}[0,\tau_{i}^{j}]:i\in A^{c}\}$ do not intersect (resp.\ $\{W_{i}[\sigma_{i}^{n-j},\tau_{i}^{n}]:i\in A\}$ and $\{W_{i}[\sigma_{i}^{n-j},\tau_{i}^{n}]:i\in A^{c}\}$ do not intersect).
\begin{lemma}\label{unique}
 There exist $c$, such that for $m=2,3$ and $n\ge N(\delta)$,
\[\mathbb{P}(E_{m}(j)\cap E_{1})\leq cn^{10}2^{-j-2n}.\]
\end{lemma}
\begin{proof}
We give a proof for the case of $m=2$, and the proof for $m=3$ is the same by inversion and scaling invariance of Brownian excursions.
\par
The event $E_{2}\cap E_{1}$ implies the following two events:
\par
($F^{1}$) The union of $W_{i}[\sigma_{i}^{j+1},\tau_{i}^{n}](i=1,2,\ldots,5)$ does not disconnect $\mathcal{D}_{j+1}$ from infinity.
\par
($F^{2}$) There exists non-empty $A\subsetneq\{1,2,3,4,5\}$, such that \{$W_{i}[0,\tau_{i}^{j}], i\in A\}$ does not intersect with $\{W_{i}[0,\tau_{i}^{j}], i\in A^{c}\}$.
\par
By the strong Markov property, Lemma \ref{last exit} and \cite[Lemma 2.15]{GLQ22}, conditioned on $W_{i}[0,\tau_{i}^{j}]$, $i=1,2,\ldots,5$, the law of $W_{i}[\sigma_{i}^{j+1},\tau_{i}^{n}]$, $i=1,2,\ldots,5$, is uniformly equivalent to that of five independent Brownian excursions from $\mathcal{C}_{j+1}$ to $\mathcal{C}_{n}$. By Lemma \ref{gnd prob}, it therefore follows that 
\[\mathbb{P}(F^{1}|F^{2})\leq c(n-j-1)^{5}2^{-2(n-j)}.\]
Now, for all four cases $|A|=i$, $i=1,2,3,4$, in the annulus $\mathcal{A}(0,1,2^{j})$, $F^{2}$ implies the non-intersection event of $i$ Brownian excursions and another $(5-i)$ Brownian excursions. Since $\xi(1,4)>3$ and $\xi(2,3)>3$, by Lemma \ref{ni prob}, we have 
\[\mathbb{P}(F^{2})\leq cj^{5}2^{-\xi(i,5-i)j}\leq c'j^{5}2^{-3j}.\]
We thus conclude that 
\[\mathbb{P}(F^{1}\cap F^{2})\leq c''j^{5}(n-j-1)^{5}2^{-j-2n},\]
which completes the proof.
\end{proof}
We record the following geometrical observation.
\begin{lemma}\label{unique connected component}
On $E_{2}^{c}$, the opening $\mathcal{O}(U,\mathcal{A}(0,1,2^{j}))$ contains at most one connected component, where $U$ is the union of the trajectories of $W_{i}, i=1,2,\dots,5$. Similarly, on $E_{3}^{c}$, the opening $\mathcal{O}(U,\mathcal{A}(0,2^{n-j},2^{n}))$ contains at most one connected component.    
\end{lemma}
\begin{figure}[t]
\centering
\includegraphics[width=7cm]{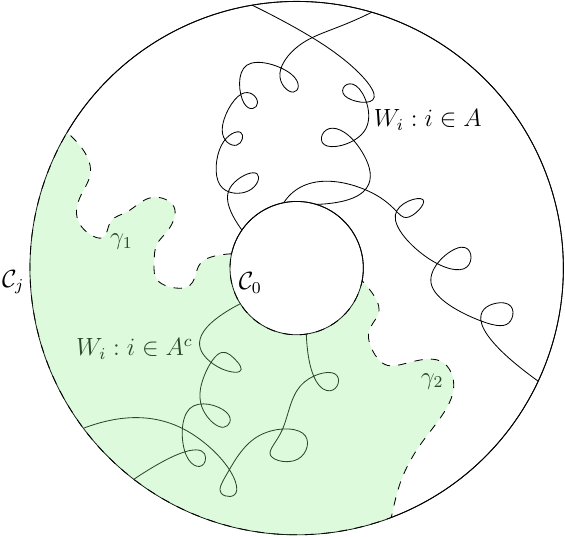}
\caption{\small Dashed curves are $\gamma_{1}$ and $\gamma_{2}$. $O_{1}$ is in green. There exists a non-empty $A\subsetneq\{1,\ldots,5\}$, such that $\{W_{i}:i\in A\}$ and $\{W_{i}:i\in A^{c}\}$ do not intersect before hitting $\mathcal{C}_{j}$.}
\label{fig_unique}
\end{figure}
\begin{proof}
It suffices to prove the first statement in the lemma. Suppose on the contrary that, on $E_{2}^{c}$, the opening $\mathcal{O}(U,\mathcal{A}(0,1,2^{j}))$ contains at least two connected components. Let $\gamma_{1},\gamma_{2}$ be two crossings that belong to different components (i.e.\ $\gamma_{i}$ is a continuous simple curve in $\mathcal{A}(0,1,2^{j})$ from $\mathcal{C}_{0}$ to $\mathcal{C}_{j}$ avoiding $U$ for $i=1,2$). Let $O_{1},O_{2}$ be the two disjoint domains surrounded by $\gamma_{1},\gamma_{2},\mathcal{C}_{0},\mathcal{C}_{j}$ (see Figure~\ref{fig_unique}). If $U\cap O_{1}=\varnothing$ or $U\cap O_{2}=\varnothing$, then $\gamma_{1}$ and $\gamma_{2}$ belong to the same connected component, which yields a contradiction. 
Now assume $U\cap O_{1}\neq\varnothing$ and $U\cap O_{2}\neq\varnothing$. 
Note that for any $i=1,\ldots,5$, if $W_{i}\cap O_{1}\neq\varnothing$ and $W_{i}\cap O_{2}\neq\varnothing$, then $W_{i}$ must hit either $\gamma_{1}$ or $\gamma_{2}$. Therefore, for each $i$, $W_{i}[0,\tau_{i}^{j}]$ belongs to either $O_{1}$ or $O_{2}$. Let $A:=\{i:W_{i}[0,\tau_{i}^{j}]\in O_{1}\}$, then $A,A^{c}\subsetneq\{1,2,3,4,5\}$, and $\{W_{i}[0,\tau_{i}^{j}]:i\in A\}$ and $\{W_{i}[0,\tau_{i}^{j}]:i\in A^{c}\}$ do not intersect, which contradicts with $E_{2}^{c}$, and concludes the proof.
\end{proof}
\begin{rmk}\label{only one opening}
Note that, on the event $E_{1}$, there is at least one connected component in $\mathcal{O}(U,\mathcal{A}(0,1,2^{n}))$. Hence, by Lemma \ref{unique connected component}, on the event $(E_{2}^{c}\cap E_{3}^{c})\cap E_{1}$, the opening $\mathcal{O}(U,\mathcal{A}(0,1,2^{n}))$ contains a unique connected component, and so do $\mathcal{O}(U,\mathcal{A}(0,1,2^{j}))$ and $\mathcal{O}(U,\mathcal{A}(0,2^{n-j},2^{n}))$. This observation will be useful in the proof of Lemma \ref{adapted}.
\end{rmk}
The next lemma shows that, the probability that five non-disconnecting Brownian excursions contain extra crossings through multiple scales decays exponentially with respect to the number of scales.
\begin{lemma}\label{local}
For $i=1,2,\ldots,5$, let $E_{i,j}$ be the event that $W_{i}$ returns to $\mathcal{D}_{jL}$ after hitting $\mathcal{C}_{(j+1)L}$. Then there exist $c,c'>0$, such that for any integer $j\leq [3L/4]$, and $i=1,2,\ldots,5$, we have
\[\mathbb{P}(E_{i,j}\cap E_{1})\leq ce^{-c'\sqrt{n}}\cdot2^{-2n}.\]
Let $E_{4}$ be the union of $(E_{i,j})_{1\le i\le 5,1\le j\le [3L/4]}$. Then by the union bound there exists $c''>0$ such that
\[\mathbb{P}(E_{4}\cap E_{1})\leq ce^{-c''\sqrt{n}}\cdot2^{-2n}.\]
\end{lemma}
\begin{proof}
Without loss of generality, we consider the case when $i=1$. Let $\tau_{i}(A)$ (resp.\ $\sigma_{i}(A)$) be the first (resp.\ last) hitting time of $A$ with respect to $W_{i}$, and we use $\tau_{i}$ to denote $\tau_{i}(\mathcal{C}_{n})$. On the event $E_{1,j}$, let
\begin{equation}\label{461}
t_{1}^{i}:=\tau_{i}(\mathcal{C}_{jL+1/3}),\ t_{2}^{i}:=\tau_{i}(\mathcal{C}_{(j+1)L-1/3}),
\end{equation}
and
\begin{equation}\label{decomposing extra crossings}
\begin{split}
&t_{3}^{1}:=\tau_{1}(\mathcal{C}_{(j+1)L},\mathcal{D}_{(j+1)L-1/3}),\quad t_{4}^{1}:=\tau_{1}(\mathcal{C}_{(j+1)L},\mathcal{D}_{jL+1/3})\\
&t_{5}^{1}:=\tau_{1}(\mathcal{C}_{(j+1)L},\mathcal{D}_{jL},\mathcal{C}_{jL+1/3}),\quad t_{6}^{1}:=\tau_{1}(\mathcal{C}_{(j+1)L},\mathcal{D}_{jL},\mathcal{C}_{(j+1)L-1/3});
\end{split}
\end{equation}
see Figure \ref{fig3} for an example of the decomposition above.
\begin{figure}[b]
\centering
\includegraphics[width=8cm]{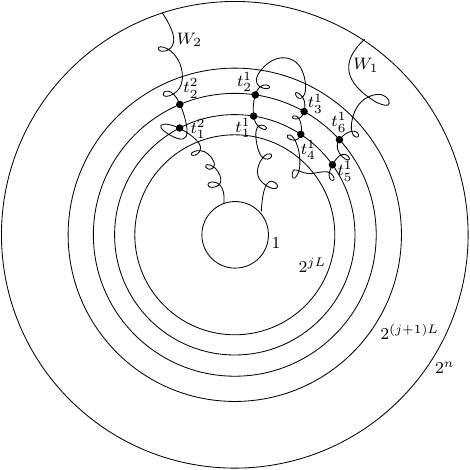}
\caption{\small An example of the decomposition on the event $E_{1,j}$. We omit $W_{3},W_{4},W_{5}$ in the figure.}
\label{fig3}
\end{figure}

Let 
\begin{equation}\label{463}
t'_{i}:=\tau_{i}(\mathcal{C}_{jL}),\quad t''_{i}:=\sigma_{i}(\mathcal{C}_{(j+1)L}),\quad i=1,2,\ldots,5.\end{equation}
Then $E_{1,j}\cap E_{1}$ implies the following three events:
\par
($F^{3}$) The union of $W_{i}[0,t'_{i}]$, $i=1,\ldots,5$, does not disconnect $\mathcal{D}_{0}$ from infinity,
\par
($F^{4}$) The union of $W_{i}[t''_{i},\tau_{i}]$, $i=1,\ldots,5$, does not disconnect $\mathcal{C}_{(j+1)L}$ from infinity.
\par
($F^{5}$) The union of $W_{i}(t)$, $t\in[t_{1}^{i},t_{2}^{i}]$, $i=1,\ldots,5$, $W_{1}(t)$, $t\in[t_{3}^{1},t_{4}^{1}]$, and $W_{1}(t)$, $t\in[t_{5}^{1},t_{6}^{1}]$ does not disconnect $\mathcal{C}_{jL+1/3}$ from infinity.

Thanks to Lemma \ref{gnd prob}, we have 
\[\mathbb{P}(F^{3})\leq c_{1}(jL)^{5}2^{-\xi(5)jL}.\]
By the strong Markov property and Lemma~\ref{lem:harnack}, conditioned on $W_{i}[0,t'_{i}]$, the joint law of $W_{i}[t_{1}^{i},t_{2}^{i}]$, $i=1,\ldots,5$, and $W_{1}[t_{3}^{1},t_{4}^{1}],W_{1}[t_{5}^{1},t_{6}^{1}]$ is uniformly equivalent to the joint law of seven independent Brownian motions, started uniformly from either $\mathcal{C}_{jL+1/3}$ or $\mathcal{C}_{(j+1)L-1/3}$ and stopped upon hitting the opposite boundary of $\mathcal{A}(0,2^{jL+1/3},2^{(j+1)L-1/3})$. Then by Lemma \ref{sep PTP} we have
\[\mathbb{P}(F^{5}|F^{3})\leq c_{2}2^{-\xi(7)L}.\]
By the strong Markov property, Lemmas~\ref{lem:harnack} and \ref{last exit}, conditioned on $W_{i}[0,t_{2}^{i}]$, $i=2,3,4,5$, and $W_{1}[0,t_{6}^{1}]$, the joint law of $W_{i}[t''_{i},\tau_{i}]$, $i=1,\ldots,5$, is uniformly equivalent to that of five independent Brownian excursions from $\mathcal{C}_{(j+1)L}$ to $\mathcal{C}_{n}$. Then by Lemma \ref{gnd prob}, we have
\[\mathbb{P}(F^{4}|F^{3},F^{5})\leq c_{3}(n-(j+1)L)^{5}2^{-\xi(5)(n-(j+1)L)}.\]
Therefore we conclude that
\[\mathbb{P}(E_{1,j}\cap E_{1})\leq\mathbb{P}(F^{3}\cap F^{4}\cap F^{5})\leq cn^{10}2^{-(\xi(7)-\xi(5))L}\cdot2^{-\xi(5)n}\leq cn^{10}2^{-\frac{11}{12}L}\cdot2^{-2n},\]
which completes the proof.
\end{proof}
As applications of Lemmas \ref{unique} and \ref{local}, the following two lemmas control the probability of ``bad events'' $B_{1},B_{2},B_{3}$ defined below with respect to $W^{i}$, ``excursions'' decomposed from a Brownian path in \eqref{excursion and bridge}. From now on, we write $\mathfrak{S}_{n}:=\mathfrak{S}_{n}(\delta)$ for brevity of notation.
\par
Recall that we let 
\begin{align}\label{delta-value}
K:=-\log_2(\delta/2)
\end{align}
be an integer. Later as we take $\delta\to 0$ in this section, we are letting $K\to \infty$ in $\delta=2^{-K+1}$.

Let $B_{1}$ (resp.\ $B_{2}$) be the event that there exists nonempty $A\subsetneq\{1,2,3,4,5\}$, such that $\{W^{i}:i\in A\}$ and $\{W^{i}:i\in A^{c}\}$ do not intersect before the first hitting of $\mathcal{C}_{-n+L}(S)$ (resp.\ after the last hitting of $\mathcal{C}_{-K-L}(S)$).
\begin{lemma}\label{unique opening}
There exist $c,c'>0$, such that
\begin{equation}\label{471}\mathbb{P}(B_{1}\cup B_{2}|S\in\mathfrak{S}_{n})\leq ce^{-c'\sqrt{n}}.\end{equation}
\end{lemma}
\begin{proof}
Note that $S\in\mathfrak{S}_{n}$ implies that the union of $W^{1},\ldots,W^{5}$ does not disconnect $\mathcal{C}_{-n}(S)$ from infinity. By Lemma \ref{fme} for PTP, we have
\begin{equation}\label{fme delta+}
\mathbb{P}(S\in\mathfrak{S}_{n})\asymp 2^{-2n}.
\end{equation}
By Lemmas \ref{ac excursions} and \ref{gnd prob}, we have
\begin{equation}\label{nd excursion}
\mathbb{P}(W^{1},\ldots,W^{5}\ \mathrm{do\ not\ disconnect}\ \mathcal{C}_{-n}(S)\ \mathrm{from\ infinity})\asymp n^{5}2^{-2n}.
\end{equation}
Thanks to Lemmas \ref{unique} and \ref{ac excursions}, we conclude that
\begin{equation}\label{bad12}
\begin{split}
&\quad\ \mathbb{P}(B_{1}\cup B_{2},\ S\in\mathfrak{S}_{n})\\
&\leq\mathbb{P}(B_{1}\cup B_{2},\ W^{1},\ldots,W^{5}\ \mathrm{do\ not\ disconnect}\ \mathcal{C}_{-n}(S)\ \mathrm{from\ infinity})\\
&\leq c\mathbb{P}((\widetilde{E}_{2}\cup\widetilde{E}_{3})\cap\widetilde{E}_{1})\leq c'n^{10}2^{-L}2^{-2n},
\end{split}
\end{equation}
where $\widetilde{E}_{1},\widetilde{E}_{2},\widetilde{E}_{3}$ are events $E_{1},E_{2}(L),E_{3}(L)$ from Lemma \ref{unique} under suitable translation and rescaling (which maps $\mathcal{A}(0,1,2^{n-K})$ to $\mathcal{A}(S,2^{-n},2^{-K})$). Combining \eqref{fme delta+} and \eqref{bad12}, we obtain \eqref{471} as desired.
\end{proof}
For $i=1,2,\ldots,5$, let $E_{i,j}$ be the event that $W^{i}$ returns to $\mathcal{D}_{-n+jL}(S)$ after the first hitting of $\mathcal{C}_{-n+(j+1)L}(S)$. Let 
\[B_{3}:=\bigcup_{i=1}^{5}\bigcup_{j=1}^{[3L/4]}E_{i,j}.\]
\begin{lemma}\label{no extra crossing}
There exist $c,c'>0$, such that
\[\mathbb{P}(B_{3}|S\in\mathfrak{S}_{n})\leq ce^{-c'\sqrt{n}}\cdot2^{-2n}.\]
\end{lemma}
\begin{proof}
Thanks to Lemmas \ref{local} and \ref{ac excursions}, we have:
\begin{equation}\label{bad3}
\begin{split}
&\quad\ \mathbb{P}(B_{3},\ S\in\mathfrak{S}_{n})\\
&\leq\mathbb{P}(B_{3},\ W^{1},\ldots,W^{5}\ \mathrm{do\ not\ disconnect}\ \mathcal{C}_{-n}(S)\ \mathrm{from\ infinity})\\
&\leq c\mathbb{P}(\widetilde{E}_{4}\cap\widetilde{E}_{1})\leq ce^{-c'\sqrt{n}}2^{-2n},
\end{split}
\end{equation}
where $\widetilde{E}_{1},\widetilde{E}_{4}$ are events $E_{1},E_{4}$ from Lemma \ref{local} under suitable translation and rescaling (which maps $\mathcal{A}(0,1,2^{n-K})$ to $\mathcal{A}(S,2^{-n},2^{-K})$). Combining \eqref{fme delta+} and \eqref{bad3}, we complete the proof.
\end{proof}
Our next lemma controls the probability that ``bridges'' $X^{i}$'s cross multiple scales. We are going to show that such probability decays exponentially with respect to the number of scales.
\par
For $i=0,2,4$, let $E^{i}$ be the event that $X^{i}$ intersects $\mathcal{C}_{-K-L}(S)$, and for $i=1,3,5$, let $E^{i}$ be the event that $X^{i}$ intersects $\mathcal{C}_{-n+L}(S)$. Let
\[B_{4}:=\bigcup_{i=0}^{5}E^{i}.\]
\begin{lemma}\label{local bridge}
There exist $c,c'>0$, such that for any $i=0,\ldots,5$, we have:
\[\mathbb{P}(S\in\mathfrak{S}_{n},E^{i})\leq ce^{-c'\sqrt{n}}2^{-2n}.\]
Consequently, there exists $c''>0$, such that
\[\mathbb{P}(B_{4}|S\in\mathfrak{S}_{n})\leq c''e^{-c'\sqrt{n}}.\]
\end{lemma}
\begin{proof}
We prove the lemma for $i=1$. Proofs for other cases are similar. On event $\{S\in\mathfrak{S}_{n}\}\cap E^{1}$, the following events occur:
\par
($F^{6}$)\quad The union of $W^{1},\ldots,W^{5}$ does not disconnect $\mathcal{C}_{-n}(S)$ from infinity.
\par
($F^{7}$)\quad $X^{1}$ intersects $\mathcal{C}_{-n+L}(S)$, and $X^{1}$ does not disconnect $\mathcal{C}_{-n}(S)$ from infinity.
\par
It has been shown in \eqref{nd excursion} that 
\[\mathbb{P}(F^{6})\asymp n^{5}2^{-2n}.\]
Thanks to Lemma \ref{ac excursions}, conditioned on the starting and ending points of $W^{1},\ldots,W^{5}$, $X^{1}$ is distributed as a Brownian bridge from $W(t_{1})$ to $W(s_{2})$ (recall \eqref{decomposition} for the definition of $t_{1}$ and $s_{2}$) conditioned on hitting $S$. On the event $F^{7}$, $X^{1}$ intersects $\mathcal{C}_{-n+L}(S)$. Note that the law of $X^{1}$ between the first hitting of $\mathcal{C}_{-n+1}(S)$ and the first hitting of $\mathcal{C}_{-n+L-1}(S)$ is uniformly equivalent to a Brownian motion started uniformly from $\mathcal{C}_{-n+1}(S)$ and stopped upon the first hitting of $\mathcal{C}_{-n+L-1}(S)$. By Lemma \ref{sep PTP}, we thus get
\[\mathbb{P}(F^{7}|W^{1},\ldots,W^{5})\leq c2^{-L\xi(1)}=c2^{-L/4}.\]
We then have
\[\mathbb{P}(F^{6}\cap F^{7})=\mathbb{P}(F^{6})\mathbb{P}(F^{7}|F^{6})\leq cn^{5}2^{-2n}2^{-L/4}.\]
Recalling that $L=[\sqrt{n}]$, we complete the proof.
\end{proof}
For $W^{i}$, $i=1,\ldots,5$, let $\tau_{i,j}$ denote the first hitting time of $\mathcal{C}_{-n+j}(S)$ with respect to $W^{i}$, and let $\sigma_{i,j}$ denote the last hitting time of $\mathcal{C}_{-n+j}(S)$ with respect to $W^{i}$. Recall the convention \eqref{delta-value} that we assume $\delta/2=2^{-K}$ for some integer $K$. We then define an increasing family of $\sigma$-fields
\begin{equation}\label{filtration}
\mathscr{F}_{m}:=\sigma\left(X^{0},\;X^i,\;W^{i}[0,\tau_{i,mL}],\;W^{i}[\sigma_{i,n-K-L},\tau_{i,n-K}],\;i=1,\ldots,5\right).
\end{equation}
Recall the event $V_{\delta}^{\mathrm{PTP}}(S)$ from Definition~\ref{gs}, and note that $\mathscr{F}_{m}$ is well-defined on $V_{\delta}^{\mathrm{PTP}}(S)$. Denoting by
\begin{equation}\label{eq:U_m}
U_{m}:=\left\{X^{0},\;X^i,\;W^{i}[0,\tau_{i,m}],\;W^{i}[\sigma_{i,n-K-L},\tau_{i,n-K}],\;i=1,\ldots,5\right\}    
\end{equation}
the collection of curves in consideration, the following lemma shows that the opening on each scale are expected to be not very bad, characterized by the non-disconnection probability up to the next scale. 
\par
For any $j\leq[3L/4]$, let $\widetilde{E}_{j}$ be the event that $U_{jL+1}$ does not disconnect $\mathcal{C}_{-n}(S)$ from infinity. Then, on the event $V_{\delta}^{\mathrm{PTP}}(S)$, the event $\mathbb{P}(\widetilde{E}_{j}|U_{jL})$ is a random variable in $\mathscr{F}_{j}$. Let
\begin{equation}\label{eq:barEj}
\overline{E}_{j}=\left\{\mathbb{P}(\widetilde{E}_{j}|V_{\delta}^{\mathrm{PTP}}(S),U_{jL})\leq n^{-10}\right\}, \quad B_{5}:=\bigcup_{j=1}^{[3L/4]}\overline{E}_{j}.
\end{equation}
\begin{lemma}\label{all good opening}
There exist $c$, such that for any integer $1\leq j\leq [3L/4]$, we have
\[\mathbb{P}(\overline{E}_{j},S\in\mathfrak{S}_{n})\leq cn^{-2}2^{-2n}.\]
Consequently, we have 
\[\mathbb{P}(B_{5}|S\in\mathfrak{S}_{n})\leq cn^{-1}.\]
\end{lemma}
\begin{proof}
By considering different parts of excursions, we see that $S\in\mathfrak{S}_{n}$ implies the following events:
\par
($F^{8}$)\quad The union of $W^{i}[0,\tau_{i,jL}]$, $i=1,\ldots,5$, does not disconnect $\mathcal{C}_{-n}(S)$ from infinity.
\par
($F^{9}$)\quad The union of $W^{i}[0,\tau_{i,jL+1}]$, $i=1,\ldots,5$, does not disconnect $\mathcal{C}_{-n}(S)$ from infinity.
\par
($F^{10}$)\quad The union of $W^{i}[\tau_{i,jL+2},\tau_{i,n-K-L-1}]$, $i=1,\ldots,5$, does not disconnect $\mathcal{C}_{-n}(S)$ from infinity.
\par
($F^{11}$)\quad The union of $W^{i}[\sigma_{i,n-K-L},\tau_{i,n-K}]$, $i=1,\ldots,5$, does not disconnect $\mathcal{C}_{-n}(S)$ from infinity.
\par
By the strong Markov property, Lemmas \ref{ac excursions}, \ref{gnd prob} and \ref{last exit}, we have
\begin{equation}\label{f8+f11}
\mathbb{P}(F^{8}\cap F^{11})\asymp(jL)^{5}L^{5}2^{-2(j+1)L}\leq c_{1}n^{8}2^{-2(j+1)L}.
\end{equation}
By the definition of $\overline{E}_{j}$ and the fact that $F_{8},F_{11}$ are $\mathscr{F}_{j}$-measurable, we have
\begin{equation}\label{f9}
\mathbb{P}(\overline{E}_{j}\cap F^{9}|F^{8},F^{11})\leq n^{-10}.
\end{equation}
Then, by Lemma~\ref{lem:harnack} and the strong Markov property, conditioned on $U_{jL+1}$, the joint law of \\$W^{i}[\tau_{i,jL+2},\tau_{i,n-K-L-1}]$, $i=1,\ldots,5$, is uniformly equivalent to the joint law of five independent Brownian motions starting uniformly from $\mathcal{C}_{-n+jL+2}(S)$ and stopped upon reaching $\mathcal{C}_{-K-L-1}(S)$ conditioned not to hit $\mathcal{C}_{-n}(S)$. By Lemma~\ref{sep PTP}, it therefore follows that
\begin{equation}\label{f10}
\mathbb{P}(F^{10}|F^{8},F^{9},F^{11},\overline{E}_{j})\leq c_{2}2^{-2(n-K-(j+1)L)}.
\end{equation}
Combining \eqref{f8+f11}, \eqref{f9} and \eqref{f10}, we complete the proof of the lemma.
\end{proof}
By Lemma \ref{delta+ good box}, we can show that Lemmas~\ref{unique opening}--\ref{all good opening} still hold when we replace $\mathfrak{S}_{n}=\mathfrak{S}_{n}(\delta)$ by $\mathfrak{S}_{n}(\delta+2^{-\sqrt{n}})$.
\par
For $S\in\mathfrak{S}_{n}(\delta+2^{-\sqrt{n}})$ and any integer $1\leq i\leq[L/2]$, we write
\[\mathcal{A}_{i}:=\mathcal{A}(S,2^{-n+iL},2^{-n+(i+1)L})\]
for short and use $N_{i}$ to denote the number of $\delta$-good boxes in $\mathcal{A}_{i}$. For $T\in\mathfrak{S}_{n}(\delta)$, if 
\[\tau(T,\mathcal{C}_{-K}(S),T,\mathcal{C}_{-K}(S),T)<\tau(S,\mathcal{C}_{-K}(S),S,\mathcal{C}_{-K}(S),S),\]
we say $T$ is globally visited three times before the third global visit of $S$, which we denote by $T\lessdot S$. Recall that we use $T\prec\mathcal{A}_{i}$ to denote the event $\mathcal{D}_{-n}(T)\subset \mathcal{A}_{i}$  (see Section~\ref{subsec:notation}). Let $\widetilde{N}_{i}$ stand for the number of $\delta$-good boxes $T$ such that $T\prec\mathcal{A}_{i}$ and $T\lessdot S$.
\par
Our next lemma shows that $\widetilde{N}_{i}$ is ``almost'' adapted to $\mathscr{F}_{i}$ under the probability measure $\mathbb{P}(\cdot|S\in\mathfrak{S}_{n})$. More precisely, we have:
\begin{lemma}\label{adapted}
For any integer $1\leq i\leq L/2$, on the event $\{S\in\mathfrak{S}_{n}(\delta+2^{-\sqrt{n}})\}\cap(\cap_{j=1}^{5}B_{j}^{c})$, whether a box $T\prec\mathcal{A}_{i}$ is $\delta$-good and $T\lessdot S$ can be fully determined by the configuration $U_{(i+3)L}$ (see \eqref{eq:U_m}).
\end{lemma}
In other words, $\widetilde{N}_{i}$ is $\mathscr{F}_{(i+3)}$-measurable under a slightly tilted probability measure
\begin{equation}\label{conditional measure}
\overline{\mathbb{P}}_{n}(\cdot):=\mathbb{P}(\cdot|S\in\mathfrak{S}_{n}(\delta+2^{-\sqrt{n}}),\cap_{j=1}^{5}B_{j}^{c}).
\end{equation}
\begin{proof}
On the event $S\in\mathfrak{S}_{n}(\delta+2^{-\sqrt{n}})$, the event $V_{\delta}^{\mathrm{PTP}}(S)$ occurs, so $U_{iL}$ and $\mathscr{F}_{i}$ are well-defined. 
\par
On $B_{1}^{c}\cap B_{2}^{c}$, openings $\mathcal{O}\left(U_{L},\mathcal{A}(S,2^{-n},2^{-n+L})\right)$ and $\mathcal{O}\left(U_{L},\mathcal{A}(S,2^{-K-L},2^{-K})\right)$ both contain only one connected component. Therefore $\mathcal{O}\left(W[0,\mathcal{T}],\mathcal{A}(S,2^{-n},2^{-K})\right)$ has a unique connected component, where $\mathcal{T}:=\mathcal{T}(S)$ is the third global visit of $S$ (see \eqref{eq:Tc}). On the event $B_{3}^{c}$, once we observe $U_{(i+3)L}$, we immediately know the full configuration of $W^{i}$ in $\mathcal{D}_{-n+(i+2)L}(S)$ since $W^{i}$ cannot return to $\mathcal{C}_{-n+(i+2)L}(S)$ after hitting $\mathcal{C}_{-n+(i+3)L}(S)$. On the event $B_{4}^{c}$, bridges $X^{i}$ are in either $\mathcal{D}_{-n+L}(S)$ or $(\mathcal{D}_{-K-L}(S))^{c}$, which will not affect the opening in $\mathcal{A}(S,2^{-n+iL},2^{-n+(i+3)L})$.
\par
For a typical $n$-box $T\prec\mathcal{A}_{i}$, we now determine whether $T\in\mathfrak{S}_{n}(\delta)$ and $T\lessdot S$ by the following procedure. First, whether $T$ is visited by three different $W^{i}$ is obviously $\mathscr{F}_{(i+3)}$-measurable. Second, we need to check whether $\mathcal{C}_{-n}(T)$ is not disconnected from infinity by the path until the third global visit of $T$. By the previous paragraph, we know that under $\overline{\mathbb{P}}_{n}$, the following cut set (see Section~\ref{subsec:notation})
\[\mathcal{O}_{i+2}:=\mathcal{O}\left(W[0,\mathcal{T}],\mathcal{A}(S,2^{-n},2^{-K}),2^{-n+(i+2)L}\right)\]
is $\mathscr{F}_{(i+3)}$-measurable, and it is a continuous arc on $\mathcal{C}_{-n+(i+2)L}(S)$. Thanks to the event $B_{3}^{c}\cap B_{4}^{c}$, the remaining part of the curve $W[\mathcal{T}(T),\mathcal{T}]$ is fully contained in $\mathcal{D}_{-n+(i+2)L}(S)$. Therefore, we have
\[\mathcal{O}_{i+2}=\mathcal{O}(W[0,\mathcal{T}(T)],\mathcal{A}(S,2^{-n},2^{-K}),2^{-n+(i+2)L}).\]
Note that $\mathrm{dist}(S,T)\leq 2^{-n+(i+1)L}<2^{-\sqrt{n}}$. On the event $\{S\in\mathfrak{S}_{n}(\delta+2^{-\sqrt{n}})\}$, we have $T\in\mathfrak{S}_{n}(\delta)$ if and only if $T$ is not disconnected by $W[0,\mathcal{T}(T)]$ from infinity.
\par
We now claim that $T$ is not disconnected by $W[0,\mathcal{T}(T)]$ from infinity if and only if there exists a continuous curve $\gamma=\gamma(t)$, $t\in[0,1]$, such that 
\[\gamma(0)\in\mathcal{C}_{-n}(T),\ \gamma(1)\in \mathcal{O}_{i+2},\ \gamma\subset\mathcal{D}_{-n+(i+2)L}(S),\ \gamma\cap U_{(i+3)L}\cap W[0,\mathcal{T}(T)]=\varnothing.\] 
The ``if'' direction is trivial. For the ``only if'' direction (see Figure \ref{fig_adapted1} for an illustration of various objects), suppose that $T$ is not disconnected by $W[0,\mathcal{T}(T)]$ from infinity and there does not exist any curve $\gamma$ as described above. In this case, there exists at least one connected component of $\mathcal{D}_{-K}(S)\backslash W[0,\mathcal{T}(T)]$ that intersects $\mathcal{C}_{-n}(T)$ and $\mathcal{C}_{-K}(S)$, denoted by $\mathcal{K}$. Denote by $\mathcal{O}'$ the cut set of $\mathcal{K}$ on $\mathcal{C}_{-n+(i+2)L}(S)$ (with respect to the general annulus $\mathcal{D}_{-K}(S)\backslash\mathcal{D}_{-n}(T)$), we then have $\mathcal{O}_{i+2}\cap\mathcal{O}'=\varnothing$, and $\mathcal{O}'$ is a continuous arc on $\mathcal{C}_{-n+(i+2)L}(S)$. We now focus on endpoints of  $\mathcal{O}_{i+2}$ and $\mathcal{O}'$, namely $P_{1},P_{2}$ and $P_{3},P_{4}$. For $k=1,2,3,4$, we assume that $P_{k}$ is visited by $W^{j_{k}}$. Note that there exist continuous curves $\gamma_{1},\gamma_{2}$, such that $\gamma_{1}$ starts from $\mathcal{C}_{-n}(S)$, passes through $\mathcal{O}_{i+2}$ and stops on $\mathcal{C}_{-K}(S)$, with $\gamma_{1}\cap W[0,\mathcal{T}]=\varnothing$, and $\gamma_{2}$ starts from $\mathcal{C}_{-n}(T)$, passes through $\mathcal{O}'$ and stops on $\mathcal{C}_{-K}(S)$, with $\gamma_{2}\cap W[0,\mathcal{T}(T)]=\varnothing$. Note that on $B_{2}^{c}$, the opening $\mathcal{O}(W[0,\mathcal{T}],\mathcal{A}(S,2^{-n+(i+2)L},2^{-K}))$ contains a unique connected component, which implies that there exists a continuous curve $\gamma_{3}$ in $\mathcal{A}(S,2^{-n+(i+2)L},2^{-K})$ that connects $\gamma_{1},\gamma_{2}$ and does not intersect $W[0,\mathcal{T}]$. We now consider the bounded domain $\mathfrak{D}$ enclosed by $\gamma_{1},\gamma_{2},\gamma_{3}$ and $\mathcal{C}_{-n+(i+1)L}$, and there exists $k$, such that the path $W^{j_{k}}$ started from $P_{j}$ cannot escape from $\mathfrak{D}$ on $B_{3}^{c}$, which yields a contradiction since all $W^{i}$ end at $\mathcal{C}_{-K}(S)$, and thus we complete the proof of the claim.
\begin{figure}[t]
\centering
\includegraphics[width=6cm]{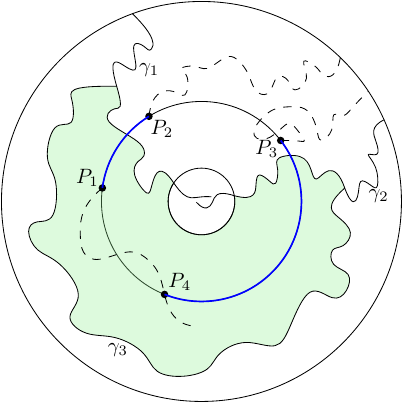}
\caption{\small An illustration for the proof of ``only if'' direction. The sets $\mathcal{O}_{i+2},\mathcal{O}'$ are in blue. Three discs in this figure are $\mathcal{C}_{-n+(i+1)L}(S)$, $\mathcal{C}_{-n+(i+2)L}(S)$ and $\mathcal{C}_{-K}(S)$. Paths $W^{i}$'s (after hitting $P_{j}$'s) are dashed. The path(s) passing through $P_{1}$ or $P_{4}$ are trapped in $\mathfrak{D}$ (which is in green).}
\label{fig_adapted1}
\end{figure}
\par
Finally, note that everything mentioned above is fully determined by $U_{(i+3)L}$, the result hence follows.
\end{proof}
\subsection{Proof of Proposition \ref{cgs}}\label{subsec:martingale}
\begin{proof}[Proof of Proposition \ref{cgs}]
Recall \eqref{delta-value} that $\delta/2=2^{-K}$ for some integer $K$. For $S\in\mathfrak{S}_{n}$, we recall the definition of $\mathcal{A}_{i},N_{i},\widetilde{N}_{i}$ above Lemma \ref{adapted}. It thus suffices to prove:
\begin{equation}\label{cgs1}
\widetilde{\mathbb{P}}_{n}\left(\sum_{i=1}^{[L/2]}\widetilde{N}_{i}\leq\lambda L\right)\leq cn^{-\gamma},
\end{equation}
where $\widetilde{\mathbb{P}}_{n}(\cdot)$ is defined in \eqref{conditioned on box}, and the constant $c$ depends on the choice of $K$ (i.e.\ $\delta$). By Lemma \ref{adapted}, $\widetilde{N}_{i}$ is $\mathscr{F}_{(i+3)}$-measurable under the conditional probability measure $\overline{\mathbb{P}}_{n}$.
\par
Next, we need the following proposition, which will be proved in the next section.
\begin{prop}\label{exist good box}
There exist $C_{1},C_{2}$, such that for any large $n$ and any $2\leq i\leq L/2$, we have
\begin{equation}\label{eq:exist good box}
\overline{\mathbb{P}}_{n}(\widetilde{N}_{i}\geq C_{1}L|\mathscr{F}_{i-1})\geq C_{2}/i.    
\end{equation}
\end{prop}
We now set $Z_{i}=\mathds{1}\{\widetilde{N}_{4i}\geq C_{1}L\}$. By Proposition \ref{exist good box}, for any $1\leq i\leq[L/8]$, we have
\begin{equation}\label{good box probability}\overline{\mathbb{P}}_{n}(Z_{i}=1|Z_{1},\ldots,Z_{i-1})\geq C_{2}/i.\end{equation}
Note that once there exists $i$, such that $Z_{i}=1$, then $\sum_{i=1}^{[L/2]}\widetilde{N}_{i}\geq C_{1}L$. Therefore, there exist $c_{1},c_{2}>0$ such that
\begin{align*}
&\overline{\mathbb{P}}_{n}\left(\sum_{i}^{[L/2]}\widetilde{N}_{i}\leq C_{1}L\right)\leq\overline{\mathbb{P}}_{n}\left(Z_{i}=0, 1\le i \le [L/8]\right)\\
\leq&\prod_{i=2}^{[L/8]}\left(1-\frac{C_{2}}{i}\right)\leq c_{1}e^{-C_{2}\log(L/8)}\leq c_{2}n^{-C_{2}/2}.
\end{align*}
By Lemmas \ref{unique opening}--\ref{all good opening}, there exist constants $c_{3},c_{4}>0$, such that
\[\widetilde{\mathbb{P}}_{n}(B_{j})\leq c_{3}e^{-c_{4}\sqrt{n}},\quad j=1,2,3,4\]
and
\[\widetilde{\mathbb{P}}_{n}(B_{5})\leq c_{3}n^{-1}.\]
We finally obtain
\[\widetilde{\mathbb{P}}_{n}\left(\sum_{j=1}^{[L/2]}\widetilde{N}_{j}\leq C_{1}L\right)\leq\overline{\mathbb{P}}_{n}\left(\sum_{j=1}^{[L/2]}\widetilde{N}_{j}\leq C_{1}L\right)+\widetilde{\mathbb{P}}_{n}(\cup_{j=1}^{5}B_{j})\leq c_{6}n^{-\gamma},\]
where $\gamma=\min\{1,C_{2}/2\}$, which completes the proof.
\end{proof}
\subsection{Proof of Proposition \ref{exist good box}}\label{subsec:pf-4.13}
In this section, we will prove Proposition \ref{exist good box}. To this end, we need a stronger version of Corollary \ref{cme} by taking the filtration $\mathscr{F}_{i}$ into consideration.
\begin{lemma}\label{cme-f}
For a fixed $n$-box $S$, recall that $\widetilde{\mathbb{P}}(\cdot):=\mathbb{P}(\cdot|S\in\mathfrak{S}_{n}(\delta+2^{-\sqrt{n}}))$, and assume that
\[\mathrm{dist}(S,T),\mathrm{dist}(S,U)\geq 2^{-n+jL}.\]
Then there exist $c,c_{1},c_{2}>0$, such that
\begin{equation}\label{cme0-1}
c_{1}\mathrm{dist}(S,T)^{-2}2^{-2n}\leq\widetilde{\mathbb{P}}_{n}(T\in\mathfrak{S}_{n}|\mathscr{F}_{j-1})\leq c_{2}\mathrm{dist}(S,T)^{-2}2^{-2n},
\end{equation}
and
\begin{equation}\label{cme0-2}
\widetilde{\mathbb{P}}_{n}(T,U\in\mathfrak{S}_{n}|\mathscr{F}_{j-1})\leq cD_{1}^{-2}D_{2}^{-2}2^{-4n},
\end{equation}
where
\[D_{1}=\min\{\mathrm{dist}(S,T),\mathrm{dist}(T,U),\mathrm{dist}(U,S)\},\]
and
\[D_{2}=\max\{\mathrm{dist}(S,T),\mathrm{dist}(T,U),\mathrm{dist}(U,S)\},\] and $\mathscr{F}_{j}$ is the increasing family of $\sigma$-fields defined in \eqref{filtration}.
Moreover,
\begin{equation}\label{cme1}
\widetilde{\mathbb{P}}_{n}(T\in\mathfrak{S}_{n},T\lessdot S|\mathscr{F}_{j-1})\geq c'\mathrm{dist}(S,T)^{-2}2^{-2n}
\end{equation}
for some $c'>0$.
\end{lemma}
We will only prove the (more technical) lower bound \eqref{cme1}, which immediately implies the left-hand side of \eqref{cme0-1}. In fact, the proof of upper bound is easier since it does not involve separation lemmas.
\begin{proof}[Proof of \eqref{cme1}] We will conclude the proof by giving upper bound for $\mathbb{P}(S\in\mathfrak{S}_{n}(\delta+2^{-\sqrt{n}})|\mathscr{F}_{j-1})$ and lower bound for $\mathbb{P}(S\in\mathfrak{S}_{n}(\delta+2^{-\sqrt{n}}),T\in\mathfrak{S}_{n},T\lessdot S|\mathscr{F}_{j-1})$. Although the filtration $\mathscr{F}_{j}$ is supported on $V_{\delta}^{\mathrm{PTP}}(S)$, we still work on $\mathbb{P}$ since $S\in\mathfrak{S}_{n}(\delta+2^{-\sqrt{n}})$ also implies $V_{\delta}^{\mathrm{PTP}}(S)$.
\par
Let 
\[d_{0}:=\mathrm{dist}(S,T)\in[2^{-n+m},2^{-n+m+1})\]
for some integer $m\geq jL$ and let $\tau_{i}(\cdot)$ be the first hitting times with respect to $W^{i}$. Recall the notation $\tau_{i,j},\sigma_{i,j}$ from the paragraph above \eqref{filtration}. Similar to the proof of Lemma \ref{sme}, we consider the case that
\[\tau(T,S,\partial\mathcal{B}(S,3\delta/4),S,T,\partial\mathcal{B}(S,3\delta/4),T,S)<\tau_{\mathbb{D}}.\]
We then consider the following events:
\par
($F^{12}$)\quad The union of $W^{i}[0,\tau_{i,m-2}]$, $i=1,\ldots,5$, together with the initial configuration $U_{(j-1)L}$, does not disconnect $\mathcal{C}_{-n}(S)$ from infinity.
\par
($F^{13}$)\quad $F^{12}$ occurs, and $W^{i}[0,\tau_{i,m-2}]$, $i=1,\ldots,5$ are $\alpha$-nice at the end.
\par
($F^{14}$)\quad $W^{1},W^{4}$ and $W^{5}$ intersect $T$ before exit $\mathcal{C}_{-n+m+1}(S)$.
\par
($F^{15}$)\quad $W^{i}[\tau_{i,m+1},\tau_{i}(\mathcal{C}_{-n+m-2}(T))]$ stays in some well-chosen tubes for $i=1,4,5$.
\par
($F^{16}$)\quad The family of curves $W^{i}[\tau_{i}(T,\mathcal{C}_{-n}(T)),\tau_{i}(T,\mathcal{C}_{-n+m-2}(T))]$, $i=1,4,5$ and $W^{i}[\tau_{i}(\mathcal{C}_{-n+m-2}(T)),\tau_{i}(\mathcal{C}_{-n}(T))]$, $i=1,4$ form an $\alpha$-nice configuration in the annulus $\mathcal{A}(T,2^{-n},2^{-n+m-2})$. 
\par
($F^{17}$)\quad $W^{i}[\tau_{i}(T,\mathcal{C}_{-n+m-2}(T)),\tau_{i}(T,\mathcal{C}_{-n+m+1}(S))]$ stays in some well-chosen tube for $i=1,4,5$, and $W^{i}[\tau_{i,m-2},\tau_{i,m+1}]$ stays in some well-chosen tube for $i=2,3$.
\par
($F^{18}$)\quad The union of $W^{i}[\tau_{i,m+1},\tau_{i,n-K-L-2}]$ does not disconnect $\mathcal{C}_{-n}(S)$ from infinity.
\par
($F^{19}$)\quad $W^{i}[\tau_{i,m+1},\tau_{i,n-K-L-2}]$, $i=1,\ldots,5$ form an $\alpha$-nice configuration.
\par
($F^{20}$)\quad $W^{i}[\tau_{i,n-K-L-2},\sigma_{i,n-K-L-1}]$ stay in some well-chosen tube for $i=1,\ldots,5$. 
\par
($F^{21}$)\quad The union of $W^{i}[\sigma_{i,n-K-L-1},\tau_{i,n-K}]$, together with the initial configuration $U_{(j-1)L}$, does not disconnect $\mathcal{C}_{-n}(S)$ from infinity.
\par
($F^{22}$)\quad $F^{21}$ occurs, and $W^{i}[\sigma_{i,n-K-L-1},\tau_{i,n-K}]$, $i=1,\ldots,5$ are $\alpha$-nice at the beginning.
\par
To ensure that $S$ is $(\delta+2^{-\sqrt{n}})$-good, we require that (recall the definition of $X^{2},X^{4}$ from \eqref{decomposition}):
\par
($F^{23}$)\quad Both bridges $X^{2},X^{4}$ intersect $\partial\mathcal{B}(S,\delta+2^{-\sqrt{n}}-2^{-n-1/2})$.
\par
Note that $F^{23}$ is $\mathscr{F}_{j-1}$-measurable, and thus we restrict $\mathscr{F}_{j-1}$ to where $F^{23}$ occurs.
\par
We then have
\[\{S\in\mathfrak{S}_{n}(\delta+2^{-\sqrt{n}})\}\Rightarrow F^{12}\cap F^{18}\cap F^{21}\cap F^{23},\]
and
\[\{S\in\mathfrak{S}_{n}(\delta+2^{-\sqrt{n}}),T\in\mathfrak{S}_{n},T\lessdot S\}\Leftarrow F^{13},F^{14},F^{15},F^{16},F^{17},F^{19},F^{20},F^{22},F^{23}.\]
By Lemma \ref{sep PTP-1}, the separation lemma with initial configuration for PTP, we have
\[\mathbb{P}(F^{13}|\mathscr{F}_{j-1})\geq c_{3}\mathbb{P}(F^{12}|\mathscr{F}_{j-1}).\]
Applying Lemma \ref{sep PTP-1} again, along with inversion invariance of Brownian motion, we have
\[\mathbb{P}(F^{22}|\mathscr{F}_{j-1})\geq c_{4}\mathbb{P}(F^{21}|\mathscr{F}_{j-1}).\]
We now recall the definition of ``almost independence'' in Section \ref{subsec:notation}. Thanks to Lemmas \ref{ac excursions} and \ref{lem:harnack}, conditioned on $\mathscr{F}_{j-1}$, events $F^{12},F^{18},F^{21}$ are almost independent. And $F^{18}$ is almost independent to $\mathscr{F}_{j-1}$. Moreover, $F^{23}$ is measurable with respect to $\mathscr{F}_{j-1}$. Therefore, we have
\[\mathbb{P}(S\in\mathfrak{S}_{n}(\delta+2^{-\sqrt{n}})|\mathscr{F}_{j-1})\leq c_{5}\mathds{1}\{F^{23}\}\mathbb{P}(F^{12}|\mathscr{F}_{j-1})\mathbb{P}(F^{21}|\mathscr{F}_{j-1})\mathbb{P}(F^{18}).\]
Thanks to Lemmas \ref{ac excursions} and \ref{lem:harnack}, conditioned on $\mathscr{F}_{j-1}$, events $F^{13},F^{14}\cap F^{16},F^{19},F^{22}$ are almost independent. Furthermore, events $F^{14}\cap F^{16},F^{19}$ are almost independent of $\mathscr{F}_{j-1}$. Moreover, the joint law of 
\[\{W^{i}[\tau_{i}(\mathcal{C}_{-n+m-2}(T)),\tau_{i}(\mathcal{C}_{-n}(T))], i=1,4\}\]
and
\[\{W^{i}[\tau_{i}(T,\mathcal{C}_{-n}(T)),\tau_{i}(T,\mathcal{C}_{-n+m-2}(T))], i=1,4,5\}\] 
is uniformly equivalent to five independent Brownian motions by Lemma~\ref{lem:harnack}.
By Lemma \ref{sep PTP}, we thus have
\[\mathbb{P}(F^{14}\cap F^{16})\asymp 2^{-2m}.\]
Moreover, similar to the proof of Lemma \ref{sme}, we can draw well-chosen tubes and force these trajectories to pass through tubes with a constant probability cost so that they will not destroy the global opening. Therefore, conditioned on $\mathscr{F}_{i-1}$ and $F^{13}\cap F^{14}\cap F^{16}\cap F^{19}\cap F^{22}$, events $F^{15},F^{17}$ and $F^{20}$ are almost independent and occur with universal positive probability. Again, note that $F^{23}$ is $\mathscr{F}_{j-1}$-measurable, we conclude that
\begin{align*}
&\quad\ \mathbb{P}(S\in\mathfrak{S}_{n}(\delta+2^{-\sqrt{n}}),T\in\mathfrak{S}_{n},T\lessdot S|\mathscr{F}_{j-1})\\&\geq c_{6}\mathds{1}\{F^{23}\}\mathbb{P}(F^{13}|\mathscr{F}_{j-1})\mathbb{P}(F^{22}|\mathscr{F}_{j-1})\mathbb{P}(F^{19})\cdot 2^{-2m}.
\end{align*}
And thus we have
\[\mathbb{P}(S\in\mathfrak{S}_{n}(\delta+2^{-\sqrt{n}}),T\in\mathfrak{S}_{n},T\lessdot S|\mathscr{F}_{j-1})\geq c_{7}2^{-2m}\mathbb{P}(S\in\mathfrak{S}_{n}(\delta+2^{-\sqrt{n}})|\mathscr{F}_{j-1}),\]
which completes the proof.
\end{proof}
The following lemma will be the last step before the proof of Proposition \ref{exist good box}, in which we prove the same inequality as \eqref{eq:exist good box} but for $\widetilde{\mathbb{P}}_{n}$ instead of $\overline{\mathbb{P}}_{n}$.
\begin{lemma}\label{exist good box+}
There exists $c,c'>0$, such that for any large $n$ and any $2\leq i\leq[L/2]$, we have
\[\widetilde{\mathbb{P}}_{n}(\widetilde{N}_{i}\geq c'L|\mathscr{F}_{i-1})\geq c/i.\]
\end{lemma}
\begin{proof}
To apply Paley-Zygmund inequality, we need to estimate the lower bound of $\widetilde{\mathbb{E}}_{n}[\widetilde{N}_{i}|\mathscr{F}_{i-1}]$ and the upper bound of $\widetilde{\mathbb{E}}_{n}[\widetilde{N}_{i}^{2}|\mathscr{F}_{i-1}]$. To this end, it suffices to consider $\widetilde{\mathbb{P}}_{n}(T\in\mathfrak{S}_{n},T\lessdot S|\mathscr{F}_{i-1})$ and $\widetilde{\mathbb{P}}_{n}(T,U\in\mathfrak{S}_{n}|\mathscr{F}_{i-1})$ for $T,U\prec\mathcal{A}_{i}$. By Lemma \ref{cme-f}, we have:
\[\widetilde{\mathbb{P}}_{n}(T\in\mathfrak{S}_{n},T\lessdot S|\mathscr{F}_{i-1})\geq c\mathrm{dist}(S,T)^{-2}2^{-2n}\]
and
\[\widetilde{\mathbb{P}}_{n}(T,U\in\mathfrak{S}_{n}|\mathscr{F}_{i-1})\leq c'D_{1}^{-2}D_{2}^{-2}2^{-4n},\]
where 
\[D_{1}=\min\{\mathrm{dist}(S,T),\mathrm{dist}(T,U),\mathrm{dist}(U,S)\},\]
and
\[D_{2}=\max\{\mathrm{dist}(S,T),\mathrm{dist}(T,U),\mathrm{dist}(U,S)\}.\]
We thus have
\begin{equation}\label{cfme+}
\begin{split}
\widetilde{\mathbb{E}}_{n}[\widetilde{N}_{i}|\mathscr{F}_{i-1}]&\geq\sum_{T\in\mathcal{S}_{n},T\prec \mathcal{A}_{i}}\widetilde{\mathbb{P}}_{n}(T\in\mathfrak{S}_{n},T\lessdot S|\mathscr{F}_{i-1})\\
&\geq\sum_{j=iL}^{(i+1)L-1}\sum_{\substack{T\in\mathcal{S}_{n},\\ T\prec\mathcal{A}(x,2^{-n+j},2^{-n+j+1})}}\widetilde{\mathbb{P}}_{n}(T\in\mathfrak{S}_{n},T\lessdot S|\mathscr{F}_{i-1})\\
&\geq\sum_{j=iL}^{(i+1)L-1}2^{2j}\cdot c'2^{2(n-j)}2^{-2n}=c'L.
\end{split}
\end{equation}
We now turn to the conditional second moment, which is more complicated. For any $n$-box $T\prec\mathcal{A}_{i}$, we denote by $T\vartriangleleft j$ if $T$ is located in scale $j$, or more precisely, that $j$ is the minimal integer such that $T\cap\mathcal{A}(x,2^{-n+j},2^{-n+j+1})\neq\varnothing$. We then assume $T\vartriangleleft j$ and $U\vartriangleleft k$. If $|j-k|\geq2$, we then have 
\[D_{1}\asymp2^{-n+\min\{j,k\}}\mbox{ and }D_{2}\asymp2^{-n+\max\{j,k\}}.\] 
And thus we have
\[\widetilde{\mathbb{P}}_{n}(T,U\in\mathfrak{S}_{n}|\mathscr{F}_{i-1})\leq c''2^{4n-2j-2k}2^{-4n}=c''2^{-2(j+k)}.\]
If $|j-k|\leq1$, then $D_{2}\asymp2^{-n+j}$. We now fix $T\vartriangleleft j$, and write
\begin{equation}\label{conditional second moment}\sum_{U\in\mathcal{S}_{n},U\prec\mathcal{A}_{i}}\widetilde{\mathbb{P}}_{n}(T,U\in\mathfrak{S}_{n}|\mathscr{F}_{i-1})=\sum_{k=iL}^{(i+1)L-1}\sum_{\substack{U\in\mathcal{S}_{n},\\ U\vartriangleleft k}}\widetilde{\mathbb{P}}_{n}(T,U\in\mathfrak{S}_{n}|\mathscr{F}_{i-1}).
\end{equation}
We divide the right-hand side of \eqref{conditional second moment} into two parts:  $|k-j|>1$ and $|k-j|\leq1$. For the first part, we have
\begin{equation}\label{case1}\sum_{\substack{U\in\mathcal{S}_{n},\\ U\vartriangleleft k}}\widetilde{\mathbb{P}}_{n}(T,U\in\mathfrak{S}_{n}|\mathscr{F}_{i-1})\leq c''2^{-2(j+k)}\cdot2^{2(-n+k)}/2^{-2n}=c''2^{-2j}.\end{equation}
For the second part, since $\mathrm{dist}(T,U)\leq 2^{-n+j+2}$, we have
\begin{equation}\label{case2}
\begin{split}
\sum_{\substack{U\in\mathcal{S}_{n},\\ U\vartriangleleft k}}\widetilde{\mathbb{P}}_{n}(T,U\in\mathfrak{S}_{n}|\mathscr{F}_{i-1})&\leq\sum_{m=n-j-1}^{n}\sum_{\substack{U\in\mathcal{S}_{n},\\ \mathrm{dist}(T,U)\in[2^{-m},2^{-m+1}]}}c''2^{2m}2^{2(n-j)}2^{-4n}\\
&\leq\sum_{m=n-j-1}^{n}c'''2^{2m}2^{2(n-j)}2^{-4n}\cdot 2^{-2m}/2^{-2n}=c'''(j+2)2^{-2j}.
\end{split}
\end{equation}
Combining \eqref{conditional second moment}--\eqref{case2}, and summing over $j$, we conclude that
\begin{equation}\label{csme+}
\sum_{\substack{T,U\in\mathcal{S}_{n} \\ T,U\prec\mathcal{A}_{i}}}\widetilde{\mathbb{P}}_{n}(T,U\in\mathfrak{S}_{n}|\mathscr{F}_{i-1})\leq\sum_{j=iL}^{(i+1)L-1}\tilde{c}2^{-2j}(j+L+2)\cdot2^{-2(n-j)}/2^{-2n}\leq\tilde{c}'iL^{2}.
\end{equation}
Combining \eqref{cfme+} and \eqref{csme+}, by Paley-Zygmund inequality, for any $\theta\in(0,1)$, we have
\[\widetilde{\mathbb{P}}_{n}(\widetilde{N}_{k}\geq \theta c'L|\mathscr{F}_{i-1})\geq (1-\theta)^{2}\frac{\widetilde{\mathbb{E}}_{n}^{2}[\widetilde{N}_{i}|\mathscr{F}_{i-1}]}{\widetilde{\mathbb{E}}_{n}[\widetilde{N}_{i}^{2}|\mathscr{F}_{i-1}]}\geq\tilde{c}''/i,\]
which completes the proof.
\end{proof}
Finally, thanks to the control of the probability of ``bad events'', we will show that the difference of $\widetilde{\mathbb{P}}_{n}$ and $\overline{\mathbb{P}}_{n}$ is very small, which completes the proof of Proposition \ref{exist good box}.
\begin{proof}[Proof of Proposition \ref{exist good box}]
Note that configurations in $U_{(i-1)L}$ which contradict with events $B_{m}^{c}$, $m=1,\ldots,5$, are excluded once we consider the conditional probability measure $\overline{\mathbb{P}}_{n}(\cdot|\mathscr{F}_{i})$.
Also note that $B_{1},B_{2},B_{4}$ are all $\mathscr{F}_{1}$-measurable, and thus $\overline{\mathbb{P}}_{n}(B_{j}|\mathscr{F}_{i-1})=\mathds{1}\{B_{j}\}=0$ for any $i\geq2$ and $j=1,2,4$. Recall the definition of $\overline{E}_{j}$ from \eqref{eq:barEj}, and let $\widehat{E}_{i}$ be the union of $\overline{E}_{j}$ for $1\leq j\leq i-1$. Then $\widehat{E}_{i}$ is $\mathscr{F}_{(i-1)}$-measurable, and thus $\overline{\mathbb{P}}_{n}(\widehat{E}_{j}|\mathscr{F}_{i-1})=\mathds{1}\{\widehat{E}_{j}\}=1$ for $1\leq j\leq i-1$. Let $W^{m,i}:=W^{m}[0,\tau_{m,(i-1)L}]$. Let $E_{m,j}^{i}$ be the event that $W^{m,i}$ returns to $\mathcal{D}_{-n+jL}(S)$ after the first hitting of $\mathcal{C}_{-n+(j+1)L}(S)$. Let
\[E_{i}^{\dagger}:=\bigcap_{m=1}^{5}\bigcap_{j=1}^{i-2}(E_{m,j}^{i})^{c}.\]
Then $E_{i}^{\dagger}$ is $\mathscr{F}_{(i-1)}$-measurable. Therefore we have $\overline{\mathbb{P}}_{n}(E_{i}^{\dagger}|\mathscr{F}_{i-1})=\mathds{1}\{E_{i}^{\dagger}\}=1$. Recall the definition of $F^{23}$ in the proof of Lemma~\ref{cme-f}. We then let $E_{i}^{\#}$ be the union of $B_{1}^{c},B_{2}^{c},B_{4}^{c},\widehat{E}_{i},E_{i}^{\dagger}$ and $F^{23}$. From the discussion above, we have $\overline{\mathbb{P}}_{n}(E_{i}^{\#}|\mathscr{F}_{i-1})=1$.
\par
Thanks to Lemma \ref{exist good box+}, it suffices to compare the probability measure $\widetilde{\mathbb{P}}_{n}$ and $\overline{\mathbb{P}}_{n}$ on $\mathscr{F}_{i-1}$ for any $i$. 
We have
\begin{align*}
\overline{\mathbb{P}}_{n}(\widetilde{N}_{k}\geq c'L|\mathscr{F}_{i-1})&=\frac{\widetilde{\mathbb{P}}_{n}(\widetilde{N}_{k}\geq c'L,B_{3}^{c}\cap B_5^c|\mathscr{F}_{i-1}, E_{i}^{\#})}{\widetilde{\mathbb{P}}_{n}(B_3^c \cap B_5^c|\mathscr{F}_{i-1}, E_{i}^{\#})}\\
&\geq\widetilde{\mathbb{P}}_{n}(\widetilde{N}_{k}\geq c'L,B_{3}^{c}\cap B_5^c|\mathscr{F}_{i-1}, E_{i}^{\#}).
\end{align*}
We then have
\[\widetilde{\mathbb{P}}_{n}(B_{5}^{c}|\mathscr{F}_{i-1},E_{i}^{\#})=\widetilde{\mathbb{P}}_{n}\left(\cap_{j=i}^{[L/2]}\overline{E}_{j}^{c}|\mathscr{F}_{i-1},E_{i}^{\#}\right).\]
By an argument similar to the proof of Lemma \ref{all good opening}, for any $j\geq i$, we have
\[\widetilde{\mathbb{P}}_{n}(\overline{E}_{j}|\mathscr{F}_{i-1},E_{i}^{\#})\leq c_{1}n^{-2}.\]
By the union bound, we thus conclude that $\widetilde{\mathbb{P}}_{n}(B_{5}|\mathscr{F}_{i-1},E_{i}^{\#})\leq c_{1}n^{-1}$.
\par
It then suffices to control $\widetilde{\mathbb{P}}_{n}(B_{3}\cap B_{5}^{c}|\mathscr{F}_{i-1},E_{i}^{\#})$. On $E_{i}^{\dagger}$, note that if $E_{m,i-1}$ does not occur, then $E_{m,j}$ will not occur for any $j<i-1$. Therefore, we only consider the conditional probability of $E_{m,j}$ when $j\geq i-1$.
\par
For any $j\geq i$, we now consider $\widetilde{\mathbb{P}}_{n}(E_{1,j-1}\cap\overline{E}_{j}^{c}|\mathscr{F}_{i-1},E_{i}^{\#})$. We first note that $\overline{E}_{j}$ is $\mathscr{F}_{j}$-measurable, and thus we have
\begin{equation}\label{eq:e1j}
\widetilde{\mathbb{P}}_{n}(E_{1,j-1}\cap\overline{E}_{j}^{c}|\mathscr{F}_{j},E_{i}^{\#})=\mathds{1}_{\overline{E}_{j}^{c}}\widetilde{\mathbb{P}}_{n}(E_{1,j-1}|\mathscr{F}_{j},E_{i}^{\#})\leq\widetilde{\mathbb{P}}_{n}(E_{1,j-1}|\mathscr{F}_{j},\overline{E}_{j}^{c},E_{i}^{\#}).
\end{equation}
On $E_{1,j-1}$, $W^{1}$ intersects $\mathcal{C}_{-n+(j-1)L}(S)$ after $\tau_{1,jL}$. We now need an upper bound of $\mathbb{P}(\{S\in\mathfrak{S}_{n}(\delta+2^{-\sqrt{n}})\}\cap E_{1,j-1}|\mathscr{F}_{j}, \overline{E}_{j}^{c},E_{i}^{\#})$ and a lower bound of $\mathbb{P}(S\in\mathfrak{S}_{n}(\delta+2^{-\sqrt{n}})|\mathscr{F}_{j}, \overline{E}_{j}^{c},E_{i}^{\#})$. For the first probability, we simply note that on $E_{1,j-1}$, $W^{1}$ will intersect $\mathcal{C}_{-n+(j-1)L}(S)$ after $\tau_{1,jL}$. By the strong Markov property and Lemma~\ref{lem:harnack}, the part of $W^{1}$ from $\mathcal{C}_{-n+(j-1)L+1}(S)$ to $\mathcal{C}_{-n+jL-1}(S)$ is uniformly equivalent to a Brownian motion started uniformly from $\mathcal{C}_{-n+(j-1)L+1}(S)$ and stopped upon reaching $\mathcal{C}_{-n+jL-1}(S)$. 
\par
Note that $\{S\in\mathfrak{S}_{n}(\delta+2^{-\sqrt{n}})\}\cap E_{1,j-1}$ will imply the following two events:
\par
($F^{24}$)\quad $W^{1}[\tau_{1}(\mathcal{C}_{-n+jL}(S),\mathcal{C}_{-n+(j-1)L}(S)),\tau_{1}(\mathcal{C}_{-n+jL}(S),\mathcal{C}_{-n+(j-1)L}(S),\mathcal{C}_{-n+jL-1}(S))]$ does not disconnect $\mathcal{C}_{-n+(j-1)L}(S)$ from infinity.
\par
($F^{25}$)\quad The union of $W^{1}[\tau_{1}(\mathcal{C}_{-n+jL}(S),\mathcal{C}_{-n+(j-1)L}(S),\mathcal{C}_{-n+jL+1}(S)),\tau_{1,n-K-L-1}]$\\ and  $W^{m}[\tau_{m,jL+1},\tau_{m,n-K-L-1}]$, $m=2,\ldots,5$ does not disconnect $\mathcal{C}_{-n+jL+1}(S)$ from infinity.
\par
By Lemma \ref{lem:harnack} and \ref{sep PTP}, and thanks to the strong Markov property, we have
\begin{equation}\label{seij}
\begin{split}
&\mathbb{P}(\{S\in\mathfrak{S}_{n}(\delta+2^{-\sqrt{n}})\}\cap E_{1,j-1}|\mathscr{F}_{j},\overline{E}_{j}^{c},E_{i}^{\#})\\
\leq&\mathbb{P}(F^{24}|\mathscr{F}_{j},\overline{E}_{j}^{c},E_{i}^{\#})\mathbb{P}(F^{25}|\mathscr{F}_{j},\overline{E}_{j}^{c},E_{i}^{\#},F^{24})
\leq c_{2}2^{-L/4}2^{-2(n-(j+1)L)}.
\end{split}
\end{equation}
\par
On the other hand, recall the definition of $\alpha$-nice configuration and $\alpha$-separated event from Section~\ref{subsec:sep}, where $\alpha$ is some appropriately chosen constant. We then consider the following four events:
\par
($F^{26}$)\quad $U_{jL+1}$ does not disconnect $\mathcal{D}_{-n}(S)$ from infinity, and the family of curves $W^{1}[0,\tau_{1,jL+1}]$, $\ldots$, $W^{5}[0,\tau_{5,jL+1}]$ are $\alpha$-separated at the end.
\par
($F^{27}$)\quad $W^{m}[\tau_{m,jL+1},\tau_{m,jL+2}]$  for $m=1,\ldots,5$ stay in some well-chosen tubes.
\par
($F^{28}$)\quad $W^{m}[\tau_{m,jL+2},\tau_{m,n-K-L-1}]$, $m=1,\ldots,5$ form an $\alpha$-nice configuration in the annulus $\mathcal{A}(S,2^{-n+jL+1},2^{-K-L-1})$.
\par
($F^{29}$)\quad $W^{m}[\tau_{m,n-K-L-1},\sigma_{m,n-K-L}]$  for $m=1,\ldots,5$ stay in some well-chosen tubes.
\par
By Lemma~\ref{sep PTP-1}, we have
\[\mathbb{P}(F^{26}|\mathscr{F}_{j},\overline{E}_{j}^{c},E_{i}^{\#})\geq c_{3}n^{-10}.\]
By Lemma~\ref{lem:harnack}, the strong Markov property and Lemma~\ref{sep PTP}, we have:
\[\mathbb{P}(F^{28}|\mathscr{F}_{j},\overline{E}_{j}^{c},E_{i}^{\#},F^{26})\geq c_{4}2^{-2(n-(j+1)L}.\]
By carefully choosing tubes, we can ensure that the conditional probabilities of $F^{27},F^{29}$ are bounded below by some constants. More precisely, we have
\[\mathbb{P}(F^{27}|\mathscr{F}_{j},\overline{E}_{j}^{c},E_{i}^{\#},F^{26},F^{28})\geq c_{5}, \quad \mathbb{P}(F^{29}|\mathscr{F}_{j},\overline{E}_{j}^{c},E_{i}^{\#},F^{26},F^{27},F^{28})\geq c_{6}.\]
Note that the union of $F^{26},F^{27},F^{28},F^{29}$ and $\overline{E}_{j}^{c},E_{i}^{\#}$ will imply $S\in\mathfrak{S}_{n}(\delta+2^{-\sqrt{n}})$, we thus conclude that
\begin{equation}\label{scfj}
\mathbb{P}(S\in\mathfrak{S}_{n}(\delta+2^{-\sqrt{n}})|\mathscr{F}_{j},\overline{E}_{j}^{c},E_{i}^{\#})\geq c_{7}n^{-10}2^{-2(n-(j+1)L)}.
\end{equation}
Dividing \eqref{seij} by \eqref{scfj}, combining with \eqref{eq:e1j}, we have
\begin{equation}\label{1jl}
\widetilde{\mathbb{P}}_{n}(E_{1,j-1}\cap\overline{E}_{j}^{c}|\mathscr{F}_{j},E_{i}^{\#})\leq c_{8}2^{-L/4}n^{10}\leq 2^{-\sqrt{n}/5}.
\end{equation}
Since $j\geq i$, by tower property of conditional expectation and applying union bound, we have
\[\widetilde{\mathbb{P}}_{n}(B_{3}\cap B_{5}^{c}|\mathscr{F}_{i-1},E_{i}^{\#})\leq 2^{-\sqrt{n}/6}.\]
Hence, combining estimates above  with Lemma \ref{exist good box+}, we conclude that
\begin{align*}
\frac{\widetilde{\mathbb{P}}_{n}(\widetilde{N}_{k}\geq c'L,B_{3}^{c}\cap B_{5}^{c}|\mathscr{F}_{i-1},E_{i}^{\#})}{\widetilde{\mathbb{P}}_{n}(B_{3}^{c}\cap B_{5}^{c}|\mathscr{F}_{i-1}, E_{i}^{\#})}&\geq \widetilde{\mathbb{P}}_{n}(\widetilde{N}_{k}\geq c'L,B_{3}^{c},B_{5}^{c}|\mathscr{F}_{i-1},E_{i}^{\#})\\
&\geq c/i-\widetilde{\mathbb{P}}_{n}(B_{3}|\mathscr{F}_{i-1},E_{i}^{\#})-\widetilde{\mathbb{P}}_{n}(B_{5}|\mathscr{F}_{i-1},E_{i}^{\#})\geq c'/i
\end{align*}
where we used $n^{-1},2^{-\sqrt{n}/6}\ll c'/i$ for any $2\le i\le L/2$ if $n$ is large. This completes the proof.
\end{proof}
\subsection{Proof of Theorem \ref{ptp}}\label{subsec:pf-ptp}
Proposition \ref{cgs} and Lemma \ref{delta+ good box} now lead to the proof of Theorem \ref{ptp}.
\begin{proof}[Proof of Theorem \ref{ptp}]
It suffices to prove that  for any $\delta,\iota>0$, the probability that the annulus $\mathcal{A}(0,\iota,1-\iota)$ contains a $\delta$-pioneer triple point ($\delta$-PTP) is 0. Assume on the contrary that this event has probability $c=c(\iota,\delta)>0$. On this event, if a box $S\in\mathcal{S}_n$ contains such a $\delta$-PTP, then $S$ is also a $\delta$-good box in $\mathcal{A}(0,\iota/2,1-\iota/2)$ if $2^{-n}<\iota/2$. Recall $\mathfrak{S}_{n}:=\mathfrak{S}_{n}(\delta)$ from Definition~\ref{gs} for the set of $\delta$-good boxes. 
Therefore, for all $n$ large enough, we have
\begin{equation}\label{eq:E0}
    \mathbb{P}(E_{0})\geq c,
\end{equation}
where
\[E_{0}=E_{0}(n):=\{\#\{S\in\mathfrak{S}_{n}:S\subset\mathcal{A}(0,\iota/2,1-\iota/2)\}\geq1\}.\]
\par
Next, we will give an upper bound on $\mathbb{P}(E_{0})$ which contradicts \eqref{eq:E0}, and conclude the proof.
On the event $E_{0}$, we randomly pick one box $\widetilde{S}$ with respect to the uniform counting measure on all boxes in $\{S\in\mathfrak{S}_{n}:S\subset\mathcal{A}(0,\iota/2,1-\iota/2)\}$. It follows that
\begin{equation}\label{count1}
\begin{split}
\mathbb{P}(E_{0})=\sum_{\substack{S\in\mathcal{S}_{n} \\ S\subset\mathcal{A}(0,\iota/2,1-\iota/2)}}\mathbb{P}(\widetilde{S}=S,E_{0})&=\sum_{\substack{S\in\mathcal{S}_{n} \\ S\subset\mathcal{A}(0,\iota/2,1-\iota/2)}}\mathbb{P}(\widetilde{S}=S|S\in\mathfrak{S}_{n})\mathbb{P}(S\in\mathfrak{S}_{n})\\
&=\sum_{\substack{S\in\mathcal{S}_{n} \\ S\subset\mathcal{A}(0,\iota/2,1-\iota/2)}}\mathbb{E}[(\#\mathfrak{S}_{n})^{-1}|S\in\mathfrak{S}_{n}]\mathbb{P}(S\in\mathfrak{S}_{n}).
\end{split}
\end{equation}
For any $S\in\mathfrak{S}_{n}$, with $\lambda$ from Proposition \ref{cgs}, we have:
\begin{equation}\label{-1 moment}
\mathbb{E}[(\#\mathfrak{S}_{n})^{-1}|S\in\mathfrak{S}_{n}]\leq\frac{1}{\lambda\sqrt{n}}+\mathbb{P}(\#\mathfrak{S}_{n}\leq\lambda\sqrt{n}|S\in\mathfrak{S}_{n}).
\end{equation}
By Proposition \ref{cgs} and Lemma \ref{delta+ good box}, we have
\begin{equation}\label{sum of boxes}
\begin{split}
&\quad\ \mathbb{P}(\#\mathfrak{S}_{n}\leq\lambda\sqrt{n}|S\in\mathfrak{S}_{n})\\
&\leq\mathbb{P}(\#\mathfrak{S}_{n}\leq\lambda\sqrt{n}|S\in\mathfrak{S}_{n}(\delta+2^{-\sqrt{n}}))+\mathbb{P}(S\notin\mathfrak{S}_{n}(\delta+2^{-\sqrt{n}})|S\in\mathfrak{S}_{n})\\
&\leq c'(n^{-\gamma}+2^{-\sqrt{n}}).
\end{split}
\end{equation}
By Lemma \ref{fme} for PTP, we have $\mathbb{P}(S\in\mathfrak{S}_{n})\leq c'(\iota,\delta)2^{-2n}$ for any $S\subset\mathcal{A}(0,\iota/2,1-\iota/2)$. Therefore, combining \eqref{-1 moment} and \eqref{sum of boxes}, and summing over $S$, we have:
\begin{equation}\label{count2}
\sum_{\substack{S\in\mathcal{S}_{n} \\ S\subset\mathcal{A}(0,\iota/2,1-\iota/2)}}\mathbb{E}[(\#\mathfrak{S}_{n})^{-1}|S\in\mathfrak{S}_{n}]\mathbb{P}(S\in\mathfrak{S}_{n})\leq c''\left(\frac{1}{\lambda\sqrt{n}}+n^{-\gamma}+2^{-\sqrt{n}}\right)\leq c''n^{-c'''}.
\end{equation}
Combining \eqref{count1} and \eqref{count2}, we obtain that 
\begin{equation}\label{eq:E0u}
    \mathbb{P}(E_{0})\leq c''n^{-c'''},
\end{equation}
which contradicts \eqref{eq:E0} when $n$ is large. This finishes the proof.
\end{proof}

\section{Non-Existence of Pioneer Double Cut Points}\label{sec:pdcp}
In this section, we will sketch the proof of Theorem \ref{pdcp}. Since the proof structure is almost the same as that of Theorem \ref{ptp}, we will skip many identical arguments. We will deal with the 2D and 3D cases together in this section. Most statements are the same in both cases but their proofs differ slightly. The main technical difference is that 3D Brownian motion is transient, and thus there exists some extra cost to force a 3D Brownian motion to enter a small ball. This fact is reflected in various moment bounds, e.g., the RHS of \eqref{eq:EijE1bound}. We point out that notation in this section is independent from that of the previous section.
\par
Similar to \eqref{l-value}, we still fix $L=[\sqrt{n}]$. Similar to the discussion next to \eqref{delta-value}, we again assume that $K:=-\log_{2}(\delta/2)$ is a positive integer. We also fix $\iota>0$, and assume $S\in\mathcal{S}_{n}$, $S\subset\mathcal{A}(0,\iota,1-\iota)$.
\par
The following proposition, a parallel version of Proposition \ref{cgs}, is the key ingredient of the proof of Theorem \ref{pdcp}.
\begin{prop}\label{cgs12}
For $d=2,3$, there exist $c,\lambda,\gamma>0$, such that for any large $n$, we have
\[\mathbb{P}(\#\mathfrak{S}_{n}(\delta)<\lambda\sqrt{n}|S\in\mathfrak{S}_{n}(\delta+2^{-\sqrt{n}}))\leq cn^{-\gamma},\]
where $\mathfrak{S}_{n}(\delta)$ is the collection of $\delta$-good boxes defined for PDCP in Definition \ref{gs12}.
\end{prop}
Proposition \ref{cgs12} and a parallel version of Lemma \ref{delta+ good box} (we omit the statement as it is completely identical to the PTP case) will complete the proof of Theorem \ref{pdcp}. We omit this part of proof as it is very similar to the PTP case.
\par
To prove Proposition \ref{cgs12}, we again need to define some bad events. We first state one such event for independent Brownian excursions, showing that conditioned on the non-intersection event, back-crossings can only survive a few scales with high (conditional) probability, which is a parallel version of Lemma \ref{local}.
\par
\begin{lemma}\label{local12}
For $d=2,3$, let $W_{i}$, $i=1,2,3$, be independent Brownian excursions from $\mathcal{C}_{0}$ to $\mathcal{C}_{n}$. Let $E_{1}$ denote the event that $W_{1}\cap(W_{2}\cup W_{3})=\varnothing$. Let $E_{i,j}$ be the event that $W_{i}$ returns to $\mathcal{C}_{jL}$ after hitting $\mathcal{C}_{(j+1)L}$. Then there exists $c$, such that for any integer $j\leq[n/L]-1$, and $i=1,2,3$,
\begin{equation}\label{eq:EijE1bound}
\mathbb{P}(E_{i,j}\cap E_{1})\leq cn^{3(3-d)}2^{-L}2^{-(4-d)n}.
\end{equation}
\end{lemma}    
\begin{proof}
For the $d=2$ case, we only consider $\mathbb{P}(E_{1,j}\cap E_{1})$ for a similar reason as in the proof of Lemma \ref{local}. On the event $E_{1,j}$, we recall stopping times \eqref{decomposing extra crossings}. Then $E_{1,j}\cap E_{1}$ implies the following events:
\par
($F^{1}$) $W_{1}[0,t_{1}]$ and $W_{i}[0,t_{i}]$, $i=2,3$, do not intersect, where $t_{i}$ is the first hitting time of $\mathcal{C}_{jL}$ with respect to $W_{i}$.
\par
($F^{2}$) $W_{i}[t'_{i},\tau_{i}]$, $i=2,3$, and $W_{1}[t'_{1},\tau_{1}]$ do not intersect, where $t'_{i}$ is the first hitting time of $\mathcal{C}_{(j+1)L}$ (specially for $W_{1}$, let $t'_{1}$ stand for the first hitting time of $\mathcal{C}_{(j+1)L}$ after $t_{6}^{1}$) and $\tau_{i}$ is the first hitting time of $\mathcal{C}_{n}$ with respect to $W_{i}$.
\par
($F^{3}$) The union of $W_{1}(t)$, $t\in[t_{2m-1}^{1},t_{2m}^{1}]$, $m=1,2,3$, does not intersect with the union of $W_{i}(t)$, $t\in[t_{1}^{i},t_{2}^{i}]$, $i=2,3$. 
\par
By Lemma \ref{ni prob}, we have 
\[\mathbb{P}(F^{1})\asymp (jL)^{3}2^{-\xi(1,2)jL}.\] 
By the strong Markov property and Lemma~\ref{lem:harnack}, conditioned on $W_{i}[0,t_{i}]$, $i=1,2,3$, the joint law of $W_{i}(t)$, $t\in[t_{1}^{i},t_{2}^{i}]$, $i=1,2,3$, $W_{1}(t)$, $t\in[t_{3}^{1},t_{4}^{1}]$ and $W_{1}(t)$, $t\in[t_{5}^{1},t_{6}^{1}]$ is uniformly equivalent to five independent Brownian motions starting uniformly from one side of the annulus $\mathcal{A}(0,2^{jL+1/3},2^{(j+1)L-1/3})$, and stopped upon hitting the other side of the annulus. Then, by Lemma \ref{sep PDCP}, we have 
\[\mathbb{P}(F^{3}|F^{1})\leq c_{1}2^{-\xi(3,2)L}.\] 
By the strong Markov property and Lemma~\ref{lem:harnack}, conditioned on $W_{1}[0,t_{6}^{1}]$ and $W_{i}[0,t_{2}^{i}]$, $i=2,3$, the joint law of $W_{i}[t'_{i},\tau_{i}]$, $i=1,2,3$ is uniformly equivalent to three independent Brownian motions starting uniformly from $\mathcal{C}_{(j+1)L}$ and stopped upon hitting $\mathcal{C}_{n}$. Then, by Lemma \ref{sep PDCP}, we have 
\[\mathbb{P}(F^{2}|F^{1},F^{3})\leq c_{2}2^{-\xi(1,2)(n-(j+1)L)}.\]
Therefore we conclude that
\begin{equation}\label{local-eq1}
\begin{split}
\mathbb{P}(E_{1,j}\cap E_{1})&\leq\mathbb{P}(F^{1}\cap F^{2}\cap F^{3})\\
&\leq c_{3}(jL)^{3}2^{-\xi(1,2)jL}2^{-\xi(1,2)(n-(j+1)L)}2^{-\xi(3,2)L}\leq c_{3}n^{3}2^{-L}2^{-2n}
\end{split}
\end{equation}
since $\xi(1,2)=2$, while $\xi(3,2)>3$. For the event $E_{m,j}(m=2,3)$, one can perform similar decomposition with respect to $W_{m}$, and a variant of \eqref{local-eq1} follows by replacing the exponent $\xi(3,2)$ by $\xi(1,4)$ in consideration. The result follows since $\xi(1,4)>3$.
\par
Let us now consider the $d=3$ case. We still consider $\mathbb{P}(E_{1,j}\cap E_{1})$ and recall the decomposition and events $F^{1},F^{2},F^{3}$ above. Note that, for a 3D Brownian motion starting from $\mathcal{C}_{r+m}$, the probability that it hits $\mathcal{C}_{m}$ is $\asymp 2^{-r}$. Thanks to the strong Markov property and Lemma~\ref{lem:harnack}, it follows that
\begin{align*}
\mathbb{P}(E_{1,j}\cap E_{1})&\leq\mathbb{P}(F^{1})\mathbb{P}(F^{3}|F^{1})\mathbb{P}(F^{2}|F^{1},F^{3})\\
&\leq c_{4}2^{-\xi_{[3]}(1,2)jL}2^{-\xi_{[3]}(1,2)(n-(j+1)L)}2^{-L}2^{-\xi_{[3]}(3,2)L}\\
&\leq c_{4}2^{-L}2^{-n},
\end{align*}
since $\xi_{[3]}(1,2)=1$, and $\xi_{[3]}(3,2)\geq1$ by monotonicity of intersection exponents, which completes the proof in this case. For $\mathbb{P}(E_{m,j}\cap E_{1})$ with $m=2,3$, the result follows similarly since $\xi_{[3]}(1,4)\geq\xi_{[3]}(1,2)=1$ by monotonicity of intersection exponents.
\end{proof}
As a direct consequence of Lemma \ref{local12}, we have:
\begin{coro}\label{local arm12}
Let $E_{2}$ be the union of $E_{i,j}(i=1,2,3,1\leq j\leq [n/L]-1)$. We then have
\[\mathbb{P}(E_{2}\cap E_{1})\leq c2^{-\sqrt{n}/2}\cdot2^{-(4-d)n}.\]
\end{coro}
We recall the path decomposition \eqref{decomposition}--\eqref{excursion and bridge} (note that the decomposition is valid in both $d=2,3$), but we will only consider $X^{0},X^{1},X^{2},X^{3}$ and $W^{1},W^{2},W^{3}$ in the PDCP case. From now on, we write $\mathfrak{S}_{n}:=\mathfrak{S}_{n}(\delta)$ for brevity. Corollary \ref{local arm12} now leads to the following lemma:
\begin{lemma}\label{no extra crossing12}
For $i=1,2,3$ and $1\leq j\leq [3L/4]$, let $\widehat{E}_{i,j}$ denote the event that $W^{i}$ returns to $\mathcal{D}_{jL}$ after hitting $\mathcal{C}_{(j+1)L}$. Let $B_{1}$ be the union of $\widehat{E}_{i,j}$ over all $i$ and $j$. We then have
\[\mathbb{P}(B_{1}|S\in\mathfrak{S}_{n})\leq c2^{-\sqrt{n}/2}.\]
\end{lemma}
\begin{proof}
We first consider the $d=2$ case. Note that $S\in\mathfrak{S}_{n}$ implies that
\par
($F^{4}$) $W^{1}\cap(W^{2}\cup W^{3})=\varnothing$.
\par
By Lemma \ref{fme} for PDCP, we have
\begin{equation}\label{551}
\mathbb{P}(S\in\mathfrak{S}_{n})\asymp 2^{-2n}.
\end{equation}
By Lemmas \ref{ni prob} and \ref{ac excursions}, we have
\begin{equation}\label{f4}
\mathbb{P}(F^{4})\asymp n^{3}2^{-2n}.
\end{equation}
Thanks to Corollary \ref{local arm12} and Lemma \ref{ac excursions}, we conclude that
\begin{equation}\label{553}
\mathbb{P}(B_{1}\cap\{S\in\mathfrak{S}_{n}\})\leq\mathbb{P}(B_{1}\cap F^{4})\leq c\mathbb{P}(\widetilde{E}_{2}\cap\widetilde{E}_{1})\leq c'2^{-\sqrt{n}/2}2^{-2n},
\end{equation}
where $\widetilde{E}_{1},\widetilde{E}_{2}$ are events $E_{1},E_{2}$ from Lemma \ref{local12} and Corollary \ref{local arm12} under suitable translation and rescaling (which maps $\mathcal{A}(0,1,2^{n-K})$ to $\mathcal{A}(S,2^{-n},2^{-K})$). Combining \eqref{551} and \eqref{553}, the claim follows.
\par
For $d=3$, note that $S\in\mathfrak{S}_{n}$ implies $F^{4}$ and the following event: 
\par
($F^{5}$) $V_{\delta}^{\mathrm{PDCP}}(S)$ does occur, which ensures the existence of $W^{1}$ and $W^{3}$. 
\par
Again, by Lemma \ref{fme} for PDCP, we have
\begin{equation}\label{554}\mathbb{P}(S\in\mathfrak{S}_{n})\asymp 2^{-3n}.\end{equation}
By Lemmas \ref{ni prob} and \ref{ac excursions}, we have
\[\mathbb{P}(F^{4}|F^{5})\asymp 2^{-n}.\]
By standard estimate on hitting probability of Brownian motion, we have
\[\mathbb{P}(F^{5})\asymp 2^{-2n}.\]
We thus have
\begin{equation}\label{555}
\begin{split}
\mathbb{P}(B_{1}\cap\{S\in\mathfrak{S}_{n}\})&\leq\mathbb{P}(B_{1}\cap F^{4}\cap F^{5})\\
&=\mathbb{P}(B_{1}\cap F^{4}|F^{5})\mathbb{P}(F^{5})\\
&\leq c''\mathbb{P}(\widetilde{E}_{1}\cap\widetilde{E}_{2})\cdot 2^{-2n}\leq c'''2^{-\sqrt{n}/2}2^{-3n},
\end{split}
\end{equation}
where $\widetilde{E}_{1},\widetilde{E}_{2}$ are events $E_{1},E_{2}$ from Lemma \ref{local12} and Corollary \ref{local arm12} under suitable translation and rescaling (which maps $\mathcal{A}(0,1,2^{n-K})$ to $\mathcal{A}(S,2^{-n},2^{-K})$). Combining \eqref{554} and \eqref{555}, we complete the proof of the $d=3$ case.
\end{proof}
Similar to Lemma \ref{local bridge}, the following lemma says that the probability that bridges $X^{i}$ for $i=0,1,2,3$ cross multiple scales decays exponentially.
\begin{lemma}\label{bridge12}
For $i=0,2$, let $E^{i}$ denote the event that $X^{i}$ intersects $\mathcal{C}_{-K-L}(S)$, and for $i=1,3$, let $E^{i}$ denote the event that $X^{i}$ intersects $\mathcal{C}_{-n+L}(S)$. Then there exist $c,c'$, such that for any $i=0,1,2,3$, we have:
\[\mathbb{P}(S\in\mathfrak{S}_{n},E^{i})\leq ce^{-c'\sqrt{n}}2^{-dn}.\]
\end{lemma}
\begin{proof}
We first consider the $d=2$ case. Without loss of generality, we only give a proof for the event $E^{1}$ by considering $X^{1}$. On $\{S\in\mathfrak{S}_{n}\}$, if $X^{1}$ disconnect $\mathcal{C}_{-n}(S)$ from $\mathcal{C}_{-n+L}(S)$, then $X^{1}$ must intersect all three excursions, which destroy global openings. Note that $\{S\in\mathfrak{S}_{n}\}\cap E^{1}$ implies $F^{4}$ and the following event:
\par
($F^{6}$)\quad $X^{1}\cap\mathcal{C}_{-n+L}(S)\neq\varnothing$, and $X^{1}$ does not disconnect $\mathcal{C}_{-n}(S)$ from $\mathcal{C}_{-n+L}(S)$.
\par
By Lemma \ref{ac excursions}, conditioned on $W^{1},W^{2},W^{3}$, $X^{1}$ is distributed as Brownian bridge of variable duration from $W(t_{1})$ to $W(s_{2})$ conditioned on hitting $S$. Therefore, the law of $X^{1}$ between its first hitting of $\mathcal{C}_{-n+1}(S)$ and $\mathcal{C}_{-n+L-1}(S)$ is uniformly equivalent to a Brownian motion. By bounds of non-disconnecting probability of one Brownian motion (Definition \ref{de} with $k=1$), we have
\[\mathbb{P}(F^{6}|F^{4})\leq c2^{-L/4}.\]
Combining this and \eqref{f4}, we complete the proof of the $d=2$ case.
\par
For the $d=3$ case, we consider a similar decomposition. Note that for a 3D Brownian motion starting from $\mathcal{C}_{-n+L}(S)$, the probability that it hits $\mathcal{C}_{-n}(S)$ is $\asymp 2^{-L}$, we thus conclude that
\[\mathbb{P}(S\in\mathfrak{S}_{n},E^{i})\leq c'2^{-L}2^{-3n},\]
which completes the proof.
\end{proof}
Recall that $K=-\log_{2}(\delta/2)$ is a positive integer. Similar to \eqref{filtration} and \eqref{eq:U_m} in the previous section, we now define an increasing family of $\sigma$-fields
\begin{equation}\label{filtration12}\mathscr{F}_{m}:=\sigma\left(X^{i},\;i=0,1,2,3,\;W^{j}[0,\tau_{j,mL}],\;W^{j}[\sigma_{j,n-K-L},\tau_{j,n-K}],\;j=1,2,3\right)\end{equation}
and let
\[U_{m}:=\left\{X^{i},\;i=0,1,2,3,\;W^{j}[0,\tau_{j,m}],\;W^{j}[\sigma_{j,n-K-L},\tau_{j,n-K}],\;j=1,2,3\right\}\]
be the collection of curves in consideration. 
\par
For any $j\leq L/2$, if $d=2$, let $\widetilde{E}_{j}^{2}$ be the event that $U_{jL+1}$ does not disconnect $\mathcal{C}_{-n}(S)$ from infinity, and there exist two different connected components of the complement of $U_{jL+1}$ in $\mathcal{D}_{-n+jL+1}(S)$, such that both components intersect $\mathcal{D}_{-n}(S)$ and are connected to infinity. If $d=3$, let $\widetilde{E}_{j}^{3}$ be the event that $W^{1}[0,\tau_{1,jL+1}]$ does not intersect with $W^{2}[0,\tau_{2,jL+1}]\cup W^{3}[0,\tau_{3,jL+1}]$.
\par
Then, for $d=2,3$, it is clear that $\mathbb{P}(\widetilde{E}_{j}^{d}|U_{jL})$ is a random variable in $\mathscr{F}_{j}$. 
The following lemma shows that the opening on each scale are expected to be not very bad, which is parallel to Lemma \ref{all good opening} and thus we omit the proof here. For $1\leq j\leq[3L/4]$, Let 
\[\overline{E}_{j}^{d}=\{\mathbb{P}(\widetilde{E}_{j}^{d}|U_{jL})\leq n^{-10}\}\]
\begin{lemma}\label{all good opening12-2}
Then there exists $c>0$, such that for any $1\leq j\leq [3L/4]$, we have
\[\mathbb{P}(\overline{E}_{j}^{d}\cap\{S\in\mathfrak{S}_{n}\})\leq cn^{-4}2^{-dn}.\]
Consequently, let $B_{3}=B_{3}(d)$ be the union of $\overline{E}_{j}^{d}\ (1\leq j\leq [3L/4])$, then 
\[\mathbb{P}(B_{3}|S\in\mathfrak{S}_{n})\leq cn^{-3}.\]
\end{lemma}
By the PDCP version of Lemma \ref{delta+ good box}, we can show that Lemmas~\ref{no extra crossing12}--\ref{all good opening12-2} still hold when we replace $\mathfrak{S}_{n}=\mathfrak{S}_{n}(\delta)$ by $\mathfrak{S}_{n}(\delta+2^{\sqrt{n}})$.
\par
For $S\in\mathfrak{S}_{n}(\delta+2^{\sqrt{n}})$ and any integer $1\leq i\leq[L/2]$, we write 
\[\mathcal{A}_{i}:=\mathcal{A}(S,2^{-n+iL},2^{-n+(i+1)L})\]
for short, and use $N_{i}$ to denote the number of $\delta$-good boxes in $\mathcal{A}_{i}$. For $T\in\mathfrak{S}_{n}(\delta)$, if
\[\tau(T,\mathcal{C}_{-K}(S),T)<\tau(S,\mathcal{C}_{-K}(S),S),\]
we say $T$ is ``globally'' visited twice  before the second global visit of $S$, which we denote by $T\lessdot_{2}S$. Let $\widetilde{N}_{i}$ stand for the number of good boxes $T$ such that $T\prec\mathcal{A}_{i}$ and $T\lessdot_{2}S$.
\par
Our next lemma shows that $\widetilde{N}_{i}$ is ``almost'' adapted to $\mathscr{F}_{i}$ under the probability measure $\mathbb{P}(\cdot|S\in\mathfrak{S}_{n})$. More precisely, we have:
\begin{lemma}\label{adapted12}
For any integer $1\leq i\leq[L/2]$, on the event $\{S\in\mathfrak{S}_{n}(\delta+2^{\sqrt{n}})\}\cap(\cap_{j=1}^{3}B_{j}^{c})$, whether an $n$-box $T\prec\mathcal{A}_{i}$ is $\delta$-good and satisfies $T\lessdot_{2}S$ can be fully determined by $U_{(i+3)L}$.
\end{lemma}
As a direct consequence, $\widetilde{N}_{i}$ is $\mathscr{F}_{(i+3)}$-measurable under the probability measure 
\begin{equation}\label{conditional measure12}
\overline{\mathbb{P}}_{n}(\cdot):=\mathbb{P}(\cdot|S\in\mathfrak{S}_{n}(\delta+2^{-\sqrt{n}}),\cap_{j=1}^{3}B_{j}^{c}).
\end{equation}
We omit the proof here since it follows from a similar geometrical observation as in the proof of Lemma \ref{adapted}. When $d=2$, the only difference is that there are two different global openings here. When $d=3$, it suffices to consider the non-intersection of paths when we determine whether $T$ is $\delta$-good, which is even simpler than the 2D case.
\par
Parallel to Proposition \ref{exist good box}, we can prove the following proposition, in which we use the separation lemma with initial configuration for PDCP as the key ingredient.
\begin{prop}\label{exist good box12}
There exist $C_{1},C_{2}$, such that for any large $n$ and any $2\leq i\leq[L/2]$, we have
\[\overline{\mathbb{P}}_{n}(\widetilde{N}_{i}\geq C_{1}L|\mathscr{F}_{i-1})\geq C_{2}/i,\]
where $\overline{\mathbb{P}}_{n}$ is defined in \eqref{conditional measure12}.
\end{prop}
Proposition \ref{exist good box12} now leads to the proof of Proposition \ref{cgs12}, and we omit details here since it can be proved in the same fashion as Proposition \ref{cgs}. Once we complete the proof of Proposition \ref{cgs12}, the proof of Theorem \ref{pdcp} follows by the same argument as the proof of Theorem \ref{ptp}.

\section{Non-Existence of Boundary Double Points on Critical Brownian Loop-Soup Clusters}\label{sec:bdp}
In this section, we will prove Theorem \ref{bdp}. Again, we will skip many identical arguments that already appeared in the proof of Theorem \ref{ptp}. We point out that notation in this section is again independent from the previous ones. We still let $L=[\sqrt{n}]$ and assume that $K:=-\log_{2}(\delta/2)$ is a positive integer. We also fix $\iota>0$, and assume $S\in\mathcal{S}_{n}$, $S\subset\mathcal{A}(0,\iota,1-\iota)$.
\par
As will be discussed in the proof of Theorem \ref{bdp}, we will focus on double points which are generated by one single loop. All lemmas and propositions have parallel versions when we consider double points visited by two different loops. 

We write $\Gamma_{0}$ for the Brownian loop soup in the unit disc with intensity $1$, and $\gamma$ for a Brownian loop independently sampled from $\mu_{\iota}^{\#}$ (recall \eqref{eq:mu_iota} for its definition). 
Recall that $\mathfrak{S}_{n}(\delta)$ is the set of $\delta$-good boxes from Definition~\ref{gs4}, and recall moment bounds in Lemmas \ref{fme}, \ref{sme} and \ref{tme} for BDP. 
The following proposition is parallel to Propositions \ref{cgs} and \ref{cgs12}, leading to the proof of Theorem \ref{bdp}.
\begin{prop}\label{cgs4}
There exist $c,\lambda,\nu>0$, such that for any large $n$,
\[\mathbb{P}(\#\mathfrak{S}_{n}<\lambda\sqrt{n}|S\in\mathfrak{S}_{n}(\delta+2^{-\sqrt{n}}))\leq cn^{-\nu}.\]
\end{prop}
We also have a parallel version of Lemma \ref{delta+ good box} for BDP that controls the difference between $\mathfrak{S}_{n}(\delta+2^{-\sqrt{n}})$ and $\mathfrak{S}_{n}(\delta)$. We omit the statement and proof.
\par
Recall that $\Gamma$ is a Brownian loop soup in the whole plane with intensity 1. The following two lemmas characterize the global opening structure and the extra crossing structure, which are parallel versions of Lemmas \ref{unique} and \ref{local}. Let $W_{i}$, $i=1,2,3,4$ be 4 independent Brownian excursions in $\mathcal{A}(0,1,2^n)$ which start uniformly on $\mathcal{C}_{0}$, and stop upon reaching $\mathcal{C}_{n}$. Let $\Gamma_{r}$ stand for the collection of loops in $\Gamma$ which are contained in $\mathcal{B}(0,r)$. Let $\widetilde{W}_{i}$ stand for the union of $W_{i}$ and clusters of loops in $\Gamma_{2^{n}}\backslash\Gamma_{1}$ that it intersects. Let $E_{1}$ be the event that the union of $\widetilde{W}_{i}$ for $i=1,2,3,4$ does not disconnect $\mathcal{C}_{0}$ from infinity. Recall the definition of $\sigma_{i}^{j}$ and $\tau_{i}^{j}$ above Lemma~\ref{unique}. For any $j\leq n/2$, let $E_{2}=E_{2}(j)$ (resp.\ $E_{3}=E_{3}(j)$) be the event that there exists non-empty $A\subsetneq\{1,2,3,4\}$, such that packets of Brownian excursions $\{W_{i}[0,\tau_{i}^{j}]:i\in A\}$ and $\{W_{i}[0,\tau_{i}^{j}]:i\in A^{c}\}$ do not intersect (resp.\ $\{W_{i}[\sigma_{i}^{n-j},\tau_{i}^{n}]:i\in A\}$ and $\{W_{i}[\sigma_{i}^{n-j},\tau_{i}^{n}]:i\in A^{c}\}$ do not intersect). 
\par
In a similar way as for Lemma \ref{unique}, we can obtain the following result, where inputs are now Lemmas~\ref{gnd prob} and \ref{sep BDP}. 
\begin{lemma}\label{unique4}
There exists $c>0$, such that for $m=2,3$, $n\ge 1$ and $j\le n/2$,
\[\mathbb{P}(E_{m}(j)\cap E_{1})\leq cn^{8}2^{-2n-j/2}.\]
\end{lemma}
\begin{proof}
We will only prove the $m=2$ case and the $m=3$ case follows from inversion invariance. Note that $E_{2}\cap E_{1}$ implies the following two events:
\par
($F^{1}$) The union of $\widetilde{W}_{i}[\sigma_{i}^{j+1},\tau_{i}^{n}]$, $i=1,\ldots,4$, does not disconnect $\mathcal{C}_{j+1}$ from infinity, where $\widetilde{W}_{i}[\sigma_{i}^{j+1},\tau_{i}^{n}]$ denotes the union of $W_{i}[\sigma_{i}^{j+1},\tau_{i}^{n}]$ and loop clusters in $\Gamma_{2^{n}}\backslash\Gamma_{2^{j+1}}$ that it intersects.
\par
($F^{2}$) For some non-empty $A\subsetneq\{1,2,3,4\}$, $\{W_{i}[0,\tau_{i}^{j}], i\in A\}$ and $\{W_{i}[0,\tau_{i}^{j}], i\in A^{c}\}$ do not intersect.
\par
By the strong Markov property and \cite[Lemma 2.15]{GLQ22}, conditioned on $W_{i}[0,\tau_{i}^{j}]$, $i=1,2,3,4$, the joint law of $W_{i}[\sigma_{i}^{j+1},\tau_{i}^{n}]$, $i=1,2,3,4$, is uniformly equivalent to the joint law of four independent Brownian excursions starting uniformly from $\mathcal{C}_{j+1}$ and stopped upon hitting $\mathcal{C}_{n}$. By Lemma \ref{gnd prob}, since $\xi_1(4)=2$ by \eqref{gde value}, it follows that
\[\mathbb{P}(F^{1}|F^{2})\leq c_{1}(n-j)^{4}2^{-2(n-j)}.\] 
Now, for all three cases $|A|=i$, $i=1,2,3$, in the annulus $\mathcal{A}(0,1,2^{j})$, $F^{2}$ implies the non-intersection event of $i$ Brownian excursions and the other $(4-i)$ Brownian excursions. Since $\xi(1,3)>5/2$ and $\xi(2,2)>5/2$, by Lemma \ref{ni prob}, we have 
\[\mathbb{P}(F^{2})\leq c_{2}j^{4}2^{-5j/2}.\] 
We thus have 
\[\mathbb{P}(F^{1}\cap F^{2})\leq c_{3}(n-j)^{4}j^{4}2^{-2n-j/2}\leq c_{3}n^{8}2^{-2n-j/2},\] 
which completes the proof.
\end{proof}
\begin{lemma}\label{local4}
For $i=1,2,3,4$, let $E_{i,j}$ be the event that $W_{i}$ returns to $\mathcal{C}_{jL}$ after hitting $\mathcal{C}_{(j+1)L}$. Then there exists $c$, such that for any $n\ge 1$, $1\leq j\leq[3L/4]$, and $i\in\{1,2,3,4\}$, we have
\[\mathbb{P}(E_{i,j}\cap E_{1})\leq cn^{4}2^{-L}\cdot2^{-2n}.\]
Moreover, letting $E_{4}$ be the union of $E_{i,j}$ over all $i$ and $j$, we have
\[\mathbb{P}(E_{4}\cap E_{1})\leq ce^{-c'\sqrt{n}}2^{-2n}.\]
\end{lemma}
\begin{proof}
It suffices to consider the $i=1$ case. On the event $E_{1,j}$, we recall stopping times \eqref{461}--\eqref{463}. Then $E_{1,j}\cap E_{1}$ implies the following events:
\par
($F^{3}$) The union of $\widetilde{W}_{i}[0,t_{i}]$, $i=1,\ldots,5$, does not disconnect $\mathcal{D}_{0}$ from infinity.
\par
($F^{4}$) The union of $\widetilde{W}_{i}[t'_{i},\tau_{i}]$, $i=1,\ldots,5$, does not disconnect $\mathcal{C}_{(j+1)L}$ from infinity.
\par
($F^{5}$) The union of $\widetilde{W}_{i}(t)$, $t\in[t_{1}^{i},t_{2}^{i}]$, $i=1,\ldots,4$, $\widetilde{W}_{1}(t)$, $t\in[t_{3}^{1},t_{4}^{1}]$  and $\widetilde{W}_{1}(t)$, $t\in[t_{5}^{1},t_{6}^{1}]$ does not disconnect $\mathcal{C}_{jL+1/3}$ from infinity.
\par
By Lemma \ref{gnd prob}, we have
\[\mathbb{P}(F^{3})\leq c_{1}(jL)^{4}2^{-\xi_{1}(4)jL}.\]
By the strong Markov property and Lemma \ref{sep BDP}, we have
\[\mathbb{P}(F^{5}|F^{3})\leq c_{2}2^{-\xi_{1}(6)L}.\]
By the strong Markov property, Lemmas \ref{last exit} and \ref{sep BDP}, we have
\[\mathbb{P}(F^{4}|F^{3},F^{5})\leq c_{3}2^{-\xi_{1}(4)(n-(j+1)L)}.\]
Therefore we conclude that
\[\mathbb{P}(E_{1,j}\cap E_{1})\leq\mathbb{P}(F^{3}\cap F^{4}\cap F^{5})\leq c_{4}(jL)^{4}2^{-(\xi_{1}(6)-\xi_{1}(4))L}\cdot2^{-\xi_{1}(4)n}\leq c_{4}n^{4}2^{-L}\cdot2^{-2n},\]
since $\xi_{1}(6)=3,\xi_{1}(4)=2$, which completes the proof.
\end{proof}
Recall parameters $\iota\in (0,1)$, $\delta\in (0,\iota/2)$ and $S\subset\mathcal{A}(0,\iota,1-\iota)$.
Also recall that $\gamma$ is sampled from $\mu_{\iota}^{\#}$ (see \eqref{eq:mu_iota}). Let $re^{i\theta}$ be the point on $\gamma$ which is farthest from $0$. In the following, we further condition on $r$ and $\theta$. Given $r$ and $\theta$, $\gamma$ is now a Brownian loop sampled according to the bubble measure $\mu_{\mathcal{B}(0,r)}^{\mathrm{bub}}(re^{i\theta})$ conditioned on $\mathrm{diam}(\gamma)>\iota/2$. Note that normalized constants here are uniform in $r,\theta$ for given $\iota$, so it is safe to work with the normalized conditional bubble measure. Below, all constants can be made uniform in $r\in [\iota,1-\iota]$ and $\theta\in[0,2\pi]$.
\par
If $S$ is close to the root of $\gamma$, we can adapt our arguments below by rerooting $\gamma$ to the farthest point from $S$. Hence, it is safe to consider only the case that $\mathrm{dist}(re^{i\theta},S)\ge\delta$ below.
\par
We now recall the excursion-bridge decomposition of a Brownian loop $\gamma$ in \cite[Section 5.2]{GLQ22}. Let $u_{1}=\tau(\partial\mathcal{B}(S,\delta-2^{-n-1/2}))$, and
\begin{equation}\label{decomposition of loop}
\begin{split}
&v_{i}=\tau_{\gamma}(u_{i},S),\quad u_{i+1}=\tau_{\gamma}(v_{i},\partial\mathcal{B}(S,\delta-2^{-n-1/2})),\\
&s_{2i-1}=\sup\{t<v_{2i-1}:\gamma_{t}\in\partial\mathcal{B}(S,\delta/2)\},\quad t_{2i-1}=\tau_{\gamma}(s_{2i-1},\mathcal{C}_{-n}(S)),\\
&s_{2i}=\sup\{t<u_{2i+1}:\gamma_{t}\in\mathcal{C}_{-n}(S)\},\quad t_{2i}=\tau_{\gamma}(s_{2i},\partial\mathcal{B}(S,\delta/2)).\\
\end{split}
\end{equation}
We then define $W^{1}=\gamma[s_{1},t_{1}]$, and for $i=1,2,3$ and $j=1,2$,
\begin{equation}\label{excursion and bridge of loop}
X^{i}=W[t_{i},s_{i+1}],\quad W^{2j-1}=(W[s_{2j-1},t_{2j-1}])^{\mathcal{R}},\quad W^{2j}=W[s_{2j},t_{2j}].
\end{equation}
Finally, we let $X^{4}$ be the concatenation of $\gamma[t_{4},t_{\gamma}]$ (recall that $t_{\gamma}$ is the duration of $\gamma$) and $\gamma[0,s_{1}]$.
\par
Thanks to \cite[Lemma 5.4]{GLQ22}, the joint law of $W^{1},\ldots,W^{4}$ is uniformly equivalent to that of four independent Brownian excursions. Therefore, results in Lemmas \ref{unique4} and \ref{local4} can be directly applied to events with respect to $W^{1},\ldots,W^{4}$ after suitable rescaling. More precisely, we have
\begin{lemma}\label{unique opening4}
Let $B_{1}$ (resp.\ $B_{2}$) stand for the event that there exists non-empty $A\subsetneq\{1,2,3,4\}$, such that $\{W^{i}:i\in A\}$ and $\{W^{i}:i\in A^{c}\}$ do not intersect before their first hitting at $\mathcal{C}_{-n+L}(S)$ (resp.\ after their last hitting at $\mathcal{C}_{-K-L}(S)$). Then there exist  $c,c'$, such that
\[\mathbb{P}(B_{1}\cup B_{2}|S\in\mathfrak{S}_{n})\leq ce^{-c'\sqrt{n}},\]
where we use $\mathbb{P}$ to denote the product probability measure of $\mu_{\iota}^{\#}$ and an independent critical Brownian loop soup from now on.
\end{lemma}
\begin{lemma}\label{no extra crossing4}
For $i=1,2,3,4$, let $E_{i,j}$ stand for the event that $W^{i}$ returns to $\mathcal{C}_{-n+jL}(S)$ after hitting $\mathcal{C}_{-n+(j+1)L}(S)$. Let
\[B_{3}:=\bigcup_{i=1}^{4}\bigcup_{j=1}^{[3L/4]}E_{i,j}.\]
Then, there exist $c,c'$, such that
\[\mathbb{P}(B_{3}|S\in\mathfrak{S}_{n})\leq ce^{-c'\sqrt{n}}.\]
\end{lemma}
Our next lemma concerns the probability that ``bridges'' $X^{i}$'s cross multiple scales. We show that such probability decays exponentially with respect to the number of scales, parallel to Lemma \ref{local bridge}. We hence omit the proof.
\begin{lemma}\label{local bridge4}
For $i=2,4$, let $E^{i}$ stand for the event that $X^{i}$ intersects $\mathcal{C}_{-K-L}(S)$, and for $i=1,3$, let $E^{i}$ stand for the event that $X^{i}$ intersects $\mathcal{C}_{-n+L}(S)$. Then there exist $c,c'>0$, such that for any $i=1,2,3,4$, we have:
\[\mathbb{P}(S\in\mathfrak{S}_{n},E^{i})\leq ce^{-c'\sqrt{n}}2^{-2n}.\]
Consequently, let $B_{4}$ be the union of $E^{i}$, then there exists $c''>0$, such that
\[\mathbb{P}(B_{4}|S\in\mathfrak{S}_{n})\leq c''e^{-c'\sqrt{n}}.\]
\end{lemma}
We additionally need the following lemma dealing with the probability that there exists a loop-soup cluster that crosses multiple scales without disconnecting $\mathcal{C}_{-n}(S)$ from infinity. We show that such probability decays exponentially with respect to the number of scales.
\par
For any $1\leq i<n-K$ and $j\geq L$ such that $i+j\leq n-K-1$, let $\widetilde{E}_{i,j}$ be the event that there exists a loop cluster $\mathcal{K}$ of $\Gamma_{0}$, such that 
\begin{enumerate}
\item $\mathcal{K}$ does not disconnect $S$ from infinity,
\item $\mathcal{K}\cap\mathcal{C}_{-n+i}(S)\neq\varnothing$,
\item $\mathcal{K}\cap\mathcal{C}_{-n+i+j}(S)\neq\varnothing$,
\item $\mathcal{K}\cap\mathcal{C}_{-n+i-1}(S)=\varnothing$, and 
\item $\mathcal{K}\cap\mathcal{C}_{-n+i+j+1}(S)=\varnothing$.
\end{enumerate} If $i=0$ (resp.\ $i+j=N-K$), define $\widetilde{E}_{i,j}$ similarly  without the fourth (resp.\ fifth) condition.
\begin{lemma}\label{local cluster}
There exist $c,c'>0$, such that
\[\mathbb{P}(\widetilde{E}_{i,j}\cap\{S\in\mathfrak{S}_{n}\})\leq ce^{-c'j}2^{-2n}.\]
Consequently, let $B_{5}$ be the union of $\widetilde{E}_{i,j}$, which stands for the event that there exists a cluster $\mathcal{K}$ and an integer $0\leq i\leq n-K-L$, such that $\mathcal{K}\cap\mathcal{C}_{-n+i}(S)\neq\varnothing,\mathcal{K}\cap\mathcal{C}_{-n+i+L}(S)\neq\varnothing$. Then there exist $c,c'$, such that
\[\mathbb{P}(B_{5}|S\in\mathfrak{S}_{n})\leq ce^{-c''\sqrt{n}}.\]
\end{lemma}
\begin{proof}
We will prove the lemma for $i\geq1$ and $i+j\leq N-K-1$, as the proof of other cases is very similar (and even simpler). Let $\tau_{k}(A)$ be the first hitting time of $A$ with respect to $W^{k}$. Let $\Gamma_{r}(S)$ be the collection of loops in $\Gamma$ which is fully contained in $\mathcal{D}_{r}(S)$. On the event $\widetilde{E}_{i,j}\cap\{S\in\mathfrak{S}_{n}\}$, the following events occur:
\par
($F^{6}$) The union of $\widetilde{W}^{k}[0,\tau_{k}(\mathcal{C}_{-n+i-1}(S))]$, $k=1,2,3,4$, (where $\widetilde{W}^{k}$ is the union of $W^{k}$ together with loop-soup clusters of $\Gamma_{-n+i}(S)\backslash\Gamma_{-n}(S)$) does not disconnect $\mathcal{C}_{-n}(S)$ from infinity.
\par
($F^{7}$) The union of $W^{k}[\tau_{k}(\mathcal{C}_{-n+i}(S)),\tau_{k}(\mathcal{C}_{-n+i+j}(S))]$, $k=1,2,3,4$, does not disconnect $\mathcal{C}_{-n+i}(S)$ from infinity.
\par
($F^{8}$) The union of $\widetilde{W}^{k}[\tau_{k}(\mathcal{C}_{-n+i+j+1}(S)),\tau_{k}(\mathcal{C}_{-K}(S))](k=1,2,3,4)$ (where $\widetilde{W}^{k}$ is the union of $W^{k}$ together with loop-soup clusters of $\Gamma_{-K}(S)\backslash\Gamma_{-n+i+j-1}(S)$) does not disconnect $\mathcal{C}_{-n+i+j+1}(S)$ from infinity. 
\par
By \cite[Lemma 5.4]{GLQ22} and Lemma \ref{gnd prob}, we have
\[\mathbb{P}(F^{6})\leq c_{1}i^{4}2^{-\xi_{1}(4)i}.\]
By the strong Markov property, Lemmas~\ref{lem:harnack}, \ref{gnd prob} and \cite[Lemma 5.4]{GLQ22}, we have
\[\mathbb{P}(F^{7}|F^{6})\leq c_{2}j^{4}2^{-\xi(4)j}\]
and
\[\mathbb{P}(F^{8}|F^{6},F^{7})\leq c_{3}(n-i-j)^{4}2^{-\xi_{1}(4)(n-i-j)}.\]
By \cite[Proposition 3.2]{GNQ24+}, we have
\[\mathbb{P}(\widetilde{E}_{i,j})\leq c_{4}2^{-\alpha_{2}(4)j},\]
where $\alpha_{2}(4)=1/2$ is the two-arm exponent of $\mathrm{SLE}_{4}$ \cite{MR2153402}.
\par
Note that $\widetilde{E}_{i,j}$ is independent of $F^{6},F^{7},F^{8}$, we thus conclude that
\[\mathbb{P}(\widetilde{E}_{i,j}\cap\{S\in\mathfrak{S}_{n}\})\leq\mathbb{P}(\widetilde{E}_{i,j})\mathbb{P}(F^{6}\cap F^{7}\cap F^{8})\leq c_{5}n^{12}2^{-2n}2^{(2-\xi(4)-\alpha_{2}(4))j}.\]
The result therefore follows since $2-\xi(4)-1/2<0$.
\end{proof}
With a slight abuse of notation, we still write $\tau_{i,j}$ (resp.\ $\sigma_{i,j}$) for the first (resp.\ last) hitting time of $\mathcal{C}_{-n+j}(S)$ with respect to $W^{i}$.
We then define an increasing family of $\sigma$-algebras as follows:
\begin{equation}\label{filtration4}
\mathscr{F}_{m}:=\sigma\left(X^{j},\;W^{j}[0,\tau_{j,mL}],\;W^{j}[\tau_{j,n-K-L},\sigma_{j,n-K}],\;j=1,2,3,4,\;\Gamma_{-n+mL}(S)\right).
\end{equation}
Denote the associated collection of curves by
\[U_{m}:=\left\{X^{j},\;W^{j}[0,\tau_{j,m}],\;W^{j}[\tau_{j,n-K-L},\sigma_{j,n-K}],\;j=1,2,3,4\right\}.\]
Moreover, we write
\[\overline{U}_{m}:=U_m\cup\Gamma_{-n+m}(S)\]
for the collection of curves and loops.
\par
For any $j\leq[L/2]$, let $\widetilde{E}_{j}$ be the event that $\overline{U}_{jL+1}$ does not disconnect $\mathcal{C}_{-n}(S)$ from infinity. Then $\mathbb{P}(\widetilde{E}_{j}|\overline{U}_{jL})$ is a random variable in $\mathscr{F}_{j}$. Let
\[\overline{E}_{j}:=\{\mathbb{P}(\widetilde{E}_{j}|U_{jL})\leq n^{-10}\}.\]
The following lemma shows that the opening on each scale is expected to be not very bad, which is parallel to Lemma \ref{all good opening}, and thus we omit the proof here. 
\begin{lemma}\label{all good opening4}
There exists $c>0$, such that for any $1\leq j\leq[3L/4]$, we have
\[\mathbb{P}(\overline{E}_{j},S\in\mathfrak{S}_{n})\leq cn^{-4}2^{-2n}.\]
Moreover, let $B_{6}$ be the union of $\overline{E}_{j}(1\leq j\leq[L/2])$, then 
\[\mathbb{P}(B_{6}|S\in\mathfrak{S}_{n})\leq c'n^{-3}\]
for some $c'>0$.
\end{lemma}
By the BDP version of Lemma \ref{delta+ good box}, we can show that Lemmas~\ref{unique opening4}--\ref{all good opening4} still hold with $\mathfrak{S}_{n}(\delta+2^{-\sqrt{n}})$ in place of $\mathfrak{S}_{n}=\mathfrak{S}_{n}(\delta)$.
\par
For $S\in\mathfrak{S}_{n}(\delta+2^{-\sqrt{n}})$, and any integer $1\leq i\leq[L/2]$, we write
\[\mathcal{A}_{i}:=\mathcal{A}(S,2^{-n+iL},2^{-n+(i+1)L})\]
for short, and use $N_{i}$ to denote the number of $\delta$-good boxes $\prec\mathcal{A}_{i}$ (see Section \ref{subsec:notation}). The next lemma shows that $N_{i}$ is ``almost'' adapted to $\mathscr{F}_{i}$ under the probability measure $\mathbb{P}(\cdot|S\in\mathfrak{S}_{n})$, which is parallel to Lemmas \ref{adapted} and \ref{adapted12}.
\begin{lemma}\label{adapted4}
For any integer $1\leq i\leq[L/2]$, on the event $\{S\in\mathfrak{S}_{n}(\delta+2^{\sqrt{n}})\}\cap(\cap_{j=1}^{6}B_{j}^{c})$, whether a box $T\prec\mathcal{A}_{i}$ is $\delta$-good can be fully determined by $\overline{U}_{(i+4)L}$.
\par
As a direct consequence, $N_{i}$ is $\mathscr{F}_{(i+4)}$-measurable under the probability measure $\overline{\mathbb{P}}_{n}$, which is defined by
\begin{equation}\label{conditional measure4}
\overline{\mathbb{P}}_{n}(\cdot):=\mathbb{P}(\cdot|S\in\mathfrak{S}_{n}(\delta+2^{-\sqrt{n}}),\cap_{j=1}^{6}B_{j}^{c}).
\end{equation}
\end{lemma}
\begin{proof}
The proof is based on a similar geometrical observation as in the proof of Lemma \ref{adapted}. On $V_{\delta}^{\mathrm{BDP}}(S)$, we can naturally define ``excursions'' $W^{m}$ and ``bridges'' $X^{m}$.
\par
Similar to the proof of Lemma \ref{adapted}, on $B_{1}^{c}\cap B_{2}^{c}$, openings $\mathcal{O}(U_{L},\mathcal{A}(S,2^{-n},2^{-n+L}))$ and $\mathcal{O}(U_{L},\mathcal{A}(S,2^{-K-L},2^{-K}))$ both contain only one connected component. Recall that $\gamma$ is the Brownian loop that we consider throughout this section, which is sampled according to $\mu_{\iota}^{\#}$; see Definition \ref{gs4}. Therefore, the opening $\mathcal{O}(\widetilde{\gamma},\mathcal{A}(S,2^{-n},2^{-K}))$ has a unique connected component, where $\widetilde{\gamma}$ is the union of $\gamma$ together with loop clusters in $\Gamma_0$ that it intersects.
\par
Now, on the event $B_{3}^{c}$, the path configuration in $\mathcal{D}_{-n+(i+3)L}(S)$ is fully determined by $U_{(i+4)L}$. On the event $B_{5}^{c}$, loop clusters of $\Gamma_{0}$ that intersect $\mathcal{D}_{-n+(i+3)L}(S)$ is fully determined by $U_{(i+4)L}$ since there is no loop cluster surrounding $\mathcal{D}_{-n}(S)$ on the event $\{S\in\mathfrak{S}_{n}\}$.
\par
We then consider outer boundaries of loop-soup clusters in $\mathcal{A}(S,2^{-n},2^{-n+(i+3)L})$, which is fully determined by $\overline{U}_{(i+4)L}$ on $B_{5}^{c}$. If a cluster is not visited by $U_{(i+4)L}$, then it will not be taken into consideration when we determine the ``opening'' even if the corresponding cluster can be visited by the extension of $W^{m}$. More precisely, we use $\widetilde{U}_{(i+4)L}$ to denote the union of $U_{(i+4)L}$ together with (pieces of) loop clusters in $\mathcal{A}(S,2^{-n},2^{-n+(i+3)L})$ that it intersects. Then, the opening 
\[\mathcal{O}_{i+3}:=\mathcal{O}(\widetilde{U}_{(i+4)L}\cup\gamma,\mathcal{A}(S,2^{-n},2^{-K}),2^{-n+(i+3)L})\]
is fully determined by $\overline{U}_{(i+4)L}$. Moreover, $\mathcal{O}_{i+3}$ is indeed a continuous arc on $\mathcal{C}_{-n+(i+3)L}(S)$, following from a geometrical observation.
\par
Whether a box $T\prec\mathcal{A}_{i}$ is $\delta$-good will be determined as follows. First of all, we check whether $T$ is visited by two different $W^{m}[0,\tau_{m,(i+4)L}]$, which is fully determined by $U_{(i+4)L}$. If this holds, we claim that $T$ is good \textit{if and only if} $\mathcal{C}_{-n}(T)$ is not disconnected from $\mathcal{O}_{i+3}$ by $\widetilde{U}_{(i+4)L}$ in $\mathcal{D}_{-n+(i+3)L}(S)$, i.e., there exists a continuous curve $\eta=\eta(t)$, $t\in[0,1]$, such that $\eta(0)\in\mathcal{C}_{-n}(T)$, $\eta(1)\in\mathcal{O}_{i+3}$, $\eta\subset\mathcal{D}_{-n+(i+3)L}(S)$ and $\eta\cap(\gamma\cup\widetilde{U}_{(i+4)L})=\varnothing$.
\par
The ``if'' direction is trivial. For the ``only if'' direction (see Figure \ref{fig_adapted2} for an illustration of various objects), suppose that $T$ is good and $\mathcal{C}_{-n}(T)$ is disconnected from $\mathcal{O}_{i+3}$ in $\mathcal{D}_{-n+(i+3)L}(S)$. Therefore there exists at least one connected component of $\mathcal{D}_{-K}(S)\backslash\widetilde{\gamma}$ that intersects $\mathcal{C}_{-n}(T)$ and $\mathcal{C}_{-K}(S)$, denoted by $\mathcal{K}$. Define $\mathcal{O}'$ as the cut set of $\mathcal{K}$ on $\mathcal{C}_{-n+(i+3)L}(S)$ (with respect to the general annulus $\mathcal{D}_{-n}(S)\backslash\mathcal{D}_{-n}(T)$), we then have $\mathcal{O}_{i+3}\cap\mathcal{O}'=\varnothing$, and $\mathcal{O}'$ is a continuous arc on $\mathcal{C}_{-n+(i+3)L}(S)$. We now focus on end points of $\mathcal{O}'$ and $\mathcal{O}_{i+3}$, namely $P_{1},P_{2}$ and $P_{3},P_{4}$. For $k=1,2,3,4$, we assume that $P_{k}$ is visited by $\widetilde{W}^{j_{k}}$ (where $\widetilde{W}^{m}$ is the union of $W^{m}$ and loop clusters in $\Gamma_{0}$ that it intersects).
\par
As in the proof of Lemma \ref{adapted}, there exist continuous curves $\eta_{1},\eta_{2}$, such that $\eta_{1}$ starts from $\mathcal{C}_{-n}(S)$, passes through $\mathcal{O}_{i+3}$ and stops on $\mathcal{C}_{-K}(S)$, with $\eta_{1}\cap\widetilde{\gamma}=\varnothing$, and $\eta_{2}$ starts from $\mathcal{C}_{-n}(T)$, passes through $\mathcal{O}'$ and stops on $\mathcal{C}_{-K}(S)$, with $\eta_{2}\cap\widetilde{\gamma}=\varnothing$. On $B_{2}^{c}$, there exists a continuous curve $\eta_{3}$ in $\mathcal{A}(S,2^{-n+(i+3)L},2^{-K})$, such that $\eta_{3}$ connects $\eta_{1},\eta_{2}$ without intersecting $\widetilde{\gamma}$. We now consider the bounded domain $\mathfrak{D}$ enclosed by $\eta_{1},\eta_{2},\eta_{3}$ and $\mathcal{C}_{-n+(i+1)L}(S)$. Note that there exists $k$, such that $P_{k}\in\mathfrak{D}$. If $P_{k}$ is visited by $W^{j_{k}}$, then $W^{j_{k}}$ cannot escape from $\mathfrak{D}$ after hitting $P_{k}$ on the event $B_{3}^{c}$, which yields a contradiction since all $W^{m}$ end at $\mathcal{C}_{-K}(S)$. Therefore $P_{k}$ is visited by a loop cluster $\mathfrak{K}$ that intersects $W^{j_{k}}$. Note that on $B_{5}^{c}$, $\mathfrak{K}$ cannot intersect $\mathcal{C}_{-n+(i+1)L}(S)$. Also, we have $\mathfrak{K}\cap\eta_{l}=\varnothing,\ l=1,2,3$ since $\mathfrak{K}\subset\widetilde{\gamma}$. We thus conclude that $\mathfrak{K}\subset\mathfrak{D}$. Since the loop cluster $\mathfrak{K}$ cannot intersect $\mathcal{D}_{-n+(i+2)L}(S)$ on the event $B_{5}^{c}$, $W^{j_{k}}$ must pass through a point in $\mathfrak{D}\cap\mathcal{A}(S,2^{-n+(i+2)L},2^{-n+(i+3)L})$. Therefore, $W^{j_{k}}$ cannot escape from $\mathfrak{D}$ on the event $B_{3}^{c}$, which also yields a contradiction. This completes the proof of the claim.
\begin{figure}[t]
\centering
\includegraphics[width=7cm]{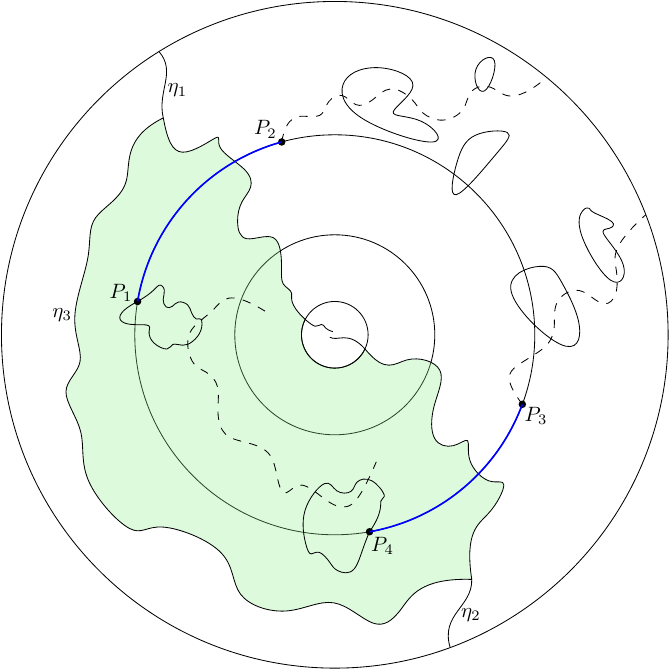}
\caption{\small An illustration for the proof of ``only if'' direction. Dashed curves are (pieces of) $W^{m}$. $\mathcal{O}_{i+2},\mathcal{O}'$ are in blue. The four circles in this figure are $\mathcal{C}_{-n+(i+1)L}(S)$, $\mathcal{C}_{-n+(i+2)L}(S)$, $\mathcal{C}_{-n+(i+3)L}(S)$ and $\mathcal{C}_{-K}(S)$. Then, after hitting $P_{1}$ or the loop cluster passing through $P_{1}$, $W^{j_{1}}$ is trapped in the bounded domain $\mathfrak{D}$.}
\label{fig_adapted2}
\end{figure}
\par
Finally, note that $\mathrm{dist}(S,T)\leq 2^{-n+(i+1)L}<2^{-\sqrt{n}}$, hence $\{S\in\mathfrak{S}_{n}(\delta+2^{-\sqrt{n}})\}$ implies that $\{T\in\mathfrak{S}_{n}(\delta)\}$. Note that everything mentioned above is fully determined by $\overline{U}_{(i+4)L}$. This completes the proof.
\end{proof}
Similar to Propositions \ref{exist good box} and \ref{exist good box12}, we can show that:
\begin{prop}\label{exist good box4}
There exist $C_{1},C_{2}$, such that for any large $n$ and for any $2\leq i\leq[L/2]$, we have
\[\overline{\mathbb{P}}_{n}(N_{i}\geq C_{1}L|\mathscr{F}_{i-1})\geq C_{2}/i,\]
where $\overline{\mathbb{P}}_{n}$ is defined in \eqref{conditional measure4}.
\end{prop}
With all ingredients ready, Proposition \ref{cgs4} can be proved in the same fashion as Propositions \ref{cgs} and \ref{cgs12}, and hence we omit details here.
\par
Proposition \ref{cgs4} now leads to the proof of Theorem \ref{bdp}.
\begin{proof}[Proof of Theorem \ref{bdp}]
We start by proving that there does not exist boundary double points on the outer boundary of outermost clusters. We first show the non-existence of such BDPs visited by a single loop.
\par
By decompositions of loop measures given in \cite[Section 4.2]{LW04}, $\mu_{\iota}$ (see \eqref{eq:mu_iota}) is just $\mu_{\mathbb{D}}^{\mathrm{loop}}$ (the Brownian loop measure in $\mathbb{D}$; see Section~\ref{subsec:bls}) restricted to loops $\gamma$ that satisfy $\sup_{0\le t\le t_{\gamma}}|\gamma_t|\in (\iota,1-\iota)$ and $\text{diam}(\gamma)>\iota/2$. Let $\Gamma^{\iota}$ denote loops in $\Gamma_{0}$ that satisfy the previous condition. Then, $\Gamma^{\iota}$ can be sampled by the following procedure: we first sample a Poisson random variable $\mathcal{N}_{\iota}$ with parameter $|\mu_{\iota}|$, and we then sample $\gamma_{i} (1\leq i\leq\mathcal{N}_{\iota})$ independently according to $\mu_{\iota}^{\#}$. Conditioned on $\{\mathcal{N}_{\iota}\geq1\}$ (which happens with probability $1-o(1)$, where $o(1)$ tends to 0 as $\iota\to0$), the law of $(\gamma_{1},\Gamma_{0}\backslash\{\gamma_{1}\})$ is uniformly equivalent to the law of $(\xi,\Gamma_{0})$, where $\xi$ is an independent sample from $\mu_{\iota}^{\#}$.
\par
Therefore, it is enough to work under the setup introduced at the beginning of this section. More precisely, we only need to show that for all $\iota\in(0,1/2)$ and $\delta\in (0,\iota/2)$, if we sample $\gamma$ according to $\mu_{\iota}^{\#}$ and take an independent sample $\Gamma_0$ of Brownian loop soup with critical intensity $1$ in $\mathbb D$, then there is no BDP produced by $\gamma$ inside $\Gamma_0$. Similar to the proof of Theorem \ref{ptp}, it further reduces to the proof of that the probability of existence of a $\delta$-good box is $O(n^{-c})$ for some constant $c>0$. With Proposition \ref{cgs4} as an input, we can show it immediately, and thus complete the proof of this case. The case of BDP visited by two different loops can be handled similarly, so we omit its proof. 
Finally, following exactly the same argument as in \cite[Section 7]{GLQ22}, which deals with the Hausdorff dimension of multiple points on the boundaries of Brownian loop-soup clusters, we can rule out the existence of boundary double points on all outermost clusters, and then all clusters (possibly not outermost). This concludes the proof of Theorem \ref{bdp}.
\end{proof}

%
%

\begin{acks}[Acknowledgments]
Part of YG's and WQ's work was done while working at City University of Hong Kong.
WQ is on leave from CNRS, Laboratoire de Math\'ematiques d'Orsay, Universit\'e Paris-Saclay.
The authors thank two anonymous referees for helpful and detailed comments.
\end{acks}
\begin{funding}
YG and WQ are supported by National Key R\&D Program of China (No.\ 2023YFA1010700) and a grant from City University of Hong Kong (Project No.\ 7200745). XL and RL are supported by National Key R\&D Program of China (No.\ 2021YFA1002700 and No.\ 2020YFA0712900) and NSFC (No.\ 12071012).
WQ is further supported by a GRF grant from the
Research Grants Council of the Hong Kong SAR (project 11305823).
\end{funding}

\end{document}